\documentclass[a4paper,12pt]{article}
\usepackage[utf8x]{inputenc}
\usepackage[T1]{fontenc}
\usepackage{amsmath}
\usepackage{amsfonts}
\usepackage{amssymb}
\usepackage{amsthm}
\usepackage{pstricks}
\usepackage{graphicx}
\usepackage{color}

\newtheorem{theorem}{Theorem}[section]
\newtheorem{lemma}[theorem]{Lemma}
\newtheorem{corollary}[theorem]{Corollary}
\newtheorem{proposition}[theorem]{Proposition}

\theoremstyle{definition}

\theoremstyle{remark}

\def \cA {\mathcal{A}}
\def \cB {\mathcal{B}}

\def \cE {\mathcal{E}}

\def \cM {\mathcal{M}}
\def \cN {\mathcal{N}}

\def \cP {\mathcal{P}}
\def \cQ {\mathcal{Q}}
\def \cR {\mathcal{R}}
\def \cS {\mathcal{S}}
\def \cT {\mathcal{T}}

\def \cW {\mathcal{W}}

\def \a {\alpha}
\def \b {\beta}
\def \g {\gamma}
\def \d {\delta}
\def \e {\varepsilon}

\def \t {\theta}

\def \k {\kappa}
\def \l {\lambda}
\def \s {\sigma}

\def \T {\Theta}
\def \L {\Lambda}

\def \N {\mathbb{N}}

\def \R {\mathbb{R}}

\def \Z {\mathbb{Z}}

\def \ra {\rightarrow}
\def \la {\leftarrow} 
\def \lra {\longrightarrow}

\def \Ra {\Rightarrow}
 
\def \Lra {\Longrightarrow}

\def \lma {\longmapsto}

\def \Al {{\mathcal A}^\ell}
\def \Alm {\big({\mathcal A}^\ell\big)^m}
\def \Ot {(O_t)_{t\geq 0}}
\def \Tt {(\T_t)_{t\geq 0}}
\def \Ott {(O^{\t}_t)_{t\geq 0}}
\def \Otl {(O^{\ell}_t)_{t\geq 0}}
\def \Otk {(O^{K+1}_t)_{t\geq 0}}
\def \Ztt {(Z^{\t}_t)_{t\geq 0}}
\def \Ztl {(Z^{\ell}_t)_{t\geq 0}}
\def \Ztk {(Z^{K+1}_t)_{t\geq 0}}

\def \pml {\cP^m_{\ell +1}}

\def \zl {\{\, 0,\dots,\ell \,\}}
\def \zk {\{\, 0,\dots,K \,\}}
\def \zm {\{\, 0,\dots,m \,\}}

\def \ote {o^\theta_{\text{exit}}}

\def \cW {{\mathcal W}^*}
\def \ozk {o(0)+\cdots +o(K)}
\def \exa {e^{-a}}

\def \xt {(X_t)_{t\geq 0}}

\def \ztk {(Z^\t_t(k))_{t\geq 0}}
\def \zt {(Z^\t_t)_{t\geq 0}}

\def \Dt {{(D_t)_{t\geq 0}}}
\def \tk {{\tau^*_K}}
\def \ul {\{\,1,\dots,\ell\,\}}
\def \zk {\{\,0,\dots,K\,\}}
\def \kl {\{\,K+1,\dots,\ell\,\}}
\def \ltq {\ell^{3/4}}

\def \cMH {{\mathcal M}_H}
\def \um {\{\,1,\dots,m\,\}}
\def \lk {\ell_\kappa}
\def \card {\text{card}\,}

\def \cMH {{\mathcal M}_H}

\def \mo {{\mu_O}}
\def \mol {{\mu_O^\ell}}
\def \mok {{\mu_O^{K+1}}}

\def \ztl {{(Z^L_t)_{t\geq 0}}}
\def \ztu {{(Z^U_t)_{t\geq 0}}}
\def \Zt {{(Z_t)_{t\geq 0}}}
\def \oe {{o_{\text{exit}}}}

\begin{document}
%
%
%

\begin{center}
\begin{LARGE}
The distribution of the quasispecies for a Moran model on the sharp peak landscape
\end{LARGE}

\begin{large}
Rapha\"el Cerf and Joseba Dalmau

\vspace{-12pt}
Universit\'e Paris Sud and ENS Paris
%
%
%

\vspace{4pt}
\today
\end{large}
\end{center}

\begin{abstract}
\noindent
We consider the Moran model on the sharp peak landscape,
in the asymptotic regime studied in~\cite{Cerf},
where a quasispecies is formed. 
We find explicitly the distribution of this quasispecies.
\end{abstract}

\section{Introduction}\label{Intro}
In his paper \cite{Eigen1}, 
Eigen introduced the model of quasispecies
to describe the evolution of a population of macromolecules
which is subject to two main forces:
mutation and selection.
The model was developed further in a series of papers by Eigen and Schuster
\cite{ES1,ES2,ES3},
and analysed in great detail by Eigen, McCaskill and Schuster in \cite{EMS}.
A major conclusion is that this kind of evolutionary process,
rather than selecting a single dominant species,
is more likely to select a master sequence 
(the macromolecule with the highest fitness)
along with a cloud of mutants that closely resemble the master sequence.
Hence the name quasispecies.
One other major discovery that Eigen made
on this model
was the existence of an error threshold 
allowing a quasispecies to form:
if the mutation rate exceeds the error threshold, 
then the population evolves towards a totally random state, 
whereas if the mutation rate is below the error threshold,
a quasispecies can be formed.

Even if Eigen's original goal 
was to explain the behaviour 
of a population of macromolecules,
the theory of quasispecies 
rapidly extended to other areas of biology.
In particular, experimental studies support the validity of the model
in virology \cite{Domingo}.
Some RNA viruses are known to have very high mutation rates,
like the HIV virus,
and this is a factor of resistance against conventional drugs.
A promising strategy to combat this kind of viruses consists in
developing mutagenic drugs 
that would increase the mutation rate 
beyond the error threshold,
in order to induce an error catastrophe
\cite{ADL,TBVD}.
This strategy has successfully been applied 
to several types of RNA viruses \cite{CCA}.
Moreover, several similarities have been observed 
between the evolution of cancer cell populations and RNA viruses,
in particular, the possibility of inducing an error catastrophe \cite{SD}.

Two important features of Eigen's model are its deterministic nature
(the model is based on a system of differential equations
derived from certain chemical and physical laws)
and the fact that the population is considered to be infinite.
When dealing with simple macromolecules,
these assumptions are quite natural.
Nevertheless, they become unrealistic 
if we want to apply this model to population genetics,
and they are two of the major drawbacks when applying it to virus populations,
as pointed out by Wilke \cite{Wilke}.
On one hand,
we have to take into account the stochastic nature of the evolution of a finite population.
The higher the complexity of the individuals,
the harder it is to explain the replication and mutation schemes via chemical reactions.
This fact, 
together with the widely recognised role of randomness in evolutionary processes 
strongly suggest a stochastic approach to the matter.
On the other hand, 
when dealing with populations of complex individuals,
the amount of possible genotypes 
largely exceeds the size of the population.
Therefore,
if we want to use Eigen's model in population genetics,
a finite and stochastic version of the model is called for.

The interest of a finite stochastic counterpart to Eigen's model
is not new. 
Eigen, McCaskill and Schuster 
already emphasise the importance of developing such a model \cite{EMS}, 
so does Wilke in the more recent paper \cite{Wilke}.
Several researchers have pursued this task. 
Demetrius, Schuster and Sigmund \cite{DSS} 
introduce stochasticity into Eigen's model using branching processes.
McCaskill \cite{McCaskill} also develops a stochastic version of Eigen's model. 
Nowak and Schuster \cite{NS} 
use birth and death Markov processes 
to give a finite stochastic version of Eigen's model 
on the sharp peak landscape. 
Alves and Fontanari \cite{AF} 
study the dependence of the error threshold 
on the population size 
for the single sharp peak replication landscape.
Saakian, Deem and Hu \cite{SAA1} compute the variance of the mean fitness in
a finite population model in order to control how it approximates the
infinite population model. Deem, Mu\~noz and Park \cite{PEM} use a field
theoretic representation in order to derive analytical results.
Other recent papers introduce finite stochastic models 
that approach Eigen's model asymptotically when the population size goes to $\infty$, 
like Musso \cite{Musso} 
or Dixit, Srivastava, Vishnoi \cite{DSV}.

In \cite{Cerf}, Cerf studies a population of size $m$
of chromosomes of length $\ell$
over an alphabet $\cA$ of cardinality $\k$ 
evolving according to a Moran model \cite{Moran}.
The mutation probability per locus is $q$.
Only the sharp peak landscape is considered:
the master sequence, which we denote by $w^*$,
replicates with rate $\s>1$,
while all the other sequences replicate with rate $1$.
In the asymptotic regime where
$$\displaylines{
\ell\to +\infty\,,\qquad m\to +\infty\,,\qquad q\to 0\,,\cr
{\ell q} \to a\,,
\qquad\frac{m}{\ell}\to\alpha\,,}$$
a critical curve is obtained in the parameter space
$(a,\a)$, which is given by ${\a\phi(a)=\ln\k}$.
If $\a\phi(a)<\ln\k$, then the population is totally random, i.e.,
the fraction of the master sequence in a population at equilibrium
converges to 0.
On the contrary, if $\a\phi(a)>\ln\k$, then a quasispecies is formed, i.e.,
at equilibrium, the population contains a positive fraction of the master sequence,
which in the asymptotic regime presented above converges to $(\s\exa-1)/(\s-1)$.

The aim of our article is to obtain the whole distribution of the quasispecies.
As it is customary with this kind of models,
we introduce Hamming classes with respect to the master sequence
in the space $\cA^\ell$ of sequences of length~$\ell$.
We say that a chromosome $u\in\cA^\ell$ 
belongs to the class $d\in\zl$ 
if it differs from the master sequence in exactly $d$ characters, i.e.,
$$\card\big\lbrace\,
i\in\zl:
w^*(i)\neq u(i)
\,\big\rbrace
\,=\,d\,.$$
We study then the concentration of each of these classes in a population at equilibrium.
For $k\geq 0$ fixed,
in the above asymptotic regime,
we recover the critical curve $\a\phi(a)=\ln\k$.
If $\a\phi(a)<\ln\k$, then
the fraction of the class $k$
converges to 0,
whereas if $\a\phi(a)>\ln\k$, then 
the fraction of the class $k$ in a population at equilibrium
converges to
$$\rho^*_k\,=\,(\s\exa-1)\frac{a^k}{k!}\sum_{i\geq 1}\frac{i^k}{\s^i}\,.$$
We denote by $\cQ(\s,a)$ the probability distribution 
which assigns mass $\rho^*_k$ to $k$, for $k\geq 0$,
and we call it the distribution of the quasispecies with parameters $\s,a$.

The article is organised as follows.
First, we present our main result,
along with a sketch of the proof,
a brief discussion about the distribution of the quasispecies
and some background material from \cite{Cerf}.
The remaining sections are devoted to the proof.

\begin{figure}
\centering
\includegraphics[trim=0.6cm 0cm 0cm 0cm, clip=true, scale=1]{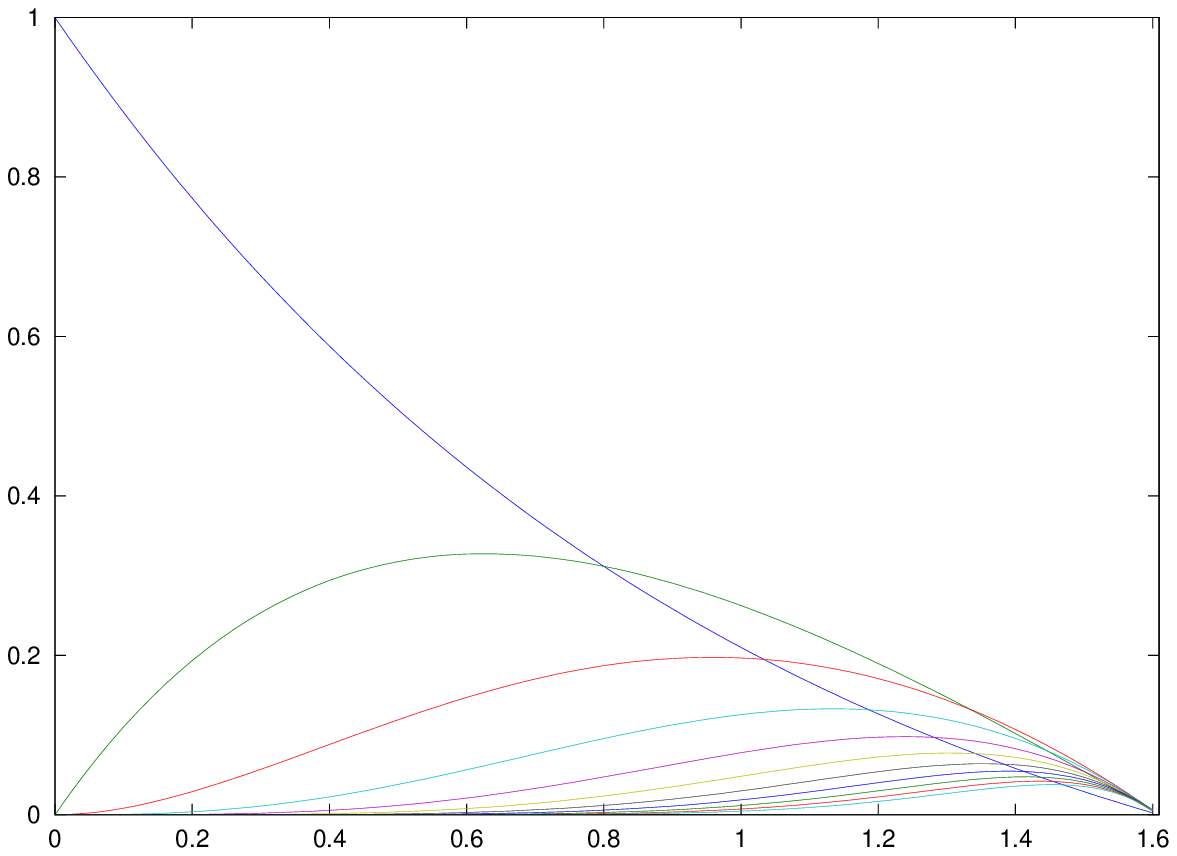}
%
Frequency of the master sequence
and the first 10 classes
for $\s=5$

\end{figure}

\begin{figure}
 \centering
 \includegraphics[trim=0.6cm 0cm 0cm 0cm, clip=true, scale=1]{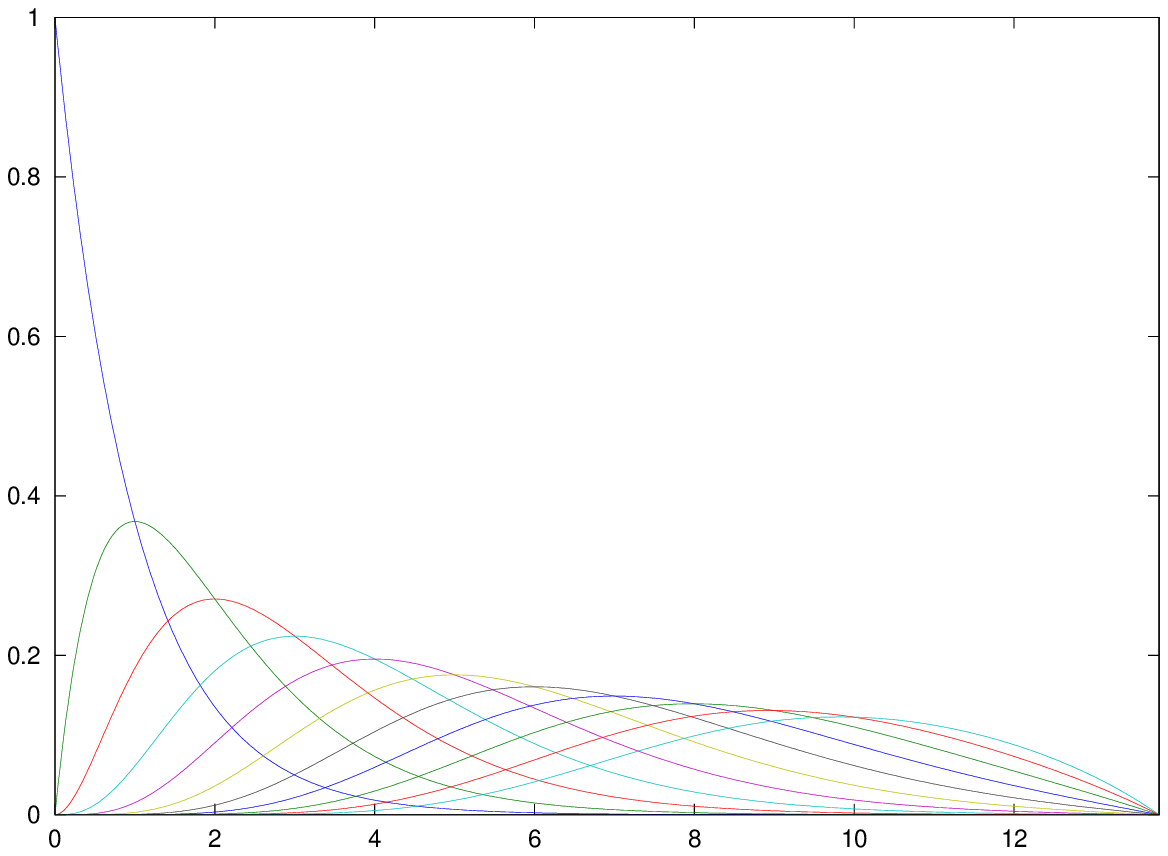}

Frequency of the master sequence
and the first 10 classes
for $\s=10^6$
\end{figure}

\section{Main Result}\label{Mainres}
Let $\cA$ be a finite alphabet of cardinality $\k$
and $\ell\geq 1$ an integer.
We consider the space $\Al$ 
of sequences of length $\ell$ 
over the alphabet $\cA$.
Elements of the space $\Al$
represent the chromosome of an haploid individual.
We consider a population of size $m$
of individuals from $\Al$.
The size of the population $m$
is kept constant throughout the evolution.

When a reproduction occurs,
the chromosome is subject to mutations.
We suppose that mutations occur independently at random at each locus,
with probability $q\in\,]0,1-1/\k[\,$.
If a mutation occurs,
we replace the letter with a new one,
chosen uniformly at random between the remaining $\k-1$ letters of the alphabet $\cA$.
The mutation mechanism is encoded in a mutation matrix $\big( M(u,v),\ u,v\in\Al\big)$,
where $M(u,v)$ is the probability that the chromosome $u$ is transformed into $v$ by mutation.
We have the following analytical expression for $M(u,v)$:
$$
M(u,v)\,=\,
\prod_{j=1}^\ell
\Big(
(1-q)1_{u(j)=v(j)}
+
\frac{q}{\k-1}
1_{u(j)\neq v(j)}
\Big)\,.
$$
The only allowed transformations in a population consist of replacing a chromosome of the population with a new one.
For a population $x\in\Alm$,
$j\in\um$, 
$u\in\Al$,
we denote by $x(j\la u)$
the populaton $x$ where the $j$--th chromosome $x(j)$ has been replaced by $u$:
$$
x(j\leftarrow u)
\,=\,
\left(
\begin{matrix}
x(1)\\
\vdots\\
x(j-1)\\
u\\
x(j+1)\\
\vdots\\
x(m)
\end{matrix}
\right)\,.
$$
The replication mechanism is encoded in a fitness function:
$$A:\Al\lra[0,+\infty[\,\,.$$
The discrete time Moran model is the 
Markov chain $\xt$,
whose transition matrix is given by
\begin{multline*}
\forall t\in \N\quad
\forall x\in\Alm\quad
\forall j\in\ul\quad
\forall u\in\Al\setminus\{x(j)\}\\
P\big(X_{t+1}=x(j\la u)
\,|\,
X_t=x \big)\,=\,
\frac{1}{m^2\l}
\sum_{1\leq i\leq m}
A(x(i))M(x(i),u)\,,
\end{multline*}
where $\l>0$ is a constant such that
$$\l\geq \max\,
\{\, A(u): u\in\Al \,\}\,.$$
The other non--diagonal coefficients of the matrix are null.
The diagonal coefficients are arranged so that the matrix is stochastic,
i.e.,
the sum over each row equals $1$.
We will only consider the sharp peak landscape.
Let $\s>1$ be a real number.
There exists a particular sequence,
called the master sequence or the wild type,
and denoted by $w^*$,
for which the replication rate is $\s$.
The replication rate for all other sequences is $1$.
The fitness function is then given by
$$
\forall u\in\Al\qquad
A(u)\,=\,
\begin{cases}
\quad 1 &\quad \text{if}\ u\neq w^*\,,\\
\quad \s &\quad \text{if}\ u=w^*\,.
\end{cases}
$$
We denote by $d_H$ the Hamming distance between two chromosomes:
$$
\forall u,v\in\Al\qquad
d_H(u,v)\,=\,
\text{card}\big\lbrace\,
i:1\leq i\leq m,
u(i)\neq v(i)
\,\big\rbrace\,.
$$
Let $x$ be a population in $\Alm$.
We fix an integer $K\geq 0$ 
and we look at the number $N^K(x)$ 
of chromosomes in $x$ 
which are at distance $K$ or less from the master sequence:
$$
N^K(x)\,=\,
\text{card}
\big\lbrace\,
i:1\leq i\leq m,
d_H(x(i),w^*)\leq K
\,\big\rbrace\,.
$$
Let $\phi:\R^+\lra\R^+\cup\{+\infty\}$
be the map given by
$$
\forall a<\ln\s\qquad
\phi(a)\,=\,
\frac
{ \displaystyle \sigma(1-e^{-a})
\ln\frac{\displaystyle\sigma(1-e^{-a})}{\displaystyle\sigma-1}
+\ln(\sigma e^{-a})}
{
\displaystyle (1-\sigma(1-e^{-a}))
}\,,
$$
and $\phi(a)=0$ for $a\geq\ln\s$.
Let $(\rho^*_k)_{k\geq 0}$
be the sequence given by
$$
\forall k\geq 0\qquad
\rho^*_k\,=\,
(\s\exa-1)
\frac{a^k}{k!}
\sum_{i\geq 1}
\frac{i^k}{\s^i}\,.
$$
We have then the following result:
\begin{theorem}\label{main}
Suppose that
$$\displaylines{
\ell\to +\infty\,,\qquad m\to +\infty\,,\qquad q\to 0\,,\cr
{\ell q} \to a\in \,]0,+\infty[\,,
\qquad\frac{m}{\ell}\to\alpha\in [0,+\infty]\,.}$$
We have the following dichotomy:
\medskip

\noindent
$\bullet\quad$ 
If $\alpha\,\phi(a)<\ln\kappa$, then
$$\forall K\geq 0\qquad
\lim_{
\genfrac{}{}{0pt}{1}{\ell,m\to\infty,\,
q\to 0
}
{{\ell q} \to a,\,
\frac{\scriptstyle m}{\scriptstyle \ell}\to\alpha
}
}
\,\lim_{t\ra\infty}\,
E\Big(
\frac{1}{m}N^K(X_t)
\Big)\,=\,0\,.
$$
\noindent
$\bullet\quad$ 
If $\alpha\,\phi(a)>\ln\kappa$, then
$$\forall K\geq 0\qquad
\lim_{
\genfrac{}{}{0pt}{1}{\ell,m\to\infty,\,
q\to 0
}
{{\ell q} \to a,\,
\frac{\scriptstyle m}{\scriptstyle \ell}\to\alpha
}
}
\,\lim_{t\ra\infty}\,
E\Big(
\frac{1}{m}N^K(X_t)
\Big)\,=\,\rho^*_0+\cdots+\rho^*_K\,.
$$
Furthermore, in both cases
$$\forall K\geq 0\qquad
\lim_{
\genfrac{}{}{0pt}{1}{\ell,m\to\infty,\,
q\to 0
}
{{\ell q} \to a,\,
\frac{\scriptstyle m}{\scriptstyle \ell}\to\alpha
}
}
\,\lim_{t\ra\infty}\,
\mathrm{Var}\Big(
\frac{1}{m}N^K(X_t)
\Big)\,=\,0\,.
$$
\end{theorem}

\subsection{Sketch of proof}
The state space of the Markov chain
$\xt$
has cardinality 
$\k^{\ell m}$,
which is too big to work with.
The first step in proving theorem \ref{main}
is to reduce the state space.
We use a classical technique called lumping
and we build a simpler process $\Ot$,
called the occupancy process, 
whose state space is much smaller.
The occupancy process 
$\Ot$
keeps track of the number of chromosomes in each of the 
$\ell+1$ Hamming classes.
This process will be the main subject of our study.
In addition,
the state space of the process
$\Ot$
is naturally endowed with a partial order 
which allows us to use coupling and monotonicity arguments.

We then compare the time that the process 
$\Ot$
spends 
having at least a sequence in one of the Hamming classes
$0,\dots,K$ (which we call the persistence time),
with the time the process 
$\Ot$ 
spends having no sequences in any of the Hamming classes
$0,\dots,K$ (which we call the discovery time).
Asymptotically, 
when ${\a\phi(a)>\ln\k}$,
the persistence time becomes negligible with respect to the discovery time,
whereas when ${\a\phi(a)<\ln\k}$,
it is
the discovery time that becomes negligible with respect to the persistence time.
This already proves the first assertion in theorem \ref{main}.

The second statement of the theorem requires much more work.
We build first a coupling 
to compare the occupancy process
$\Ot$
with simpler processes:
a lower process $\Otl$,
and an upper process $\Otk$.
These processes stochastically bound the occupancy process 
$\Ot$, 
and they only keep track of the number of chromosomes
in the Hamming classes $0,\dots,K$.
The goal is to show that the 
invariant probability measures of the processes
$\Otl$ and $\Otk$
both converge to the Dirac mass at the point $(\rho^*_0,\dots,\rho^*_K)$.
This is achieved by estimating the typical time that the processes
spend inside and outside a neighbourhood of $(\rho^*_0,\dots,\rho^*_K)$.
The time they spend inside a neighbourhood of $(\rho^*_0,\dots,\rho^*_K)$
is typically of exponential order in $m$,
whereas the time they spend outside such a neighbourhood 
is typically of polynomial order in $m$.
These estimates are enough to complete the proof of theorem \ref{main}.
The strategy to obtain the estimates is as follows.
The mutation probabilities $M(u,v)$ for $u>v$
go to $0$ when $\ell\ra\infty$, $q\ra0$ and $\ell q\ra a$.
Thanks to this fact we can build the processes
$\Otl$ and $\Otk$
in such a way that, for $0\leq k\leq K$,
the evolution of the Hamming classes $0,\dots,k$
does not depend on the remaining Hamming classes.
We can then proceed to prove the estimates by induction.
Both the initial case and the inductive step 
boil down to the study of birth and death Markov chains,
for which explicit formulas are available.

\subsection{The distribution of the quasispecies}
Let $a$ be such that $\s\exa>1$.
As stated in theorem \ref{main},
the distribution of the quasispecies of parameters $\s,a$,
is given by the sequence 
$(\rho^*_k)_{k\geq 0}$:
$$
\forall k\geq 0\qquad
\rho^*_k\,=\,
(\s\exa-1)
\frac{a^k}{k!}
\sum_{i\geq 1}
\frac{i^k}{\s^i}\,.
$$
Nevertheless,
we will carry out the proof by induction,
and in our proof we will not find the sequence 
$(\rho^*_k)_{k\geq 0}$
in the above form,
it will be given by the following recurrence relation:
\begin{align*}
\rho^*_0\,&=\,
\frac{\s\exa-1}{\s-1}\,,\\
\rho^*_k\,&=\,
\frac{\exa}{\displaystyle (\s-1)\rho^*_0+1-\exa}
\Big(
\s\frac{a^k}{k!}\rho^*_0
+\sum_{l=1}^{k-1}\frac{a^{k-l}}{(k-l)!}\rho^*_l
\Big)\,,\qquad k\geq 1\,.
\end{align*}
We show next how to derive the explicit formula for 
$(\rho^*_k)_{k\geq 0}$
from the recurrence relation.
Firstly,
we remark that replacing $\rho^*_0$ by its value,
$$\frac{\exa}{(\s-1)\rho^*_0+1-\exa}\,=\,
\frac{1}{\s-1}\,.$$
Therefore,
we can rewrite the recurrence relation as follows:
\begin{align*}
\rho^*_0\,&=\,
\frac{\s e^{-a}-1}{\s-1}\\
\rho^*_k\,&=\,
\frac{a^k}{k!}\rho^*_0+
\frac{1}{\s}\sum_{l=1}^k\frac{a^{k-l}}{(k-l)!}\rho^*_l
\,,\qquad 
k\geq 1\,.
\end{align*}
Let $f$ be the generating function of the sequence $(\rho^*_k)_{k\geq 0}$:
$$f(X)\,=\,
\sum_{k\geq 0}\rho^*_k X^k\,.
$$
Let us consider the following formal series:
$$
\frac{1}{\s}e^{aX}\,=\,
\sum_{k\geq 0}\frac{1}{\s}\frac{a^k}{k!}X^k\,.$$
We do the Cauchy product of the two formal series above and we obtain:
$$f(X)\Big(
\frac{1}{\s}e^{aX}
\Big)\,=\,
\sum_{k\geq 0}\biggr(
\frac{1}{\s}
\sum_{l=0}^k \frac{a^{k-l}}{(k-l)!}\rho^*_l
\biggr)X^k\,.$$
Yet, thanks to the recurrence relation,
$$
\frac{1}{\s}
\sum_{l=0}^k 
\frac{a^{k-l}}{(k-l)!}\rho^*_l\,=\,
\rho^*_k+\Big(
\frac{1}{\s}-1
\Big)\frac{a^k}{k!}\rho^*_0\,.
$$
Thus,
$$f(X)\Big(
\frac{1}{\s}e^{aX}
\Big)\,=\,
f(X)+\Big(
\frac{1}{\s}-1
\Big)\rho^*_0 e^{aX}\,.$$
Replacing $\rho^*_0$ with its value gives the following expression for $f$:
$$f(X)\,=\,
(\s e^{-a}-1)\frac{e^{aX}}{\s-e^{aX}}\,.$$
We remark that $f(1)=1$, thus
$(\rho^*_k)_{k\geq 0}$ is indeed a probability distribution on $\N$.
We develop this last expression as follows:
$$
\frac{e^{aX}}{\s-e^{aX}}\,=\,
\sum_{i\geq 1}\bigg(
\frac{e^{aX}}{\s}
\bigg)^i\,=\,
\sum_{i\geq 1}
\frac{1}{\s^i}
\sum_{k\geq 0}
\frac{(aiX)^k}{k!}\,=\,
\sum_{k\geq 0}\bigg(
\sum_{i\geq 1}
\frac{i^k}{\s^i}
\bigg)\frac{a^k}{k!}X^k\,.
$$
We obtain finally 
$$\forall k\geq 0\qquad
\rho^*_k\,=\,
(\s e^{-a}-1)
\frac{a^k}{k!}
\sum_{i\geq 1}
\frac{i^k}{\s^i}\,.
$$
We call this the probability distribution of the quasispecies
with parameters
$\s$, $a$
and we denote it by 
$\cQ(\s,a)$.
A short calculation shows that the expectation and the variance of 
$\cQ(\s,a)$ are given by:
$$
\text{E}(\cQ)\,=\,
\frac{\s a e^{-a}}{\s e^{-a}-1}\,,\qquad
\text{Var}(\cQ)\,=\,
\frac{\s a e^{-a}(\s e^{-a}+a-1)}{(\s e^{-a}-1)^2}\,.
$$
The graphs at the end of the introduction show the frequency of the master sequence 
and the first 10 Hamming classes for $\s=5$ and $\s=10^6$.
The graphs closely resemble those obtained by solving
the differential equations from Eigen's original model \cite{EMS,Schuster}.

\subsection{The occupancy process}
The occupancy process
$\Ot$
will be the main subject of our study,
it is obtained from the original process
$\xt$
via lumping, as in section 6.3 of \cite{Cerf}.
Let $\pml$ be the set of the ordered partitions 
of the integer $m$ in at most $\ell+1$ parts:
$$
\pml\,=\,
\big\lbrace\,
(o(0),\dots,o(\ell))\in\N^{\ell+1}:
o(0)+\cdots+o(\ell)=m
\,\big\rbrace\,.
$$
A partition $(o(0),\dots,o(\ell))$
is interpreted as an occupancy distribution,
which corresponds to a population with $o(l)$ 
individuals in the Hamming class $l$, for $0\leq l\leq \ell$.
Since we are working with a Moran model,
only a chromosome can change classes at a time,
i.e.,
the only possible transitions for the occupancy process
$\Ot$ are of the form
$$
o\ \lra\ o(k\ra l)\,,\qquad 
0\leq k,l\leq \ell\,,
$$
where $\displaystyle o(k\ra l)$
is the occupancy distribution obtained by transferring a chromosome
from the Hamming class $k$ to the class $l$, i.e.,
$$
\forall h\in\zl\qquad
o(k\ra l)(h)\,=\,
\begin{cases}
\quad o(h)\quad &\quad \text{if } h\neq k,l\,,\\
\quad o(k)-1\quad &\quad \text{if } h=k\,,\\
\quad o(l)+1\quad &\quad \text{if } h=l\,.
\end{cases}
$$
We will work with a discrete time occupancy process $\Ot$,\
whose transition matrix is given by
\begin{multline*}
\forall o\in\pml\quad\forall k,l\in\zl\,,
\quad k\neq l\,,\cr
p_{O}\big(o,o(k\rightarrow l)\big)\,=\,
\frac{\displaystyle
o(k)
\sum_{h=0}^\ell
o(h)\,
{A_H(h)}\, M_H(h,l)}
{\displaystyle
m\sum_{h=0}^\ell
o(h)\,
{A_H(h)}}\,,
\end{multline*}
where $A_H$ is the lumped fitness function, 
defined as follows
$$\forall b \in \zl\qquad
A_H(b)\,=\,
\begin{cases}
\quad \s\quad & \text{if } b=0\,,\\
\quad 1\quad & \text{if } b\geq 1\,,
\end{cases}
$$
and $M_H$ is the lumped mutation matrix:
for $b,c\in\zl$ the coefficient $M_H(b,c)$ is given by
$$
\sum_{
\genfrac{}{}{0pt}{1}{0\leq k\leq\ell-b}{
\genfrac{}{}{0pt}{1}
 {0\leq l\leq b}{k-l=c-b}
}
}
{ \binom{\ell-b}{k}}
{\binom{b}{l}}
\Big(p\Big(1-\frac{1}{\kappa}\Big)\Big)^k
\Big(1-p\Big(1-\frac{1}{\kappa}\Big)\Big)^{\ell-b-k}
\Big(\frac{p}{\kappa}\Big)^l
\Big(1-\frac{p}{\kappa}\Big)^{b-l}\,.
$$

\section{Stochastic bounds}\label{Sbounds}
In this section we will build a lower process $\Otl$
and an upper process $\Otk$ in order to bound stochastically 
the occupancy process $\Ot$.
The space $\pml$ of the occupancy distributions
is endowed with a natural order $\preceq$.
If $o,o'$ are two occupancy distributions 
we write that $o\preceq o'$ if and only if
$$\forall l\in\zl\qquad 
o(0)+\cdots+o(l)\,\leq\,
o'(0)+\cdots+o'(l)\,.$$
We will construct the lower process $\Otl$
and the upper process $\Otk$
in such a way that for any $o\in \pml$,
if $O^\ell_0=O_0=O^{K+1}_0=o$, then
$$\forall t\geq 0\qquad
O^\ell_t\,\preceq\,
O_t\,\preceq\,
O^{K+1}_t\,.$$
The processes 
$\Otl$ and $\Otk$
will be much simpler than the occupancy process $\Ot$.

\subsection{The lower process}\label{Lower}
We start by building the lower process $\Otl$.
First of all, let us explain loosely the dynamics of the lower process $\Otl$.
As long as there is no master sequence present in the population,
the lower process $\Otl$ evolves exactly as the original process $\Ot$.
When a master sequence appears,
all the chromosomes in the Hamming classes $K+1,\dots,\ell$
are sent to the class $\ell$.
As long as the master sequence is present in the population,
a mutation to any of the classes 
$K+1,\dots,\ell$
is directly sent to the class $\ell$.
Furthermore, 
every mutation from a Hamming class to a lower class is also sent to the class $\ell$.
To make this construction rigorous,
we will modify the coupling map $\Phi_O$
defined in section 7.1 of \cite{Cerf}.
To do so, we will also use the maps
$\cM_H$ and $\cS_O$ defined in the section 7.1 of \cite{Cerf}.
We take $\cR$ to be the set
$$\cR\,=\,
[0,1]\times\zm^2\times [0,1]^\ell\,,
$$
and we define a map
$$\underline{\Phi}_O:
\pml\times\cR\lra\pml\,,$$
as follows.
Let $r=(s,i,j,u_1,\dots,u_\ell)\in\cR$ 
and $o\in\pml$.
We take $l=\cS_O(o,s)$ 
and $k$ the only index in $\{\,0,\dots,\ell\,\}$ such that
$$o(0)+\cdots+o(k-1)\,<\,
j\,\leq\,
o(0)+\cdots+o(k).$$
We define the map 
$\underline{\Phi}_O$
by:
\begin{align*}
\underline{\Phi}_O(o,r)\,&=\,
\begin{cases}
\quad o(k\ra \ell) & \text{ if\ } \cM_H(l,u_1,\dots,u_\ell)<l\,,\\
\quad o\big(k\ra \cM_H(l,u_1,\dots,u_\ell) \big) & \text{ otherwise}.
\end{cases}
\end{align*}
From this construction we see that
$$\forall r\in\cR\quad 
\forall o\in\pml\qquad
\underline{\Phi}_O(o,r)\preceq\Phi_O(o,r)\,.
$$
We define a map
$\pi_\ell:\pml\to\pml$
by setting, for
$o\in\pml$ and $l\in\zl$,
\begin{align*}
\pi_\ell(o)(l)\,&=\,
\begin{cases}
\quad o(l) & \quad\text{if } 0\leq l\leq K\,, \\
\quad 0 & \quad\text{if } K<l<\ell\,, \\
\quad m-(o(0)+\cdots +o(K)) & \quad\text{if }\ l=\ell\,.
\end{cases}
\end{align*}
This map satisfies
$$\forall o\in\pml\qquad
\pi_\ell(o)
\, \preceq\, 
o\,.$$
We denote by $\cW$ the set of the occupancy distributions
having at least one master sequence, i.e.,
$$\cW\,=\,\big\{\,o\in\pml:o(0)\geq 1\,\big\}\,,$$
and we denote by $\cN$ the set of the occupancy distributions
having no master sequence, i.e.,
$$\cN\,=\,\big\{\,o\in\pml:o(0)=0\,\big\}\,.$$
Let us define
\begin{align*}
o^\ell_{\text{exit}}\,&=\,
(0,\dots,0,m)\,,
\quad 
&&o^\ell_{\text{enter}}\,=\,
(1,0,\dots,0,m-1)\,.
\end{align*}
We define a lower map $\Phi_O^{\ell}$ by setting for $o\in\pml$ and $r\in\cR$,
\begin{equation*}
\index{$\Phi_O^{\ell}$}
\Phi_O^{\ell}(o,r)\,=\,
\begin{cases}
\quad 
\Phi_{O}(o,r)
& \quad\text{if }o\in\cN \,\,{\text{ and }}\,\, 
\Phi_O(o,r)\not\in\cW\,, \\
\quad 
o^\ell_{\text{enter}}
& \quad\text{if }o\in\cN \,\,{\text{ and }}\,\, 
\Phi_O(o,r)\in\cW\,, \\
\quad 
\pi_\ell\big(\underline{\Phi}_{O}(\pi_\ell(o),r)\big)
& \quad\text{if }o\in\cW\ \text{ and }\ 
\underline{\Phi}_{O}(\pi_\ell(o),r)\not\in\cN\,,\\
\quad 
o^\ell_{\text{exit}}
&\quad\text{if }o\in\cW\ \text{ and }\ 
\underline{\Phi}_{O}(\pi_\ell(o),r)\in\cN\,.\\
\end{cases}
\end{equation*}
The next proposition compares the map
$\Phi_O^{\ell}$ to the map
$\Phi_{O}$.
\begin{proposition}\label{compphiol} 
For all
$r\in\cR$
and for all
$o\in\pml$,
$$\Phi_O^{\ell}(o,r)
\,\preceq\,
\Phi_{O}(o,r)\,.$$
\end{proposition}
\begin{proof}
Let us take $r\in\cR$
and
$o\in\pml$. We consider the four following cases:

$\bullet$ If $o\in\cN$ and $\Phi_O(o,r)\not\in\cW$, then
$$\Phi^\ell_O(o,r)\,=\,
\Phi_O(o,r)\,.$$
$\bullet$ If $o\in\cN$ and $\Phi_O(o,r)\in\cW$,
we have
$$\Phi^\ell_O(o,r)\,=\,o^\ell_{\text{enter}}
\qquad \text{and}\qquad 
\Phi_O(o,r)(0)\,=\,1\,,$$
the inequality holds since for all
$o\in\pml$ with $o(0)=1$, we have
$o^\ell_{\text{enter}}\,\preceq\,o\,.$

$\bullet$ If $o\in\cW$ and $\underline{\Phi}_O(\pi_\ell(o),r)\not\in\cN$,
since the mapping
$\underline{\Phi}_O$
is lower than 
$\Phi_O$, we have
$$
\underline{\Phi}_O(\pi_\ell(o),r)
\,\preceq\,
\Phi_O(\pi_\ell(o),r)\,.
$$
Also
$\pi_\ell(o)\preceq o$,
and $\phi_O$ is monotone,
so that
$$\pi_\ell\big( \underline{\Phi}_O(\pi_\ell(o),r) \big)\,\preceq\,
\underline{\Phi}_O(\pi_\ell(o),r)\,\preceq\,
\Phi_O(\pi_\ell(o),r)\,\preceq\,
\Phi_O(o,r)\,.$$
$\bullet$ If $o\in\cW$ and $\underline{\Phi}_O(\pi_\ell(o),r)\in\cN$,
we have
$\Phi^\ell_O(o)=o^\ell_{\text{exit}}$,
we remark then that for all  $o\in\pml$,
$o^\ell_{\text{exit}}\,\preceq\, o\,.$

We finally conclude that
$\Phi^\ell_O(o,r)\,\preceq\,
\Phi_O(o,r)\,,$
for all $o\in\pml$ and for all $r\in\cR$.
\end{proof}
We define next the lower process
$(O^\ell_t)_{t\geq 0}$.
Let
$$R_n=
(I_n,J_n,S_n,U_{n,1},\dots,U_{n,\ell})\,,\quad
n\geq 1\,,$$
be an i.i.d. sequence of random vectors with values in $\cR$, 
as defined in \cite{Cerf}. 
The components of $R_n$ 
are independent random variables
with uniform distribution
on their corresponding spaces.
Let $o\in\pml$
be the starting point of the process.
We set
$O^\ell_0=o$ and
$$\forall n\geq 1\qquad
O^\ell_n\,=\,\Phi_O^{\ell}\big(O^\ell_{n-1}, R_n\big)\,.$$
\begin{proposition}\label{compol}
We suppose that the processes
$(O^\ell_t)_{t\geq 0}$ and
$(O_t)_{t\geq 0}$
have the same starting occupancy distribution $o$.
We have then
$$\forall t\geq 0\qquad
O^\ell_t
\,\preceq\,
O_t\,.$$
\end{proposition}
The proof is similar to the proof of proposition 8.1 of \cite{Cerf}.

\subsection{Dynamics of the lower process}\label{Dynalow}
We study now the dynamics of the lower process 
$(O^\ell_t)_{t\geq 0}$
in $\cW$.
For the process
$(O^\ell_t)_{t\geq 0}$, 
the states in the set
$$\cT^\ell\,=\,\big\{\,o\in\pml:
o(0)\geq 1\text{ and }\ozk+o(\ell)<m
\,\big\}\,
\index{$\cT^\ell$}
$$
are transient,
and the states in the set
$\smash{\cN\cup\big(\cW\setminus\cT^\ell\big)}$ 
form a recurrent class.
We will therefore focus on the dynamics of the process $\Otl$
restricted to
$\smash{\cW\setminus\cT^\ell}$.
Since
$$\smash{\cW\setminus\cT^\ell}
\,=\,\big\{\,
o\in\pml:
o(0)\geq 1\text{ and }\ozk+o(\ell)=m
\,\big\}\,,$$
a state in
$\smash{\cW\setminus\cT^\ell}$
is completely determined by the occupation numbers of the classes
$0,\dots,K$.
The process 
$(O^\ell_t)_{t\geq 0}$ always enters the set
$\cW\setminus\cT^\ell$ at
$o^\ell_{\text{enter}}$. 
For $i\in\zk$, we denote by $w_i$ the vector of
$\N^{K+1}$ given by:
$$
\forall l\in\zk\qquad 
w_i(l)\,=\, \begin{cases}
1 &\text{if $l=i$}\,, \\
0&\text{otherwise}\,.
\end{cases}
$$
The only possible transitions for the Hamming classes
$0,\dots,K$ of the process
$(O^\ell_t)_{t\geq 0}$
starting from a state in
$\smash{\cW\setminus\cT^\ell}$ are
\begin{multline*}
(o(0),\dots,o(K))
\quad \longrightarrow\quad
(o(0),\dots,o(K))-w_i \\
\text{if }\ 
1\leq o(i),\quad 0\leq i\leq K\,,
\end{multline*}
\begin{multline*}
(o(0),\dots,o(K))
\quad \longrightarrow\quad
(o(0),\dots,o(K))+w_i \\
\text{if }\ 
\ozk\leq m-1,\quad 0\leq i\leq K\,,
\end{multline*}
\begin{multline*}
(o(0),\dots,o(K))
\quad \longrightarrow\quad
(o(0),\dots,o(K))-w_i+w_j \\
\text{if }\ 
1\leq o(i),\quad 0\leq i,j\leq K,\quad i\neq j\,.
\end{multline*}
The process
$(O^\ell_t)_{t\geq 0}$ 
always exits the set
$\cW\setminus\cT^\ell$ at
$o^\ell_{\text{exit}}$.
If the process $\Otl$ starts from a state in $\smash{\cW\setminus\cT^\ell}$,
until the time of exit from $\smash{\cW\setminus\cT^\ell}$,
the dynamics of $\big( O^\ell_t(0),\dots,O^\ell_t(K) \big)$
is that of a Markov chain on the state space
$$\cE_K\,=\,
\big\lbrace\, z\in\N^{K+1} : z_0+\cdots+z_K \leq m  \,\big\rbrace\,.$$
Let us compute the associated transition probabilities.
Let $z\in\cE_K$.

$\bullet $ For
$0\leq i\leq K$ and
$0<z_0+\cdots+z_K<m\,,$
$$
p(z,z+w_i)\,=\,
\frac{m-\displaystyle \sum_{l=0}^K z_l}{m((\s-1)z_0+m)}
\times\Biggr( \s z_0 M_H(0,i)+
\sum_{l=1}^{i} z_l M_H(l,i)
\Biggr)\,.
$$
\par\noindent
$\bullet $ For
$0\leq i\leq K$ and
$1\leq z_i\,,$
\begin{multline*}
p(z,z-w_i)\,=\,
\frac{z_i}{m((\s-1)z_0+m)}\times\\
\Biggr( \s z_0 \biggr(1-\sum_{h=0}^K M_H(0,h)\biggr)
+\sum_{l=1}^K z_l \biggr(1-\sum_{h=l}^K M_H(l,h)\biggr)+
m-\sum_{l=0}^K z_l \Biggr)\,.
\end{multline*}
\par\noindent
$\bullet $ For
$0\leq i,j\leq K$, $i\neq j$ and
$1\leq z_i\,,$
\begin{multline*}
p(z,z-w_i+w_j)\,=\,
\frac{z_i}{m((\s-1)z_0+m)}
\times\Biggr( \s z_0 M_H(0,j)+
\sum_{l=1}^{j} z_l M_H(l,j)
\Biggr).
\end{multline*}
The other non--diagonal coefficients of the matrix are null.
The diagonal coefficients are arranged so that the matrix is stochastic, i.e.,
the sum over each row equals 1.

Since we are interested in the dynamics of $\Otl$
in $\smash{\cW\setminus\cT^\ell}$,
the transition probabilities starting from a point in
$\lbrace\, z\in\cE_K :
z_0=0 \,\rbrace$
are not relevant.
Moreover, the law of the exit point from
$\lbrace\, z\in\cE_K :
z_0\geq 1 \,\rbrace$
is also not relevant,
what matters is the law of the exit time.
Therefore we will modify the matrix $p$
into another stochastic matrix $p^\ell$ 
such that:

$\bullet$ Starting from 
$\lbrace\, z\in\cE_K :
z_0=0 \,\rbrace$,
there is a jump with probability $1$
to $(1,0,\dots,0)$.

$\bullet$ The law of the exit time from
$\lbrace\, z\in\cE_K :
z_0\geq 1 \,\rbrace$
is unchanged,
but the exit point is
$(0,\dots,0)$
with probability 1.

More precisely, 
we define the matrix $p^\ell$
as follows:

$\bullet$ For $z,z'\in\cE_K$ with $z_0=0$,
\begin{align*}
p^\ell\big(z,(1,0,\dots,0)\big)\,&=\,1\,,\\
p^\ell(z,z')\,&=\,0\qquad
\text{if }\ z'\neq(1,0,\dots,0)\,.
\end{align*}
$\bullet $ For $z,z'\in\cE_K$ with
$z_0=1$ and $z'_0=0$,
\begin{align*}
p^\ell\big(z,(0,\dots,0)\big)\,&=\,
p(z,z-w_0)+\sum_{i=1}^K p(z,z-w_0+w_i)\,,\\
p^\ell(z,z')\,&=\,0\qquad \text{if }\ z'_0=0\ \text{ and }\ z'\neq(0,\dots,0)\,.
\end{align*}
Finally,
$p^\ell(z,z')=p(z,z')$
for all remaining
$z,z'\in\cE_K$.

Let $\Ztl$
be a Markov chain with state space $\cE_K$,
starting at the point $z^\ell=(1,0,\dots,0)$
and having for transition matrix $p^\ell$.
Since $\Otl$ always enters $\smash{\cW\setminus\cT^\ell}$ 
at $o^\ell_\text{enter}$
and its dynamics inside $\smash{\cW\setminus\cT^\ell}$
is the same as the dynamics of the chain $\Ztl$,
we will rely on this Markov chain to compute the desired estimates.


\subsection{The upper process}\label{Upper}
We build now the upper process $\Otk$.
First of all, let us explain loosely the dynamics of the upper process $\Otk$.
As long as there is no master sequence present in the population,
the upper process $\Otk$ evolves exactly as the original process $\Ot$.
When a master sequence appears,
all the chromosomes in the Hamming classes $K+1,\dots,\ell$
are sent to the class $K+1$.
As long as the master sequence is present in the population,
a mutation to any of the classes 
$K+1,\dots,\ell$
is directly sent to the class $K+1$.
Furthermore, 
for all $c<b$, the mutation probability from the Hamming class $b$
to the Hamming class $c$ is taken to be equal to
$M_H(c+1,c)$.
To make this construction rigorous,
we modify the mutation probabilities
${\big(
M_H(b,c),\ 0\leq b,c\leq \ell
\big)}$
and we define new mutation probabilities
${\big(
M^K_H(b,c),\ 0\leq b,c\leq \ell
\big)}$
for the process $\Otk$.
Let us set for $b\in\{0,\dots,K+1\}$ and
$c\in\{0,\dots,\ell\}$,
$$M_H^{K+1}(b,c)\,=\,
\begin{cases}
\quad M_H(c+1,c) & \quad\text{ if }\ 0\leq c<b\leq K+1\,,\\
\quad M_H(b,c) & \quad\text{ if }\ b\leq c\leq K\,,\\
\quad 0 & \quad\text{ if } c\in\{\,K+2,\dots,\ell\,\}\,.
\end{cases}$$
The coefficient
$M_H^{K+1}(b,K+1)$
is adjusted so that each row adds up to 1, i.e., we take, for
$b\in\{\,0,\dots,K+1\,\}$,
\begin{multline*}
M_H^{K+1}(b,K+1)\,=\,
1-\sum_{h=0}^K M_H^{K+1}(b,h)\,=\\
1
-\sum_{h=0}^{b-1} M_H(h+1,h)
-\sum_{h=b}^K M_H(b,h)\,.
\end{multline*}
Moreover, for $b\in\{K+2,\dots,\ell\}$,
we set
$$\forall c\in\{0,\dots,\ell\}\qquad
M_H^{K+1}(b,c)\,=\,
M_H(b,c)\,.$$
We must verify that 
$(M_H^{K+1}(b,c),0\leq b,c\leq\ell)$
is a stochastic matrix, i.e.,
that all entries of the matrix are non--negative
and that each row adds up to 1.
Since  $(M_H(b,c),0\leq b,c\leq\ell)$
is already a stochastic matrix,
these conditions are satisfied for the rows
$M_H^{K+1}(b,\cdotp)$,
$b\in\{\,K+2,\dots,\ell\,\}$.
For the first $K+2$ rows,
the only thing left to verify is that the coefficient
$M_H^{K+1}(b,K+1)$
is non--negative, in other words, that
$$\forall b\in\{\,0,\dots,K+1\,\}\qquad
\sum_{h=0}^{b-1} M_H(h+1,h)+
\sum_{h=b}^K M_H(b,h)
\,\leq\, 1\,.$$
However, we are interested in the asymptotic regime
$$\displaylines{
\ell\to +\infty\,,\qquad m\to +\infty\,,\qquad q\to 0\,,\cr
{\ell q} \to a\,,
\qquad\frac{m}{\ell}\to\alpha\,.}$$
Thus, it is enough to verify the preceding inequalities for
$\ell,m$ big enough and $q$ small enough.
The mutation probabilities have the following limits:
$$\forall b,c\geq 0\qquad
\lim_{
\genfrac{}{}{0pt}{1}{\ell,m\to\infty}
{q\to 0,\,
{\ell q} \to a}
} M_H(b,c)\,=\,
\begin{cases}
\quad 0 & \quad\text{ if }\ 0\leq c<b\,,\\
\quad\displaystyle \frac{a^{c-b}}{(c-b)!}e^{-a} & \quad\text{ if }\ 0\leq b\leq c\,. 
\end{cases}$$
We deduce that, for $b\in\lbrace\,0,\dots,K+1\,\rbrace$,
\begin{multline*}
\lim_{
\genfrac{}{}{0pt}{1}{\ell,m\to\infty}
{q\to 0,\,
{\ell q} \to a}
}\sum_{h=0}^{b-1} M_H(h+1,h)+
\sum_{h=b}^K M_H(b,h)\,=\\
\sum_{h=b}^K\frac{a^{h-b}}{(h-b)!}\exa\,=\,
\sum_{k=0}^{K-b}\frac{a^k}{k!}\exa\,<\,1\,.
\end{multline*}
Therefore, for
$\ell,m$ big enough and $q$ small enough,
the modified mutation matrix ${(M_H^{K+1}(b,c), 0\leq b,c\leq\ell)}$
is indeed stochastic.
We build now two maps
$$\cM'_H,\ \cM_H^{K+1}:
\zl\times[0,1]\lra\zl$$
in order to couple the mutation mechanisms of the processes
$\Ot$ and $\Otk$.
Naturally, this coupling will allow us to compare these processes.
Let ${b\in\zl}$ and $u\in [0,1]$.
We define $\cM'_H(b,u)$
to be the only index
$c\in\zl$ such that
$$M_H(b,0)+\cdots+M_H(b,c-1)\,<\,
u\,\leq\,
M_H(b,0)+\cdots+M_H(b,c)\,.
$$
Likewise, we define $\cM^{K+1}_H(b,u)$
to be the only index
$c\in\zl$ such that
$$M^{K+1}_H(b,0)+\cdots+M^{K+1}_H(b,c-1)\,<\,
u\,\leq\,
M_H^{K+1}(b,0)+\cdots+M_H^{K+1}(b,c)\,.
$$
\begin{lemma}\label{domimut}
The map
$\cM'_H$
is above the map
$\cM^{K+1}_H$ in the following sense:
$$
\forall b\in\zl\quad \forall u\in [0,1]\qquad
\cM'_H(b,u)\,\geq\,
\cM_H^{K+1}(b,u)\,.$$
\end{lemma}
\begin{proof}
Since 
$M_H(b,c)\leq M_H(c+1,c)$ for $b>c$,
it follows from the definition of
the matrix
$(M_H^{K+1}(b,c),0\leq b,c\leq \ell)$ 
that
$$\forall b,c\in\zl\qquad
\sum_{h=0}^c M_H(b,h)\,\leq\,
\sum_{h=0}^c M_H^{K+1}(b,h)\,.$$
These inequalities imply the desired result.
\end{proof}
We have also the following result:
\begin{lemma}\label{monomut}
The map $\cM'_H$ is non--decreasing with respect to its first argument, i.e.,
$$\forall b,c\in\zl\quad
\forall u\in [0,1]\qquad
b\leq c\,\Ra\,
\cM'_H(b,u)\leq \cM'_H(c,u)\,.$$
\end{lemma}
\begin{proof}
We consider the map
$$\cM_H:\zl\times [0,1]^\ell\to\zl$$
defined in section 7.1 of \cite{Cerf} by
\begin{multline*}
\forall b\in\zl\quad
\forall u_1,\dots,u_\ell\in[0,1]^\ell\\
\cM_H(b,
u_1,\dots,u_\ell)\,=\,
b-\sum_{k=1}^b1_{u_k<p/\kappa}
+\sum_{k=b+1}^\ell1_{u_k>1-p(1-1/\kappa)}
\,.
\end{multline*}
The interest of this map lies in the following fact:
if $U_1,\dots,U_\ell$
are i.i.d. uniform random variables taking values on the interval $[0,1]$,
then for all
$b\in\zl$, the law of 
$\cM_H(b,U_1,\dots,U_\ell)$
is given by the $b$--th row of the mutation matrix $M_H$, i.e.,
$$\forall c\in\zl\qquad P\big(
\cM_H(b,
U_1,\dots,U_\ell)=c\big)\,=\,M_H(b,c)\,.$$
Moreover, we know thanks to lemma 7.1 of \cite{Cerf}
that the map $\cM_H$ is non--decreasing with respect to the Hamming class, i.e.,
for all
$b,c\in\zl$ and $u_1,\dots,u_\ell\in[0,1]$,
$$b\leq c\,\Ra\, 
\cM_H(b,u_1,\dots,u_\ell)\leq
\cM_H(c,u_1,\dots,u_\ell)\,.$$
Take 
$a,b,h\in\zl$ with $a\leq b$ and let
$U_1,\dots,U_\ell$ be i.i.d. uniform random variables on $[0,1]$.
Thanks to the properties of the map $\cM_H$, we have
\begin{multline*}
M_H(a,0)+\cdots+M_H(a,h)\,=\,
P\big(\cM_H(a,U_1,\dots,U_\ell)\leq h\big)\,\geq\\
P\big(\cM_H(b,U_1,\dots,U_\ell)\leq h\big)\,=\,
M_H(b,0)+\cdots+M_H(b,h)\,.
\end{multline*}
This implies the desired result.
\end{proof}
Let us define
$$\cR'\,=\, 
[0,1]\times\zl^2\times [0,1]\,.$$
We build next two coupling maps
$$\Phi'_O, 
\overline{\Phi}_O
:
\pml\times\cR'\lra\pml\,.$$
Take $r=(s,i,j,u)\in\cR'$, and
$o\in\pml$.
We set $l=\cS_O(o,s)$ 
and we set $k$ to be the only index in $\{0,\dots,\ell\}$ such that
$$o(0)+\cdots+o(k-1)\,<\,
j\,\leq\,
o(0)+\cdots+o(k).$$
The maps
$\Phi'_O$ et $\overline{\Phi}_O$
are defined by:
\begin{align*}
\Phi'_O(o,r)\,&=\,
o\big(k\ra \cM'_H(l,u)\big)\,,\\
\overline{\Phi}_O(o,r)\,&=\,
o\big( k\ra \cM^{K+1}_H(l,u) \big)\,.
\end{align*}
We have thanks to lemma \ref{domimut} that
$$\forall r\in\cR'\quad 
\forall o\in\pml\qquad
\Phi'_O(o,r)
\preceq\overline{\Phi}_O(o,r)
\,.$$
Let
$$
R'_n=(S'_n,I'_n,J'_n,U'_n),\quad
n\geq 1
$$
be an i.i.d. sequence of random vectors with values in $\cR'$,
the random variables $S'_n,I'_n,J'_n,U'_n$
being independent and having the uniform law in their corresponding spaces.
We also take the sequence $(R'_n)_{n\geq 1}$
to be independent of the sequence $(R_n)_{n\geq 1}$
defined in section \ref{Lower}.
We build the process $\Ot$
with the help of the sequence $(R'_n)_{n\geq 1}$.
Let $o\in\pml$ be the starting point of the process,
we set $O_0=o$ and
$$\forall n\geq 1\qquad
O_n\,=\,
\Phi'_O(O_{n-1},R'_n)\,.$$
The next lemma shows that the process $\Ot$ 
is monotone.
\begin{lemma}\label{monophiopr}
The coupling map $\Phi'_O$ 
is non--decreasing with respect to the occupancy distribution,
i.e.,
$$\forall o,o' \in\pml \quad \forall r\in\cR'
\qquad
o\preceq o' 
\ \Rightarrow\ 
\Phi'_O(o,r)\,\preceq\,
\Phi'_O(o',r)\,.
$$
\end{lemma}
The proof is very similar to that of lemma 7.5 in \cite{Cerf},
so we do not include it here.
We build next the upper process
$\Otk$.
We define a map
$\pi_{K+1}:\pml\to\pml$
as follows: for
$o\in\pml$ and $l\in\zl$,
\begin{align*}
\pi_{K+1}(o)(l)\,&=\,
\begin{cases}
\quad o(l) & \text{if}\ 0\leq l\leq K\,, \\
\quad m-(o(0)+\cdots +o(K)) & \text{if}\ l=K+1\,,\\
\quad 0 & \text{if}\ K+2\leq l\leq\ell\,.
\end{cases}
\end{align*}
This map satisfies
$$
\forall o\in\pml\qquad 
o \, \preceq\, 
\pi_{K+1}(o)\,.$$
We also define
\begin{align*}
o^{K+1}_{\text{exit}}\,&=\,
(0,m,0,\dots,0)\,,
\qquad 
&&o^{K+1}_{\text{enter}}\,=\,
(1,m-1,0,\dots,0)\,.
\end{align*}
We build an upper map
$\Phi_O^{K+1}$ by setting
for $o\in\pml$ and $r\in\cR'$,
\begin{equation*}
\index{$\Phi_O^{K+1}$}
\Phi_O^{K+1}(o,r)=
\begin{cases}
\,
\Phi'_{O}(o,r)
&\text{if }o\in\cN \,\,{\text{ and }} 
\Phi'_O(o,r)\not\in\cW \\
\,
o^{K+1}_{\text{enter}}
&\text{if }o\in\cN \,\,{\text{ and }}
\Phi'_O(o,r)\in\cW \\
\,
\pi_{K+1}\big(\overline{\Phi}_{O}(\pi_{K+1}(o),r)\big)
&\text{if }o\in\cW\ \text{and }
\overline{\Phi}_{O}(\pi_{K+1}(o),r)\not\in\cN\\
\,
o^{K+1}_{\text{exit}}
&\text{if }o\in\cW\ \text{and }
\overline{\Phi}_{O}(\pi_{K+1}(o),r)\in\cN\\
\end{cases}
\end{equation*}
A proof similar to that of proposition \ref{compphiol} shows that:
\begin{proposition}\label{comphiok} For all
$r\in\cR'$
and for all
$o\in\pml$,
$$
\Phi'_{O}(o,r)
\,\preceq\,
\Phi_O^{K+1}(o,r)
\,.
$$
\end{proposition}
%
%
We define an upper process
$(O^{K+1}_t)_{t\geq 0}$ 
with the help of the i.i.d. sequence
$(R'_n)_{n\geq 1}$ 
and the map
$\Phi_O^{K+1}$.
Let $o\in\pml$ be the starting point of the process,
we set
$O^{K+1}_0=o$ and
$$\forall n\geq 1\qquad
O^{K+1}_n\,=\,\Phi_O^{K+1}\big(O^{K+1}_{n-1}, R'_n\big)
\,.
$$
\begin{proposition}\label{compok}
Suppose that the processes
$(O_t)_{t\geq 0}$,
$(O^{K+1}_t)_{t\geq 0}$,
have the same starting occupancy distribution $o$.
We have then
$$\forall t\geq 0\qquad
O_t
\,\preceq\,
O^{K+1}_t
\,.$$
\end{proposition}
See proposition 8.1 of \cite{Cerf} for a detailed proof.

\subsection{Dynamics of the upper process}\label{Dynaup}
We will now study the dynamics of the upper process 
$\Otk$
in $\cW$.
For the process
$\Otk$, 
the states in the set
$$\cT^{K+1}\,=\,\big\{\,o\in\pml:
o(0)\geq 1\text{ and }\ozk+o(K+1)<m
\,\big\}\,
$$
are transient,
and the states in the set
$\smash{\cN\cup\big(\cW\setminus\cT^{K+1}\big)}$ 
form a recurrent class.
We will therefore focus on the dynamics of the process $\Otk$
restricted to
$\smash{\cW\setminus\cT^{K+1}}$.
Since
$$\smash{\cW\setminus\cT^{K+1}}
\,=\,\big\{\,
o\in\pml:
o(0)\geq 1\text{ and }\ozk+o(K+1)=m
\,\big\}\,,$$
a state in
$\smash{\cW\setminus\cT^{K+1}}$
is completely determined by the occupation numbers of the classes
$0,\dots,K$.
The process 
$\Otk$ always enters the set
$\cW\setminus\cT^{K+1}$ at
$o^{K+1}_{\text{enter}}$. 
The only possible transitions for the Hamming classes
$0,\dots,K$ of the process
$\Otk$
starting from a state in
$\smash{\cW\setminus\cT^{K+1}}$ are
\begin{multline*}
(o(0),\dots,o(K))
\quad \longrightarrow\quad
(o(0),\dots,o(K))-w_i \\
\text{if }\ 
1\leq o(i),\quad 0\leq i\leq K\,,
\end{multline*}
\begin{multline*}
(o(0),\dots,o(K))
\quad \longrightarrow\quad
(o(0),\dots,o(K))+w_i \\
\text{if }\ 
\ozk\leq m-1,\quad 0\leq i\leq K\,,
\end{multline*}
\begin{multline*}
(o(0),\dots,o(K))
\quad \longrightarrow\quad
(o(0),\dots,o(K))-w_i+w_j \\
\text{if }\ 
1\leq o(i),\quad 0\leq i,j\leq K,\quad i\neq j\,.
\end{multline*}
The process
$\Otk$ 
always exits the set
$\cW\setminus\cT^{K+1}$ at
$o^{K+1}_{\text{exit}}$.
If the process
$\Otk$ starts from a state in
$\smash{\cW\setminus\cT^{K+1}}$,
until the time of exit from
$\cW\setminus \cT^{K+1}$,
the dynamics of
$(O^{K+1}_t(0),\dots,O^{K+1}_t(K))_{t\geq 0}$ 
is that of a Markov chain
on the state space
$$\cE_K\,=\,
\lbrace\, z\in\N^{K+1} : z_0+\cdots+z_K \leq m  \,\rbrace\,.$$
Let us compute the associated transition probabilities.
Let $z\in\cE_K$.

$\bullet $ For
$0\leq i\leq K$ and
$0<z_0+\cdots+z_K<m\,,$
\begin{multline*}
p(z,z+w_i)\,=\,
\frac{m-\displaystyle \sum_{l=0}^K z_l}{m\big((\s-1)z_0+m\big)}\\
\times\Biggr( \s z_0 M^{K+1}_H(0,i)+
\sum_{l=1}^K z_l M^{K+1}_H(l,i)+
\bigg( m-\sum_{l=0}^K z_l \bigg)M^{K+1}_H({K+1},i) \Biggr)\,.
\end{multline*}
\par\noindent
$\bullet $ For
$0\leq i\leq K$ and
$1\leq z_i\,,$
\begin{multline*}
p(z,z-w_i)\,=\,
\frac{z_i}{m\big((\s-1)z_0+m\big)}\times
\Biggr( \s z_0 \biggr(1-\sum_{h=0}^K M^{K+1}_H(0,h)\biggr)\\
+\sum_{l=1}^K z_l \biggr(1-\sum_{h=0}^K M^{K+1}_H(l,h)\biggr)+
\bigg( m-\sum_{l=0}^K z_l \bigg)\biggr(1-\sum_{h=0}^K M^{K+1}_H({K+1},h)\biggr) \Biggr)\,.
\end{multline*}
\par\noindent
$\bullet $ For
$0\leq i,j\leq K$, $i\neq j$ and
$1\leq z_i\,,$
\begin{multline*}
p(z,z-w_i+w_j)\,=\,
\frac{z_i}{m\big((\s-1)z_0+m\big)}\\
\times\Biggr( \s z_0 M^{K+1}_H(0,j)+
\sum_{l=1}^K z_l M^{K+1}_H(l,j)+
\bigg( m-\sum_{l=0}^K z_l \bigg)M^{K+1}_H({K+1},j) \Biggr)\,.
\end{multline*}
The other non--diagonal coefficients of the matrix are null.
The diagonal coefficients are arranged so that the matrix is stochastic, i.e.,
the sum over each row equals 1.

Since we are interested in the dynamics of $\Otk$
in $\smash{\cW\setminus\cT^{K+1}}$,
the transition probabilities starting from a point in
$\lbrace\, z\in\cE_K :
z_0=0 \,\rbrace$
are not relevant.
Moreover, the law of the exit point from
$\lbrace\, z\in\cE_K :
z_0\geq 1 \,\rbrace$
is also not relevant,
what matters is the law of the exit time.
Therefore we will modify the matrix $p$
into another stochastic matrix $p^{K+1}$ 
such that:

$\bullet$ Starting from 
$\lbrace\, z\in\cE_K :
z_0=0 \,\rbrace$,
there is a jump with probability $1$
to $(1,m-1,0,\dots,0)$.

$\bullet$ The law of the exit time from
$\lbrace\, z\in\cE_K :
z_0\geq 1 \,\rbrace$
is unchanged,
but the exit point is
$(0,\dots,0)$
with probability 1.

More precisely, 
we define the matrix $p^{K+1}$
as follows:

$\bullet$ For $z,z'\in\cE_K$ with $z_0=0$,
\begin{align*}
p^{K+1}\big(z,(1,m-1,\dots,0)\big)\,&=\,1\,,\\
p^{K+1}(z,z')\,&=\,0\qquad
\text{if }\ z'\neq(1,m-1,\dots,0)\,.
\end{align*}
$\bullet $ For $z,z'\in\cE_K$ with
$z_0=1$ and $z'_0=0$,
\begin{align*}
p^{K+1}\big(z,(0,\dots,0)\big)\,&=\,
p(z,z-w_0)+\sum_{i=1}^K p(z,z-w_0+w_i)\,,\\
p^{K+1}(z,z')\,&=\,0\qquad \text{if }\ z'_0=0\ \text{ and }\ z'\neq(0,\dots,0)\,.
\end{align*}
Finally,
$p^{K+1}(z,z')=p(z,z')$
for all remaining
$z,z'\in\cE_K$.

Let $\Ztk$
be a Markov chain with state space $\cE_K$,
starting at the point $z^{K+1}=(1,m-1,0,\dots,0)$
and having for transition matrix $p^{K+1}$.
Since the process $\Otk$ always enters $\smash{\cW\setminus\cT^{K+1}}$ 
at $o^{K+1}_\text{enter}$
and its dynamics inside $\smash{\cW\setminus\cT^{K+1}}$
is the same as the dynamics of the chain $\Ztk$,
we will rely on this Markov chain to compute the desired estimates.

\subsection{Bounds on the invariant measure}\label{Bounds}
We denote by 
$\mol$, $\mo$, $\mok$
the invariant probability measures of
$\Otl$, $\Ot$, $\Otk$.
Let $\nu_K$
be the image measure of $\mo$ through the map
$$o\in\pml\,\lma\,
\frac{1}{m}\big(
o(0)+\cdots+o(K)
\big)\in
[0,1]\,.$$
For any function $f:[0,1]\lra\R$,
\begin{align*}
\int_{[0,1]}f\,d\nu_K
\,&=\,
\int_{\textstyle\pml}
f\bigg(\frac{
\ozk
}{m}
\bigg)\,d\mu_O(o)\\
\,&=\,
\lim_{t\to\infty} 
E\bigg(
f\bigg(
\frac{O_t(0)+\cdots+O_t(K)}{m}
\bigg)
\bigg)
\,.
\end{align*}
We fix a non--decreasing function
$f:[0,1]\to\R$ such that $f(0)=0$.
Thanks to
proposition~\ref{compol} we have the following inequality:
$$
\forall t\geq 0\qquad
f\bigg(
\frac{O_t^\ell(0)+\cdots+O_t^\ell(K)}{m}
\bigg)\,
\leq\,
f\bigg(
\frac{O_t(0)+\cdots+O_t(K)}{m}
\bigg)\,.$$
Moreover, thanks to proposition~\ref{compok},
$$
\forall t\geq 0\qquad
f\bigg(
\frac{O_t(0)+\cdots+O_t(K)}{m}
\bigg)
\,\leq\,
f\bigg(
\frac{O_t^{K+1}(0)+\cdots+O_t^{K+1}(K)}{m}
\bigg)
\,.$$
We take the expectations 
and we send $t$ to $\infty$, and we obtain
\begin{multline*}
\int_{\textstyle\pml}
f\bigg(
\frac{
\ozk
}{m}
\bigg)
\,d\mu_O^\ell(o)
\\
\leq\,
\int_{[0,1]}f\,d\nu_K
\,\leq
\\
\int_{\textstyle\pml}
f\bigg( \frac{ \ozk }{m} \bigg)
\,d\mu_O^{K+1}(o)
\,.
\end{multline*}
Our next goal is to find estimates of the above integrals.
The strategy is the same for the lower and upper integrals.
Let $\t$ be either $K+1$ or $\ell$
and let us study the invariant probability measure $\mu^\theta_O$.
We will rely on a renewal result.
Let $\xt$ be a discrete time Markov chain
taking values in a finite space $\cE$.
We suppose that $\xt$ is irreducible and aperiodic 
and we call $\mu$ its invariant probability measure.
\begin{proposition}\label{renewal}
Let $\cW$ be a subset of $\cE$
and let $e$ be a point in $\cE\setminus\cW$.
Let $f$ be a function from $\cE$ to $\R$. We define
$$\tau^* \,=\,\inf\,\big\{\,t\geq 0: 
X_t\in\cW
\,\big\}\,,\qquad
\tau \,=\,\inf\,\big\{\,t\geq \tau^*: 
X_t=e
\,\big\}\,.
$$
We have
$$
\int_\cE f(x)\,d\mu(x)\,=\,
\frac{1}{E(\tau\,|\,X_0=e)}\,
E\bigg(\int_0^{\tau}
f(X_s)\,ds\,\Big|\,X_0=e\bigg)
\,.$$
\end{proposition}
The proof is standard and similar to the proof 
of proposition 8.2 in \cite{Cerf}, so we omit it.
We apply this renewal result to the process
$(O^\theta_t)_{t\geq 0}$ restricted to
$\smash{\cN\cup\big(\cW\setminus\cT^\theta\big)}$,
the set $\cW\setminus\cT^\theta$, 
the occupancy distribution $\ote$
and the function
${o\longmapsto f\big(\big(\ozk\big)/m\big)}$.
Set
\begin{align*}
\index{$\tau^*$}
\tau^* \,&=\,\inf\,\big\{\,t\geq 0: 
O^\theta_t\in\cW
\,\big\}\,,\cr
\tau \,&=\,\inf\,\big\{\,t\geq \tau^*: 
O^\theta_t=\ote
\,\big\}\,.
\end{align*}
We then have
\begin{multline*}
\int_{\textstyle\pml}
f\bigg( \frac{\ozk }{m} \bigg)
\,d\mu_O^\theta(o)
\\=\,
\frac{
\displaystyle
E\bigg(\int_0^{\tau}
f\bigg( \frac{ 
O^\theta_s(0)+\cdots+O^\theta_s(K)
 }{m} \bigg)
\,ds\,\Big|\,
O^\theta_0=\ote
\bigg)
}{
\displaystyle
E\big(\tau\,|\,
O^\theta_0=\ote
\big)}\\=\,
\frac{
\displaystyle
E\bigg(\int_0^{\tau^*}
f\bigg( \frac{ 
O^\theta_s(0)+\cdots+O^\theta_s(K)
 }{m} \bigg)
\,ds\,\Big|\,
O^\theta_0=\ote
\bigg)}
{\displaystyle
E\big(\tau\,|\,
O^\theta_0=\ote
\big)}\\
+\frac{
\displaystyle
E\bigg(\int_{\tau^*}^{\tau}
f\bigg( \frac{ 
O^\theta_s(0)+\cdots+O^\theta_s(K)
 }{m} \bigg)
\,ds\,\Big|\,
O^\theta_0=\ote
\bigg)}
{\displaystyle
E\big(\tau\,|\,
O^\theta_0=\ote
\big)}\,
\,.
\end{multline*}
As long as
$(O^\theta_t)_{t\geq 0}$ is in
$\cW\setminus\cT^\theta$,
the dynamics of
$(O^\theta_t(0),\dots,O^\t_t(K))_{t\geq 0}$ 
is that of the Markov chain
$(Z^\theta_t)_{t\geq 0}$ 
defined at the end of the sections \ref{Dynalow} and \ref{Dynaup}.
Suppose that
$(Z^\theta_t)_{t\geq 0}$ 
starts from
$z^\t$,
where $z^\t$ is the point of $\cE_K$ given by
$$
z^\t\,=\,
\begin{cases}
\quad z^{K+1}\,=\,(1,0,\dots,0)\qquad&\text{if }\ \t=K+1\,,\\
\quad z^\ell\,=\,(1,m-1,0,\dots,0)\qquad&\text{if }\  \t=\ell\,.
\end{cases}$$
Let $\tau_0$ be the first time that the coordinate $0$ becomes null, i.e.,
$$\tau_0\,=\,\inf\,\big\{\,n\geq 0: Z^\theta_n(0)=0\,\big\}\,.
\index{$\tau_0$}$$
Since the process $\Ott$
always enters the set $\cW\setminus\cT^\theta$
at $o^\t_{\text{enter}}$,
the law of $\tau_0$ 
is the same as the law of $\tau-\tau^*$
whenever the process $\Ott$ starts from $\ote$.
We conclude that the laws of
$\big((O^\theta_t(0),\dots,O^\t_t(K)),\, {\tau^*}\leq t  \leq {\tau}\big)$
and
$\big(Z^\theta_t\,, 0\leq t\leq \tau_0\big)$
are the same. In particular,
$$
E\big(\tau-{\tau^*}
\,\big|\,
O^\theta_0=\ote
\big)
\,=\,
E\big({\tau_0}
\,\big|\,
Z^\t_0=z^\t
\big)\,,
$$
and also
\begin{multline*}
E\bigg(\int_{\tau^*}^{\tau}
f\bigg(\frac{O^\theta_s(0)+\cdots+O^\t_s(K)}{m}\bigg)\,ds\,\Big|\,
O^\theta_0=\ote
\bigg)
\,=\\
E\bigg(\int_{0}^{\tau_0}
f\bigg(\frac{Z^\theta_s(0)+\cdots+Z^\t_s(K)}{m}\bigg)\,ds
\,\Big|\, 
Z^\t_0=z^\t\bigg)\,.
\end{multline*}
The formula for the invariant measure
$\mu^\theta_O$ can then be written as follows:
\begin{multline*} \int_{\textstyle\pml}
f\bigg( \frac{ \ozk }{m} \bigg)
\,d\mu_O^\theta(o)
\,=\\
\frac{
\displaystyle
E\bigg(\int_0^{\tau^*}
f\bigg( \frac{ 
O^\theta_s(0)+\cdots+O^\theta_s(K)
 }{m} \bigg)
\,ds\,\Big|\,
O^\theta_0=\ote
\bigg)}
{\displaystyle
E\big(\tau^*\,|\,
O^\theta_0=\ote
\big)+E\big(
\tau_0\,|\,
Z^\t_0=z^\t
\big)}\\
+\frac{
\displaystyle
E\bigg(\int_0^{\tau_0}
f\bigg( \frac{ 
Z^\theta_s(0)+\cdots+Z^\theta_s(K)
 }{m} \bigg)
\,ds\,\Big|\,
Z^\theta_0=z^\t
\bigg)}
{\displaystyle
E\big(\tau^*\,|\,
O^\theta_0=\ote
\big)+E\big(
\tau_0\,|\,
Z^\t_0=z^\t
\big)}
\,
\,.
\end{multline*}
The process $\Ztt$ always enters the set
$\{z\in\cE_K:z_0\geq 1\}$ at
$z^\t$, and always exits the set
$\{z\in\cE_K:z_0\geq 1\}$ at $(0,\dots,0)$.
In order to write the previous formula in terms of the invariant probability measure of
$(Z^\theta_t)_{t\geq 0}$,
we apply the renewal result stated in proposition \ref{renewal}
to the process
$(Z^\theta_t)_{t\geq 0}$,
the set ${\{z\in\cE_K:z_0\geq 1\}}\,$, 
the point $0$ 
and the map $z\longmapsto f\big((z_0+\cdots+z_K)/m\big)$.
We set
\begin{align*}
\tau_1 \,&=\,\inf\,\big\{\,t\geq 0: 
Z^\theta_t(0)\geq 1
\,\big\}\,,\\
\tau'_0\,&=\,\inf\,\big\{\,t\geq \tau_1: 
Z^\theta_t= 0
\,\big\}\,,
\end{align*}
and we denote by $\nu^\theta$ the invariant probability measure of the process
$(Z^\theta_t)_{t\geq 0}\,.$
We have
\begin{multline*}
\int_{\textstyle\cE_K}
f\Big( \frac{ z_0+\cdots+z_K }{m} \Big)\,
d\nu^\theta(z)
\,=\\
\frac{
\displaystyle
E\bigg(\int_0^{\tau'_0}
f\bigg(\frac{Z^\theta_s(0)+\cdots+Z^\t_s(K)}{m}\bigg)\,ds\,\Big|\,
Z^\theta_0= 0\bigg)
}{
\displaystyle
E\big(\tau'_0\,|\,
Z^\theta_0= 0
\big)
}
\,.
\end{multline*}
Since $p^\t(0,z^\t)=1$, the Markov property yields
\begin{multline*}
\int_{\textstyle\cE_K}
f\Big( \frac{ z_0+\cdots+z_K }{m} \Big)\,
d\nu^\theta(z)
\,=\\
\frac{
\displaystyle
E\bigg(\int_0^{\tau_0}
f\bigg(\frac{Z^\theta_s(0)+\cdots+Z^\t_s(K)}{m}\bigg)\,ds
\,\Big|\,
Z^\t_0=z^\t \bigg)
}{
\displaystyle 1+
E\big(\tau_0\,|\,
Z^\t_0=z^\t\big)
}
\,.
\end{multline*}
Reporting back in the formula for
$\mu^\t_O$, we get:
\begin{multline*}
\int_{\textstyle\pml}
f\bigg( \frac{ \ozk }{m} \bigg)
\,d\mu_O^\theta(o)
\,=\,
\cr
\frac{
\displaystyle
E\bigg(\int_0^{\tau^*}
f\bigg( \frac{ 
O^\theta_s(0)+\cdots+O^\theta_s(K)
 }{m} \bigg)
\,ds\,\Big|\,
O^\theta_0=\ote
\bigg)
}{
\displaystyle
E\big(\tau^*\,|\,
O^\theta_0=\ote
\big)+
E\big(\tau_0\,|\,
Z^\t_0=z^\t\big)
}\\
+\frac{
\displaystyle
1
+
E\big(\tau_0\,|\,
Z^\t_0=z^\t\big)
}{
\displaystyle
E\big(\tau^*\,|\,
O^\theta_0=\ote
\big)+
E\big(\tau_0\,|\,
Z^\t_0=z^\t\big)
}
\int_{\textstyle\cE_K}
f\Big( \frac{ z_0+\cdots+z_K }{m} \Big)
d\nu^\theta(z)
\,.
\end{multline*}
In the sequel, 
we will try to estimate each of the terms appearing on the right--hand side of this formula.

\section{Induction and estimates}\label{Indest}
In this section we estimate the integral
$$\int_{\textstyle\cE_K}
f\Big(
\frac{z_0+\cdots+z_K}{m}
\Big)
d\nu^\t(z) \,.$$
Suppose that $\s\exa>1$
and let $\rho^*=(\rho^*_0,\dots,\rho^*_K)$ be the point of $[0,1]^{K+1}$ given by
$$
\rho^*_k\,=\,
(\s\exa-1)
\frac{a^k}{k!}
\sum_{i\geq 1}
\frac{i^k}{\s^i}\,,\qquad 
0\leq k\leq K\,.
$$
Let 
$f:\pml\lra\R$
be a non--decreasing continuous function with $f(0)=0$.
We consider the following asymptotic regime:
$$\displaylines{
\ell\to +\infty\,,\qquad m\to +\infty\,,\qquad q\to 0\,,\cr
{\ell q} \to a\,,
\qquad\frac{m}{\ell}\to\alpha\,.}$$
Our goal is to prove that in this regime
$$\lim_{
\genfrac{}{}{0pt}{1}{\ell,m\to\infty,\,
q\to 0
}
{{\ell q} \to a,\,
\frac{\scriptstyle m}{\scriptstyle \ell}\to\alpha
}
}\,
\int_{\textstyle\cE_K}
f\Big(
\frac{z_0+\cdots+z_K}{m} 
\Big)
d\nu^\t(z)\,=\,
f(\rho^*_0+\cdots+\rho^*_K)\,.
$$
In order to prove this convergence,
we estimate some hitting times associated to the Markov chain
$\Ztt$.
Let us define for 
$0\leq k\leq K$ and for
$\d>0$,
$$U_k(\d)\,=\,
\bigg\lbrace\,
z\in\cE_K : 
\Big| 
\frac{z_i}{m}-\rho^*_i
\Big|<\d,\,
0\leq i\leq k
\,\bigg\rbrace\,.$$
We also define, for any subset $A\subset\cE_K$
the hitting time of $A$:
$$\tau(A)\,=\,
\inf\lbrace\,
t\geq 0:
Z^\t_t\in A
\,\rbrace\,.$$
\begin{theorem}\label{polexp}
Let $\d>0$.
For all 
$0\leq k\leq K$,
there exist positive real numbers
$\a_k,\a'_k,\b_k,\b'_k$ 
(all of them depending on $\d$), such that for
$\ell,m$ big enough and for $q$ small enough, we have:

$\bullet$ For all $z\in\cE_K$,
$$P\big(\tau(U_k(\d))\geq m^{\a_k}
\,\big|\, 
Z^\t_0=z\big)\,\leq\,
\exp(-\a'_k m)\,.$$
$\bullet$ For all $z\in U_k(\d)$,
$$P\big(\tau(U_k(2\d)^c)\leq \exp(\b_k m)
\,\big|\, 
Z^\t_0=z\big)\,\leq\,
\exp(-\b'_k m)\,.$$
\end{theorem}
We will prove this theorem by induction on $k$.
The strategy is as follows. From the definition of $U_k(\d)$ we see that
$$U_0(\d)\,\supset\,
U_1(\d)\,\supset\,
\cdots\,\supset\,
U_K(\d)\,,$$
and that in order to know whether the process
$\Ztt$ is in $U_k(\d)$,
it is enough to check the first $k+1$ coordinates,
$Z^\t_t(0),\dots,Z^\t_t(k)$.
The process $\Ztt$
has been built in such a way that the dynamics of
$\ztk$
does not depend on the coordinates $k+1,\dots,K$.
The case $k=0$ 
boils down to the study of birth and death Markov chains,
which are very similar to the ones studied in section 9 of \cite{Cerf}.
If the estimates hold at rank $k-1$,
then we know that the process 
$\Ztt$
spends almost all of its time inside the set
$U_{k-1}(\d)$.
As long as the process is in $U_{k-1}(\d)$,
we can bound stochastically the dynamics of 
$\ztk$ with a pair of birth and death chains,
which can be studied with the same techniques
as the birth and death chains from the case $k=0$. 
To begin with,
we show that the process 
$\ztk$
does not depend on the coordinates 
$k+1,\dots,K$.
Let us take $k$ in $\{0,\dots,K\}$, 
$z=(z_0,\dots,z_K)$ a point in $\cE_K$, 
and let us compute the following transition probabilities:
\begin{align*}
P\big(
Z^\t_{t+1}(k)=z_k+1
\,|\,
Z^\t_t=z
\big)\,&=\,
p^\t(z,z+w_k)+
\sum_{\genfrac{}{}{0pt}{1}{0\leq i\leq K}
{i\neq k}} p^\t(z,z+w_k-w_i)\,,\\
P\big(
Z^\t_{t+1}(k)=z_k-1
\,|\,
Z^\t_t=z
\big)\,&=\,
p^\t(z,z-w_k)+
\sum_{\genfrac{}{}{0pt}{1}{0\leq i\leq K}
{i\neq k}} p^\t(z,z-w_k+w_i)\,.
\end{align*}
In the case of the lower process $\Ztl$, we obtain
\begin{multline*}
p^\ell(z,z+w_k)+
\sum_{\genfrac{}{}{0pt}{1}{0\leq i\leq K}
{i\neq k}} p^\ell(z,z+w_k-w_i)\,=\\
\frac{m-z_k}{m\big((\s-1)z_0+m\big)}\times
\Bigg(
\s z_0 M_H(0,k)+
\sum_{l=1}^{k} z_l M_H(l,k)
\Bigg)\,,
\end{multline*}
\begin{multline*}
p^\ell(z,z-w_k)+
\sum_{\genfrac{}{}{0pt}{1}{0\leq i\leq K}
{i\neq k}} p^\ell(z,z-w_k+w_i)\,=\,
\frac{z_k}{m\big((\s-1)z_0+m\big)}\\
\times\Bigg(
\s z_0 \big(1-M_H(0,k)\big)+
\sum_{l=1}^{k} z_l \big(1-M_H(l,k)\big)
+m-\sum_{l=0}^k z_l
\Bigg)\,.
\end{multline*}
In the case of the upper process $\Ztk$, we obtain
\begin{multline*}
p^{K+1}(z,z+w_k)+
\sum_{\genfrac{}{}{0pt}{1}{0\leq i\leq K}
{i\neq k}} p^{K+1}(z,z+w_k-w_i)\,=\,
\frac{m-z_k}{m\big((\s-1)z_0+m\big)}\\
\times\Bigg(
\s z_0 M_H(0,k)+
\sum_{l=1}^{k} z_l M_H(l,k)+
\bigg( m-\sum_{l=0}^k z_l \bigg)M_H(k+1,k)
\Bigg)\,,
\end{multline*}
\begin{multline*}
p^{K+1}(z,z-w_k)+
\sum_{\genfrac{}{}{0pt}{1}{0\leq i\leq K}
{i\neq k}} p^{K+1}(z,z-w_k+w_i)\,=\\
\frac{z_k}{m\big((\s-1)z_0+m\big)}
\Bigg(
\s z_0 \big(1-M_H(0,k)\big)+
\sum_{l=1}^{k} z_l \big(1-M_H(l,k)\big)\\+
\bigg( m-\sum_{l=0}^k z_l \bigg) \big(1-M_H(k+1,k)\big)
\Bigg)\,.
\end{multline*}
We observe that none of the above expressions depend on
$z_{k+1},\dots,z_K$.
Moreover, if we define for
$k\in\{0,\dots,K\}$,
$$M_H^\ell(\ell,k)\,=\,0\,,\qquad
M_H^{K+1}(K+1,k)\,=\,
M_H(k+1,k)\,,$$
we can rewrite the above transition probabilities as follows:
\begin{multline*}
p^\t(z,z+w_k)+
\sum_{\genfrac{}{}{0pt}{1}{0\leq i\leq K}
{i\neq k}} p^\t(z,z+w_k-w_i)\,=\,
\frac{m-z_k}{m\big((\s-1)z_0+m\big)}\\
\times\Bigg(
\s z_0 M_H(0,k)+
\sum_{l=1}^{k} z_l M_H(l,k)+
\bigg( m-\sum_{l=0}^k z_l \bigg)M_H^\t(\t,k)
\Bigg)\,,
\end{multline*}
\begin{multline*}
p^\t(z,z-w_k)+
\sum_{\genfrac{}{}{0pt}{1}{0\leq i\leq K}
{i\neq k}} p^\t(z,z-w_k+w_i)\,=\\
\frac{z_k}{m\big((\s-1)z_0+m\big)}
\Bigg(
\s z_0 \big(1-M_H(0,k)\big)+
\sum_{l=1}^{k} z_l \big(1-M_H(l,k)\big)\\+
\bigg( m-\sum_{l=0}^k z_l \bigg) \big(1-M^\t_H(\t,k)\big)
\Bigg)\,.
\end{multline*}
We first give some general results concerning birth and death Markov chains,
which we then use to study both the initial case $k=0$
and the inductive step.

\subsection{Birth and death Markov chains}\label{Bdchains}
First of all we give some explicit formulas for mean hitting times
of birth and death Markov chains.
These formulas can be derived from the classical formulas for finite state space Markov chains (see \cite{FW}, chapter 6).
Next, we apply these formulas to study the limiting behaviour of a family
of birth and death Markov chains,
under suitable hypotheses.
Let $(Z_t)_{t\geq 0}$ be a birth and death Markov chain
taking values in $\zm$
and having the following transition probabilities
$$
\begin{array}{ll}
P\big( Z_{t+1}=i+1 \,\big|\, Z_t=i \big)\,=\,\d_i\,,
\quad & \quad 0\leq i<m\,,\\
P\big( Z_{t+1}=i-1 \,\big|\, Z_t=i \big)\,=\,\g_i\,,
\quad & \quad 0<i\leq m\,.
\end{array}
$$
Let us define
$$
\pi(0)=1\,,\qquad
\pi(i)\,=\,
\frac{{\delta_1\cdots\delta_{i}}}{{\gamma_{1}\cdots\gamma_{i}}}
\,,\qquad
0<i<m\,.
\index{$\pi(i)$}
$$
For a subset $A\subset\zm$,
we define the hitting time of $A$ by
$$\tau(A)\,=\,
\inf\big\lbrace
t\geq 0 :
Z_t\in A
\big\rbrace\,.$$
Let $a,b$ be two points in $\zm$, with $a<b$.l 
We have the following formula for the hitting time of $b$ starting from $a$:
$$
E\big(
\tau(\{b\})
\,\big|\,
Z_0=a
\big)\,=\,
\sum_{i=a}^{b-1}
\sum_{j=i}^{b-1}
\frac{1}{\d_i}
\frac{\pi(i)}{\pi(j)}\,.
$$
Likewise, we have the following formula for the hitting time of $a$ starting from $b$:
$$
E\big(
\tau(\{a\})
\,\big|\,
Z_0=b
\big)\,=\,
\sum_{i=a}^{b-1}
\sum_{j=i}^{b-1}
\frac{1}{\g_{i+1}}
\frac{\pi(i)}{\pi(j)}\,.
$$
Let $a<i<b$ be three points in $\zm$.
We have the following formulas for the exit point of $\{\,a,\dots,b\,\}$:
\begin{align*}
P\big(
Z_{\tau(\{a,b\})}=a
\,\big|\,
Z_0=i
\big)\,&=\,
\frac{\displaystyle
\sum_{j=i}^{b-1}\frac{1}{\pi(j)}}
{\displaystyle
\sum_{j=a}^{b-1}\frac{1}{\pi(j)}}\,,\\
P\big(
Z_{\tau(\{a,b\})}=b
\,\big|\,
Z_0=i
\big)\,&=\,
\frac{\displaystyle
\sum_{j=a}^{i-1}\frac{1}{\pi(j)}}
{\displaystyle
\sum_{j=a}^{b-1}\frac{1}{\pi(j)}}\,.
\end{align*}
We consider now a family of birth and death Markov chains
$\Zt$
depending on four parameters:
$m,\ell\geq 1$, $q\in[0,1]$, $\d'>0$.
For $m,\ell\geq0$, $q\in[0,1]$ and $\d'>0$,
the process
$\Zt$
is a birth and death Markov chain with state space $\zm$
and transition probabilities given by:
\begin{align*}
P\big(
Z_{t+1}=i+1
\,|\,
Z_t=i
\big)\,&=\,
\d_i\,,\qquad
0\leq i<m\,,\\
P\big(
Z_{t+1}=i-1
\,|\,
Z_t=i
\big)\,&=\,
\g_i\,,\qquad
0<i\leq m\,,
\end{align*}
where $\d_i,\g_i$ can depend on the parameters $m,\ell,q,\d'$.
We are interested in the asymptotic behaviour
of $\Zt$
when $m,\ell$ go to $\infty$ and $q,\d'$ go to $0$.
We make the following assumptions.

\textbf{Assumption 1.}
There exist a constant $C>0$ (which can depend on $\ell,q,\d'$ but not on $m$) and an integer $k\geq 1$ such that
for all $m,\ell$ large enough and $q,\d'$ small enough
\begin{align*}
\forall i\in\lbrace\,0,\dots,m-1\,\rbrace\qquad
&\d_i\geq \frac{C}{m^k}\,,\\
\forall i\in\lbrace\,1,\dots,m\,\rbrace\qquad
&\g_i\geq \frac{C}{m^k}\,.
\end{align*}

\textbf{Assumption 2.} 
For all $\rho\in[0,1]$, 
we have the following large deviation estimate:
$$\lim_{\genfrac{}{}{0pt}{1}{\ell,m\to\infty}
{q\to 0}
}\,
\frac{1}{m}\ln\pi(\lfloor \rho m\rfloor)\,=\,
f_{\d'}(\rho)\,,$$
where $f_{\d'}:[0,1]\lra \R$ is a function for which
there exists $\rho_1^*\leq \rho^*\leq \rho_2^*$ in $[0,1]$ such that 
$f_{\d'}$ is increasing over $[0,\rho_1^*[\,$ 
and decreasing over $\,]\rho_2^*,1]$.
Moreover, $\rho^*$ does not depend on $\d'$ and
$$\lim_{\d'\to 0}\rho_1^*\,=\,
\lim_{\d'\to 0}\rho_2^*\,=\,
\rho^*\,.$$

\textbf{Assumption 3.}
For each $m,\ell\geq 1$, $q\in[0,1]$, $\d'>0$,
there exists $r_1\leq r_2$ in $[0,1]$ such that
\begin{align*}
1\leq i\leq j\leq\lfloor r_1 m\rfloor\qquad
&\Longrightarrow\qquad \pi(i)\,\leq \pi(j)\,,\cr
\lfloor r_2 m\rfloor
\leq i\leq j\leq m\qquad
&\Longrightarrow\qquad \pi(i)\,\geq \pi(j)\,.
\end{align*}
Moreover,
$$
\lim_{\genfrac{}{}{0pt}{1}{\ell,m\to\infty}
{q\to 0}
}\,
r_1\,=\,\rho_1^*\,,\qquad
\lim_{\genfrac{}{}{0pt}{1}{\ell,m\to\infty}
{q\to 0}
}\,
r_2\,=\,\rho_2^*\,.
$$

These three assumptions will be verified by the birth and death Markov chains
that we will derive later from the process $\Ztt$.
With these assumptions in hand we can prove the following proposition.
Let us define, for $\d>0$,
$$V(\d)\,=\,
\Big\lbrace\,
i\in\zm:
\Big|\frac{i}{m}-\rho^*\Big|<\d
\,\Big\rbrace\,.$$
For any subset $A\subset \zm$ we define 
the hitting time of $A$ by
$$\tau(A)\,=\,
\inf\lbrace\,
t\geq 0: Z_t\in A
\,\rbrace\,.$$

\begin{proposition}\label{polexpbd}
Suppose that the three assumptions are satisfied.
Let $\d>0$. 
There exist positive real numbers 
$\a,\a',\b,\b'$
(depending on $\d,\d'$)
such that for $m,\ell$ large enough and $q,\d'$ small enough
\begin{align*}
\forall i\in\zm\qquad
&P\big(\tau(V(\d))\geq m^\a
\,|\,
Z_0=i
\big)\,\leq\,
\exp(-\a'm)\,,\\
\forall i\in V(\d)\qquad
&P\big(
\tau(V(2\d)^c)\leq\exp(\b m)
\,|\,
Z_0=i
\big)\,\leq\,
\exp(-\b'm)\,.
\end{align*}
\end{proposition}
\begin{proof}
We begin by showing the first statement in the proposition.
The cases $i\leq \rho^*m$ and $i\geq\rho^*m$ are dealt with in a similar way,
thus, we will only show the result for $i\leq \rho^*m$.
Let us call $b$ the minimum of the discrete interval $V(\d)$,
$$b\,=\,
\lfloor(\rho^*-\d)m\rfloor+1\,.$$
From the formulas provided at the beginning of the section we obtain
$$
E\big(
\tau(V(\d))
\,\big|\,
Z_0=i
\big)\,=\,
E\big(
\tau(\{b\})
\,\big|\,
Z_0=i
\big)\,=\,
\sum_{j=i}^{b-1}
\sum_{k=j}^{b-1}
\frac{1}{\d_j}
\frac{\pi(j)}{\pi(k)}\,.
$$
By assumptions 2 and 3, for $m,\ell$ big enough and $q,\d'$ small enough,
the point 
$r_1$
is in $V(\d)$,
so that $b\leq r_1$, thus
$$1\leq j\leq k\leq b-1\,\Lra\,
\frac{\pi(j)}{\pi(k)}\,\leq\,1\,.$$
Furthermore, by assumption 1, for $m,\ell$ large enough 
and $q,\d'$ small enough,
$$\forall j \in\{\, 0,\dots,m-1 \,\}\qquad
\d_j\,\geq\,
\frac{C}{m^k}\,.$$
It follows that
$$E\big(
\tau(V(\d))
\,\big|\,
Z_0=0
\big)\,\leq\,
\frac{1}{C}m^{k+2}\,.$$
Let $a>k+2$. By the Markov inequality,
$$P\big(
\tau(V(\d))>m^a
\,\big|\,
Z_0=i
\big)\,\leq\,
\frac{1}{C}
m^{-(a-k-2)}\,.$$
We will next estimate for $n\geq 1$, 
$$
P\big(
\tau(V(\d))\geq nm^{a}
\,\big|\,
Z_0=i
\big)\,.
$$
We decompose this probability according to the possible states of the process at time
${(n-1)m^a}$:
\begin{multline*}
P\big(
\tau(V(\d))\geq nm^{a}
\,\big|\,
Z_0=i
\big)\,=\\
\sum_{j<b}
P\big(
Z_{(n-1)m^a}=j,
\tau(V(\d))\geq(n-1)m^a,
\tau(V(\d))\geq nm^{a}
\,\big|\,
Z_0=i
\big)\\
=\,\sum_{j<b}
P\big(
Z_{(n-1)m^a}=j,
\tau(V(\d))\geq(n-1)m^{a}
\,\big|\,
Z_0=i
\big)\\
\times
P\big(
\tau(V(\d))\geq nm^{a}
\,\big|\,
Z_0=i,
Z_{(n-1)m^a}=j,
\tau(V(\d))\geq (n-1)m^{a}
\big)\,.
\end{multline*}
Thanks to the Markov property,
\begin{multline*}
P\big(
\tau(V(\d))\geq nm^{a}
\,\big|\,
Z_0=i,
Z_{(n-1)m^a}=j,
\tau(V(\d))\geq(n-1)m^{a}
\big)\\
=\,P\big(
\tau(V(\d))\geq nm^{a}-(n-1)m^a
\,\big|\,
Z_0=j
\big)\,
\leq\,
\frac{1}{C}
m^{-(a-k-2)}\,.
\end{multline*}
Therefore, for all $n\geq 1$,
\begin{multline*}
P\big(
\tau(V(\d))\geq nm^{a}
\,\big|\,
Z_0=i
\big)\,\leq\\
\frac{1}{C}
m^{-(a-k-2)}
P\big(
\tau(V(\d))\geq 
(n-1)m^a
\,\big|\,
Z_0=i
\big)\,.
\end{multline*}
We iterate this procedure for the times
$(n-2)m^a,\dots,2m^a,m^a$
and we obtain
\begin{align*}
P\big(
\tau(V(\d))\geq nm^{a}
\,\big|\,
Z_0=i
\big)\,&\leq\,
\bigg(
\frac{1}{C}
m^{-(a-k-2)}
\bigg)^n\\
&=\,\exp\Big(
-n\big(
(a-k-2)\ln m
+\ln C
\big)
\Big)\,.
\end{align*}
Thus, setting $n=m$,
we obtain the desired result with
$\a=k+3$ and
$\a'=(\a-k-3)\ln m_0+\ln C$
for $m_0$ large enough so that $\a'>0$.

We show next the second statement of the proposition.
Let $t>0$, $i\in V(\d)$, and let us first
estimate the value of
$$P\big(
\tau(V(2\d)^c)\leq t
\,\big|\,
Z_0=i
\big)\,.$$
Let $\t$ be the last time the process 
$\Zt$
visits the set $V(\d)$ before time $\tau(V(2\d)^c)$, i.e.,
$$\t\,=\,
\max\big\lbrace\,
s<\tau(V(2\d)^c):
Z_s\in V(\d)
\,\big\rbrace\,.$$
We denote by $b$ and $c$ 
the extreme points of the discrete interval $V(\d)$,
$$b\,=\,
\lfloor(\rho^*-\d)m\rfloor+1\,,
\qquad\qquad 
c\,=\,
\lfloor(\rho^*+\d)m\rfloor\,.$$
Likewise, we denote by $b'$ and $c'$ 
the extreme points of the discrete interval $V(2\d)$,
$$b'\,=\,
\lfloor(\rho^*-2\d)m\rfloor+1\,,
\qquad\qquad 
c'\,=\,
\lfloor(\rho^*+2\d)m\rfloor\,.$$
We have then
\begin{multline*}
P\big(
\tau(V(2\d)^c)\leq t
\,\big|\,
Z_0=i
\big)\,=\,
\sum_{s<t}
P\big(
\t=s,
\tau(V(2\d)^c)\leq t
\,\big|\,
Z_0=i
\big)\\=\,
\sum_{s<t}\Big(
P\big(
\t=s,
Z_{s}=b,
\tau(V(2\d)^c)\leq t
\,\big|\,
Z_0=i
\big)\\
+P\big(
\t=s,
Z_{s}=c,
\tau(V(2\d)^c)\leq t
\,\big|\,
Z_0=i
\big)
\Big)\,.
\end{multline*}
Let us consider the first term within the parenthesis.
By the Markov property,
\begin{multline*}
P\big(
\t=s,
Z_{s}=b,
\tau(V(2\d)^c)\leq t
\,\big|\,
Z_0=i
\big)\\=\,
P\bigg(
\begin{matrix}
Z_s=b,
Z_{s+1}=b-1,
\tau(V(2\d)^c)\leq t\\
Z_r\not\in V(\d)\ \text{for}\ s<r\leq\tau(V(2\d)^c)
\end{matrix}
\,\bigg|\,
Z_0=i
\bigg)\\
\leq\,
P\bigg(
\begin{matrix}
Z_r\not\in V(\d)\ \text{for}\ r\leq\tau(V(2\d)^c)\\
\tau(V(2\d)^c)\leq t-s-1
\end{matrix}
\,\bigg|\,
Z_0=b-1
\bigg)
\\
\leq\,
P\big(
Z_{\tau(V(2\d)^c\cup\{b\})}\in V(2\d)^c
\,\big|\,
Z_0=b-1
\big)\\
=\,P\big(
Z_{\tau(\{b'-1,b\})}=b'-1
\,\big|\,
Z_0=b-1
\big)
\,.
\end{multline*}
We now use the formulas provided at the beginning of the section:
$$P\big(
Z_{\tau(\{b'-1,b\})}=b'-1
\,\big|\,
Z_0=b-1
\big)\,=\,
\frac{\displaystyle\frac{1}{\pi(b-1)}}
{\displaystyle\sum_{i=b'-1}^{b-1}\frac{1}{\pi(i)}}\,.$$
Therefore,
$$P\big(
Z_{\tau(\{b'-1,b\})}=b'-1
\,\big|\,
Z_0=b-1
\big)\,\leq\,
\frac{\pi(b'-1)}{\pi(b-1)}\,.$$
Let $\e>0$ and let
$\ell,m$ be large enough and $q$ small enough so that
\begin{align*}
\bigg|
\frac{1}{m}
\ln\pi(b-1)-
f_{\d'}(\rho^*-\d)
\bigg|\,&<\,\frac{\e}{2}\,,\\
\bigg|
\frac{1}{m}
\ln\pi(b'-1)-
f_{\d'}(\rho^*-2\d)
\bigg|\,&<\,\frac{\e}{2}\,.
\end{align*}
We have then
\begin{multline*}
\frac{\pi(b'-1)}{\pi(b-1)}\,=\,
\exp\bigg(
m\Big(
\frac{1}{m}\ln\pi(b'-1)
-\frac{1}{m}\ln\pi(b-1)
\Big)
\bigg)\\
\leq\,
\exp\Big(-m
\big(
f_{\d'}(\rho^*-\d)-f_{\d'}(\rho^*-2\d)-\e
\big)
\Big)\,.
\end{multline*}
Thanks to assumption 2, we can choose $\e$ and $\d'$ small enough so that
$${\e\,<\,
f_{\d'}(\rho^*-\d)-f_{\d'}(\rho^*-2\d)}\,.$$
We choose $\b_1$ as follows:
$$\b_1\,=\,f_{\d'}(\rho^*-\d)-f_{\d'}(\rho^*-2\d)-\e\,.$$
We have $\b_1>0$ and
$$P\big(
\t=s,
Z_{s+1}=b-1,
\tau(V(2\d)^c)\leq t
\,\big|\,
Z_0=i
\big)\,\leq\,
\exp(-\b_1m)\,.$$
The term
$$P\big(
\t=s,
Z_{s+1}=c+1,
\tau(V(2\d)^c)\leq t
\,\big|\,
Z_0=i
\big)$$
is dealt with in a similar fashion, thus obtaining $\b_2>0$ such that
$$P\big(
\t=s,
Z_{s+1}=c+1,
\tau(V(2\d)^c)\leq t
\,\big|\,
Z_0=i
\big)
\,\leq\,
\exp(-\b_2 m)\,.$$
It follows that
$$P\big(
\tau(V(2\d)^c)\leq t
\,\big|\,
Z_0=i
\big)\,\leq\,
t\Big(
\exp(-\b_1 m)+
\exp(-\b_2 m)
\Big)\,.$$
We choose $\b<\min(\b_1,\b_2)$ and
for $m_0$ large enough
$$\b'\,=\,
\min(\b_1-\b,\b_2-\b)-\frac{1}{m_0}\ln2\,.$$
We apply the previous inequality at time
$t=\exp(\b m)$ and
we obtain the desired result.
\end{proof}

\subsection{The initial case}\label{Init}        
We study here the case $k=0$.
We will use the results from the previous section to study
the birth and death Markov chains $(Z^\t_t(0))_{t\geq 0}$,
$\t=k+1$ or $\t=\ell$.
These processes are very similar to the upper and lower
birth and death chains studied in \cite{Cerf}.
In particular,
the analysis in chapter 9 of \cite{Cerf}
is still valid for
$(Z^\t_t(0))_{t\geq 0}$.
The process $(Z^\t_t(0))_{t\geq 0}$ 
is a birth and death Markov chain taking values on $\zm$
and having the following transition probabilities
\begin{align*}
\d_0\,&=\,1
\,,\cr
\delta_i
\,&=\,
\frac{\displaystyle\sigma i(m-i) 
M_H(0,0)+ 
(m-i)^2
M^\t_H(\theta,0) }
{\displaystyle m(\sigma i + m -i)}\,,
&&\quad 1\leq i\leq m-1\,,\\
\gamma_i
\,&=\,
\frac{\displaystyle \sigma i^2 
\big(1-M_H(0,0) \big)+
i(m-i)
\big(1-M^\t_H(\theta,0) \big)
}{\displaystyle m(\sigma i + m -i)}\,,
&&\quad 1\leq i\leq m\,,
\end{align*}
where 
$M^\ell_H(\ell,0)=0$ 
and
$M^{K+1}_H(K+1,0)=M_H(1,0)$.
The mutation probabilities $M^\t_H$
depend on the parameters $\ell$ and $q$.
In particular, 
the process $(Z^\t_t(0))_{t\geq 0}$
belongs to the class of birth and death Markov chains 
studied in the previous section,
even if there is no parameter $\d'$.
We show next that $(Z^\t_t(0))_{t\geq 0}$ 
fulfils the three assumptions made in section \ref{Bdchains}.
From the expressions given for $\d_i$ and $\g_i$
we see that for all $\ell,m\geq 1$, $q\in[0,1]$,
\begin{align*}
\forall i\in\lbrace\, 0,\dots,m-1\,\rbrace\qquad
&\d_i\,\geq\,\frac{M_H(0,0)}{m^2}\,,\\
\forall i\in\lbrace\, 1,\dots,m\,\rbrace\qquad
&\g_i\,\geq\,\frac{1-M_H(0,0)}{m^2}\,.
\end{align*}
Define a function
$\phi:\,
]0,1]\times[0,1[\,\times\,]0,1[\,\lra \,]0,+\infty[\,$
by
$$\phi(\b,\e,\rho)\,=\,
\index{$\phi(\b,\e,\rho)$}
\frac{\displaystyle (1-\rho)\big(\sigma\b \rho+(1-\rho)\e\big)}
{\displaystyle\rho\big(\sigma(1-\b)\rho+(1-\rho)(1-\e)\big)}\,,$$
and let $\rho(\b,\e)$
be the only positive root of the equation 
$\phi(\b,\e,\rho)=1$, i.e.,
$$\rho(\b,\e)\,=\,
\frac{1}{2(\s-1)}\Big(
\s\b-1-\e+
\sqrt{(\s\b-1-\e)^2+4\e(\s-1)}\Big)\,.$$
As shown in \cite{Cerf}, we have:
\begin{align*}
1\leq i\leq j\leq\lfloor \rho(\beta,\e) m\rfloor\qquad
&\Longrightarrow\qquad \pi(i)\,\leq \pi(j)\,,\cr
\lfloor \rho(\beta,\e) m\rfloor
\leq i\leq j\leq m\qquad
&\Longrightarrow\qquad \pi(i)\,\geq \pi(j)\,.
\end{align*}
We have also the following result:
\begin{proposition}\label{ldpi}
Let $a \in \,]0,+\infty[\,$.
For $\rho\in[0,1]$, we have
$$
\lim_{
\genfrac{}{}{0pt}{1}{\ell,m\to\infty}
{q\to 0,\,
{\ell q} \to a}
}
\,\frac{1}{m}\ln\pi(\lfloor\rho m\rfloor)\,=\,
\int_0^\rho\ln \phi(
\exp(-a)
,0,s)\,ds\,.$$
\end{proposition}
A detailed proof is provided in \cite{Cerf}.
Let us define, for $\s\exa>1$,
$$\rho^*_0\,=\,
\rho\big(\exa,0\big)\,=\,
\frac{\displaystyle\sigma\exa-1}{\displaystyle\sigma-1}\,.
$$
The function
$$\rho\mapsto
\int_0^{\rho}\ln \phi(\exa,0,s)\,ds$$
is increasing on
$\,]0,\rho^*_0[\,$,
and decreasing on
$\,]\rho^*_0,1[\,$.
We also have
$$\lim_{
\genfrac{}{}{0pt}{1}{\ell,m\to\infty}
{q\to 0,\,
{\ell q} \to a}
}\,
\rho\big(M_H(0,0),M_H^\t(\t,0)\big)\,=\,
\rho^*_0\,.$$
Since all three assumptions are verified
we can apply proposition \ref{polexpbd}\
to the process $(Z^\t_t(0))_{t\geq 0}$.
Let $\d>0$ and define
$$V_0(\d)\,=\,
\Big\lbrace\,
i\in\zm
:
\Big|\frac{i}{m}-\rho^*_0\Big|<\d
\,\Big\rbrace\,.$$
We also define, for any subset $A\subset \zm$
the hitting time of $A$:
$$\tau(A)\,=\,
\inf\big\lbrace\,
t\geq 0
:
Z^\t_t(0)\in A
\,\big\rbrace\,.$$
\begin{corollary}\label{polexp0}
Let $\d>0$.
There exist positive real numbers
$\a_0,\a_0',\b_0,\b'_0$
(depending on $\d$),
such that for
$\ell,m$ big enough and
$q$ small enough,
\begin{align*}
\forall i\in\zm\qquad
&P\big(
\tau(V_0(\d))\geq m^{\a_0}
\,\big|\,
Z^\t_0(0)=i
\big)\,\leq\,
\exp(-\a'_0 m)\,,\\
\forall i\in V_0(\d)\qquad
&P\big(
\tau(V_0(2\d)^c)\leq \exp(\b_0 m)
\,\big|\,
Z^\t_0(0)=i
\big)\,\leq\,
\exp(-\b'_0 m)\,.
\end{align*}
\end{corollary}

\subsection{The inductive step}
We perform now the inductive step.
Let us fix an integer ${k\in\zk}$.
We suppose that the result of theorem \ref{polexp}
is true at rank $k-1$,
and we prove that it remains true at rank $k$.
We recall that, for $\d>0$,
the subset
$U_k(\d)\subset \cE_K$ is defined by
$$U_k(\d)\,=\,
\Big\lbrace
z\in\cE_K : \Big|
\frac{z_i}{m}-\rho^*_i
\Big|<\d,\quad
0\leq i\leq k
\Big\rbrace\,.$$
The hitting time of a subset $A\subset\cE_K$ 
is defined by
$$\tau(A)\,=\,
\inf\lbrace\, 
t\geq 0: Z_t^\t\in A
\,\rbrace\,.$$
We will prove that for $\d>0$, there exist positive real numbers
$\a_k,\a'_k,\b_k,\b'_k$
(depending on $\d,\d'$) 
such that for
$\ell,m$ large enough
and $q$ small enough,
\begin{align*}
\forall z\in\cE_K
&\qquad
P\big(\tau(U_k(\d))\geq m^{\a_k}
\,\big|\, 
Z^\t_0=z\big)\,\leq\,
\exp(-\a'_k m)\,,\\
\forall z\in U_k(\d)
&\qquad
P\big(\tau(U_k(2\d)^c)\leq \exp(\b_k m)
\,\big|\, 
Z^\t_0=z\big)\,\leq\,
\exp(-\b'_k m)\,.
\end{align*}
%
Let $\d'>0$. 
By the induction hypothesis,
the process $\Ztt$ spends most of its time inside the set $U_{k-1}(2\d')$.
Therefore, we will study the dynamics of the $k$th coordinate 
$\ztk$ when $Z^\t_t$ is in $U_{k-1}(2\d')$.
Conditionally on
$\smash{\big(Z^\t_t(0),\dots,
Z^\t_t(k-1)\big)_{t\geq 0}}$,
the process
$(Z^\t_t(k))_{t\geq 0}$
can be seen as a birth and death process
having time--dependent transition probabilities.
The classical formula for birth and death processes
cannot be applied directly to the process $\ztk$. 
Our goal is to get rid of this time dependence.
We will build a process
$\Zt$
whose conditional law given
$\smash{\big(Z^\t_t(0),\dots,
Z^\t_t(k-1)\big)_{t\geq 0}}$
is the same as the law of
$(Z^\t_t(k))_{t\geq 0}$.
We will then realize a coupling between the process
$\Zt$
and a pair of birth and death Markov chains,
a lower one
$\ztl$
and an upper one
$\ztu$.
The key point is that the transition probabilities of
$\ztl$ and $\ztu$
are not time dependent any more.
The coupling only works as long as the process
$(Z^\t_t)_{t\geq 0}$
is in the set $U_{k-1}(2\d')$.
Since
$\zt$
spends most of its time inside $U_{k-1}(2\d')$,
the coupling will allow us to obtain the desired estimates.
For $i\in\{\,0,\dots,K\,\}$,
we define 
a pair of maps
${\d_i,\g_i:[0,1]^k\times\zm\lra\zm}$
as follows.
Let $\rho=(\rho_0,\dots,\rho_{k-1})\in [0,1]^k$.

$\bullet$ We set $\d_m(\rho)=0$ and for $0\leq i<m$,
\begin{multline*}
\d_i(\rho)\,=\,
\frac{1-i/m}
{(\s-1)\rho_0+1}
\bigg(
\s \rho_0 M_H(0,k)+
\sum_{l=1}^{k-1}\rho_l M_H(l,k)\\
+\frac{i}{m}M_H(k,k)+
\Big(
1-\sum_{l=0}^{k-1}\rho_l-\frac{i}{m}
\Big)M_H(\t,k)
\bigg)\,.
\end{multline*}
$\bullet$ We set $\g_0(\rho)=0$ 
and for $0<i\leq m$,
\begin{multline*}
\g_i(\rho)\,=\,
\frac{i/m}
{(\s-1)\rho_0+1}
\bigg(
\s \rho_0 \big(1-M_H(0,k)\big)+
\sum_{l=1}^{k-1}\rho_l \big(1-M_H(l,k)\big)\\
+\frac{i}{m}\big(1-M_H(k,k)\big)+
\Big(
1-\sum_{l=0}^{k-1}\rho_l-\frac{i}{m}
\Big)\big(1-M_H(\t,k)\big)
\bigg)\,.
\end{multline*}
These maps allow us to express the transition probabilities of
$\ztk$
in the following way.
For $z=(z_0,\dots,z_k)\in\cE_K\setminus\{\,0\,\}$ and $i=z_k$,
\begin{align*}
P\big(
Z^\t_{t+1}(k)=i+1
\,\big|\,
Z^\t_t=z
\big)\,&=\,
\d_i\Big(
\frac{z_0}{m},\dots,\frac{z_{k-1}}{m}
\Big)\,,\\
P\big(
Z^\t_{t+1}(k)=i-1
\,\big|\,
Z^\t_t=z
\big)\,&=\,
\g_i\Big(
\frac{z_0}{m},\dots,\frac{z_{k-1}}{m}
\Big)\,.
\end{align*}
The process
$\ztk$
is not well suited to build a coupling,
therefore,
we build another process
$\Zt$
whose conditional law given
the trajectories
$\smash{\big(Z^\t_t(0),\dots,
Z^\t_t(k-1)\big)_{t\geq 0}}$
is the same as the law of $\ztk$.
We define a map
$$C:[0,1]^k\times\zm\times[0,1]\lra[0,1]$$
by setting for $\rho\in[0,1]^k$,
$i\in\zm$ and
$u\in [0,1]$,
$$C(\rho,i,u)\,=\,
i-1_{u<\g_i(\rho)}+1_{u>1-\d_i(\rho)}\,.$$
The map $C$ is defined so that if $U$ 
is a uniform random variable on $[0,1]$, 
then for
$\rho\in[0,1]^k$
and
$i\in\zm$,
$$P\big(C(\rho,i,U)=i+1\big)\,=\,
\d_i(\rho),\qquad 
P\big(C(\rho,i,U)=i-1\big)\,=\,
\g_i(\rho)\,.
$$
\begin{lemma}\label{monoC}
For $m$ large enough,
the map $C$
is non--decreasing with respect to the argument $i$:
\begin{multline*}
\forall \rho\in[0,1]^k\quad 
\forall i,j\in\zm\quad 
\forall u\in[0,1]\\
i\leq j\,\Lra\,
C(\rho,i,u)\leq C(\rho,j,u)\,.
\end{multline*}
\end{lemma}
\begin{proof}
It is enough to show that the result holds for $j=i+1$.
We have
\begin{multline*}
C(\rho,i+1,u)-C(\rho,i,u)\,=\\
1-\big(
1_{u<\g_{i+1}(\rho)}-
1_{u<\g_i(\rho)}
\big)+
\big(
1_{u>1-\d_{i+1}(\rho)}-
1_{u>1-\d_i(\rho)}
\big)\,.
\end{multline*}
This quantity is negative if and only if
$$\g_i(\rho)\,\leq\,u\,<\,\g_{i+1}(\rho)
\qquad\text{and}\qquad
1-\d_i(\rho)\,<\,u\,\leq\, 1-\d_{i+1}(\rho)\,.$$
This can only happen if
$\d_i(\rho)+\g_{i+1}(\rho)>1$. 
However, taking 
$y\in
[0,1]$ such that
$i=\lfloor ym\rfloor$,
we see that
\begin{multline*}
\lim_{m\ra\infty}\d_i(\rho)+\g_{i+1}(\rho)\,=\\
\frac{1}{(\s-1)\rho_0+1}\bigg(
\s \rho_0\Big(
(1-y)M_H(0,k)
+y\big(
1-M_H(0,k)
\big)
\Big)
\\+
\sum_{l=1}^{k-1}
\rho_l\Big(
(1-y)M_H(l,k)
+y\big(
1-M_H(l,k)
\big)
\Big)
\\+
y\Big(
(1-y)M_H(k,k)
+y\big(
1-M_H(k,k)
\big)
\Big)
\\+
\Big(
1-\sum_{l=0}^{k-1}\rho_l-y
\Big)
\Big(
(1-y)M_H(\t,k)
+y\big(
1-M_H(\t,k)
\big)
\Big)
\bigg)\,,
\end{multline*}
Yet, for $u\in\,]0,1[$ and $v\in[0,1]$ we have
$$(1-u)v+u(1-v)\,<\,
v+1-v\,=\,1\,.$$
Thus,
\begin{multline*}
\lim_{m\ra\infty}\d_i(\rho)+\g_{i+1}(\rho)\,<\\
\frac{1}{(\s-1)\rho_0+1}\bigg(
\s\rho_0+
\sum_{l=1}^{k-1}
\rho_l+
y+
\Big(
1-\sum_{l=0}^{k-1}\rho_l-y
\Big)
\bigg)\,=\,1\,,
\end{multline*}
as required.
\end{proof}
Let $(U_n)_{n\geq 1}$
be an i.i.d. sequence of uniform random variables on $[0,1]$.
We define the process $\Zt$
using the process $\zt$ and the sequence $(U_n)_{n\geq 1}$.
Let $z=(z_0,\dots,z_K)\in \cE_K\setminus\{0\}$ be the starting point of the process
$\zt$. We take $Z_0=z_k$ and
$$\forall n\geq 1\qquad
Z_n\,=\,
C\bigg(
\frac{Z^\t_{n-1}(0)}{m},
\dots,
\frac{Z^\t_{n-1}(k-1)}{m},
Z_{n-1},
U_n
\bigg)\,.$$
From this construction, conditionally on
$\smash{\big(Z^\t_t(0),\dots,
Z^\t_t(k-1)\big)_{t\geq 0}}$,
both processes $\Zt$ and
$\ztk$ have the same law.

Let us fix $\d'>0$.
We build next a lower 
birth and death Markov chain
$\ztl$
and an upper one
$\ztu$
in order to bound stochastically the process $\Zt$
when $\zt$ is in $U_{k-1}(2\d')$.
Let 
$W_{k-1}(\d')$ 
be the subset of $[0,1]^k$ 
given by
$$W_{k-1}(\d')\,=\,
\big\lbrace\,
\rho\in[0,1]^k:
|\rho_i-\rho^*_i|<\d',\ \,0\leq i<k
\,\big\rbrace\,.$$
In particular, we have
$$W_{k-1}(\d')\cap\frac{\Z^{k}}{m}\,=\,
U_{k-1}(\d')\,.$$
We define $\ztl$ to be a 
birth and death Markov chain
on the state space $\zm$,
having the following transition probabilities:
\begin{align*}
\d^L_i\,&=\,
\min\{\,
\d_i(\rho):
\rho\in W_{k-1}(2\d')
\,\}\,,
&&0\leq i\leq m\,,\\
\g^L_i\,&=\,
\max\{\,
\g_i(\rho):
\rho\in W_{k-1}(2\d')
\,\}\,,
&&0\leq i\leq m\,.
\end{align*}
Likewise, 
we define $\ztu$
to be a birth and death Markov chain
on the state space $\zm$
having the following transition probabilities:
\begin{align*}
\d^U_i\,&=\,
\max\{\,
\d_i(\rho):
\rho\in W_{k-1}(2\d')
\,\}
&&0\leq i\leq m\,,\\
\g^U_i\,&=\,
\min\{\,
\g_i(\rho):
\rho\in W_{k-1}(2\d')
\,\}
&&0\leq i\leq m\,.
\end{align*}
The processes
$\ztl$ and $\ztu$
are well defined,
since if $\rho_i$ and $\rho'_i$
are the points that
maximise
the functions $\g_i(\rho)$ and $\d_i(\rho)$ respectively, we have
\begin{align*}\d^L_i+\g^L_i\,&=\,
\d^L_i+\g_i(\rho_i)\,\leq\,
\d_i(\rho_i)+\g_i(\rho_i)\,\leq\,1\,,\\
\d^U_i+\g^U_i\,&=\,
\d_i(\rho'_i)+\g^U_i\,\leq\,
\d_i(\rho'_i)+\g_i(\rho'_i)\,\leq\,1\,.
\end{align*}
In order to couple these processes we
define the maps
$$C^L,C^U:\zm\times[0,1]\lra[0,1]$$
by setting for
$i\in\zm$ and $u\in[0,1]$,
\begin{align*}
C^L(i,u)\,&=\,
i-1_{u<\g^L_i}+
1_{u>1-\d^L_i}\,,\\
C^U(i,u)\,&=\,
i-1_{u<\g^U_i}+
1_{u>1-\d^U_i}\,.
\end{align*}
The maps $C^L$, $C^U$ are built
so that if $U$ is a uniform random variable on $[0,1]$, then
\begin{align*}
P\big(C^L(i,U)=i+1\big)\,=\,
\d^L_i\quad 
\text{ and }
\quad
P\big(C^L(i,U)=i-1\big)\,=\,
\g^L_i\,,\\
P\big(C^U(i,U)=i+1\big)\,=\,
\d^U_i\quad 
\text{ and }
\quad
P\big(C^U(i,U)=i-1\big)\,=\,
\g^U_i\,.
\end{align*}
The definition of the transition probabilities
$\d^L_i,\g^L_i,\d^U_i,\g^U_i$
implies that the map
$C^L$ is below the map $C$
and the map $C^U$
is above the map $C$, i.e., 
\begin{multline*}
\forall \rho\in W_{k-1}(2\d')\quad 
\forall i\in\zm\quad 
\forall u\in [0,1]\\
C^L(i,u)\,\leq\,
C(\rho,i,u)\,\leq\,
C^U(i,u)\,.
\end{multline*}
We define the processes
$\ztl$, $\ztu$
with the help of the same sequence $(U_n)_{n\geq 1}$ 
that was used to define $\Zt$.
Let $i\in\zm$ be the starting point of the processes.
We set $Z^L_0=Z^U_0=i$ and
$$\forall n\geq 1\qquad
Z^L_n\,=\,
C^L(Z^L_{n-1},U_n)\,,\quad
Z^U_n\,=\,
C^U(Z^U_{n-1},U_n)\,.$$
Let $\tau(U_{k-1}(2\d')^c)$ be the exit time from the set $U_{k-1}(2\d')$
for the process $\zt$:
$$\tau(U_{k-1}(2\d')^c)\,=\,
\inf\,\{\, t\geq 0 
:
Z^\t_t\not\in U_{k-1}(2\d') \,\}\,.$$
\begin{proposition}\label{domiZ}
Let $z=(z_0,\dots,z_K)\in U_{k-1}(2\d')$
be the starting point of the process $\Ztt$.
If
$Z^L_0=Z_0=Z^U_0=z_k$, 
then 
$$
\forall n \in [0,\tau(U_{k-1}(2\d')^c)]\qquad
Z^L_n\,\leq\,
Z_n\,\leq\,
Z^U_n\,.$$
\end{proposition}
\begin{proof}
We will show the inequality by induction on $n\in\mathbb N$.
For $n=0$ we have equality $Z^L_0=Z_0=Z^U_0$. 
Suppose that the inequality holds at time $n<\tau(U_{k-1}(2\d')^c)$, i.e.,
$Z_n^L\,\leq\,
Z_n\,\leq\,
Z^U_n$
and
$\smash{\big(Z^\t_n(0)/m,\dots,Z^\t_n(k-1)/m\big)}$
is in the set 
$W_{k-1}(2\d')$.
We then have
\begin{align*}
Z^L_{n+1}\,&=\,C^L(Z^L_n, U_{n+1}\big)
\,,\\
Z_{n+1}\,&=\,C\bigg(\frac{Z^\t_n(0)}{m},\dots,\frac{Z^\t_n(k-1)}{m},Z_n, U_{n+1}\bigg)
\,,\\
Z^U_{n+1}\,&=\,C^U\big(Z^U_n, U_{n+1}\big)
\,.
\end{align*}
Lemma \ref{monoC} and the induction hypothesis together imply that
\begin{multline*}
C\bigg(\frac{Z^\t_n(0)}{m},\dots,\frac{Z^\t_n(k-1)}{m},Z^L_n, U_{n+1}\bigg)\\
\leq\,C\bigg(\frac{Z^\t_n(0)}{m},\dots,\frac{Z^\t_n(k-1)}{m},Z_n, U_{n+1}\bigg)\,\leq\\
C\bigg(\frac{Z^\t_n(0)}{m},\dots,\frac{Z^\t_n(k-1)}{m},Z^U_n, U_{n+1}\bigg)\,.
\end{multline*}
Since the map
$C^L$ is below
$C$ and the map
$C^U$ is above
$C$, we have
\begin{align*}
C^L(Z^L_n, U_{n+1}\big)\,&\leq\,
C\bigg(\frac{Z^\t_n(0)}{m},\dots,\frac{Z^\t_n(k-1)}{m},Z^L_n, U_{n+1}\bigg)\,,\\
C^U\big(Z^U_n, U_{n+1}\big)\,&\geq\,
C\bigg(\frac{Z^\t_n(0)}{m},\dots,\frac{Z^\t_n(k-1)}{m},Z^U_n, U_{n+1}\bigg)\,.
\end{align*}
Combining the above inequalities we obtain
$Z^L_{n+1}\leq Z_{n+1}\leq Z^U_{n+1}$
and the induction step is completed.
\end{proof}
Let $\d>0$ and define
$$V_k(\d)\,=\,\Big\lbrace\,
i\in\zm:
\Big|
\frac{i}{m}-\rho^*_k
\Big|<\d
\,\Big\rbrace\,.
$$
We define the hitting time of a subset 
$A\subset\zm$ for the processes
$\ztl$, $\ztu$ as follows
\begin{align*}
\tau^L(A)\,&=\,
\inf\big\lbrace\,
t\geq 0:Z^L_t\in A
\,\big\rbrace\,,\\
\tau^U(A)\,&=\,
\inf\big\lbrace\,
t\geq 0:Z^U_t\in A
\,\big\rbrace\,.
\end{align*}
The following result will help to finish 
the proof of the induction step for theorem \ref{polexp}.
We recall that the definition of $\ztl$ and $\ztu$ depends on the parameter $\d'>0$. 
\begin{proposition}\label{polexpbdc}
Let $\d>0$.
There exist positive real numbers $\a,\a',\b,\b'$
(depending on $\d,\d'$)
such that for
$\ell,m$ large enough
and $q,\d'$ small enough:

$\bullet$ For all $i\in\zm$,
\begin{align*}
P\big(
\tau^L(V_k(\d))\geq m^{\a}
\,\big|\, 
Z^L_0=i
\big)\,&\leq\,
\exp(-\a' m)\,,\\
P\big(
\tau^U(V_k(\d))\geq m^{\a}
\,\big|\, 
Z^U_0=i
\big)\,&\leq\,
\exp(-\a' m)\,.
\end{align*}
$\bullet$ For all $i\in V_k(\d)$,
\begin{align*}
P\big(
\tau^L(V_k(2\d)^c)\leq\exp(\b m)
\,\big|\, 
Z^L_0=i
\big)\,&\leq\,
\exp(-\b' m)\,,\\
P\big(
\tau^U(V_k(2\d)^c)\leq\exp(\b m)
\,\big|\, 
Z^U_0=i
\big)\,&\leq\,
\exp(-\b' m)\,.
\end{align*}
\end{proposition}
We prove this proposition in the next section.
The proof is the same for both the lower and the upper
birth and death chain,
we will therefore show the result for the process $\ztl$ only.
We show now how to complete the inductive step with the help of 
this result.

Let $\d,\d'>0$ with $2\d'<\d$ 
and let $z^0\in\cE_K$ be the starting point of the process.
Thanks to the induction hypothesis,
there exist positive real numbers 
$\a_{k-1},\a'_{k-1},\b_{k-1},\b'_{k-1}$
(depending on $\d'$) 
such that for
$\ell,m$ large enough
and $q$ small enough,
\begin{align*}
\forall z\in\cE_K
&\quad\ 
P\big(\tau(U_{k-1}(\d'))\geq m^{\a_{k-1}}
\,\big|\, 
Z^\t_0=z\big)\,\leq\,
\exp(-\a'_{k-1} m)\,,\\
\forall z\in U_{k-1}(\d')
&\quad\ 
P\big(\tau(U_{k-1}(2\d')^c)\leq \exp(\b_{k-1} m)
\,\big|\, 
Z^\t_0=z\big)\,\leq\,
\exp(-\b'_{k-1} m)\,.
\end{align*}
Let $\a_k>\a_{k-1}$, we have
\begin{multline*}
P\big(
\tau(U_k(\d))\geq m^{\a_k}
\,\big|\, 
Z^\t_0=z^0
\big)\\=\,
P\big(
\tau(U_{k-1}(\d'))\geq m^{\a_{k-1}},
\tau(U_k(\d))\geq m^{\a_k}
\,\big|\, 
Z^\t_0=z^0\big)\\
+P\big(
\tau(U_{k-1}(\d'))< m^{\a_{k-1}},
\tau(U_k(\d))\geq m^{\a_k}
\,\big|\, 
Z^\t_0=z^0\big)\,.
\end{multline*}
By the induction hypothesis the first term in the sum is bounded above by
$\exp(-\a'_{k-1}m)$.
We use the Markov property to control the second term:
\begin{multline*}
P\big(
\tau(U_{k-1}(\d'))< m^{\a_{k-1}},
\tau(U_{k}(\d))\geq m^{\a_k}
\,\big|\, 
Z^\t_0=z^0\big)\\=\,
\sum_{\genfrac{}{}{0 pt}{1}{t<m^{\a_{k-1}}}{z\in U_{k-1}(\d')}}
P\big(
\tau(U_{k-1}(\d'))=t,
Z^\t_t=z,
\tau(U_{k}(\d))\geq m^{\a_k}
\,\big|\, 
Z^\t_0=z^0\big)\\=\,
\sum_{\genfrac{}{}{0 pt}{1}{t<m^{\a_{k-1}}}{z\in U_{k-1}(\d')}}
P\big(
\tau(U_{k-1}(\d'))=t,
Z^\t_t=z
\,\big|\, 
Z^\t_0=z^0\big)\\
\times
P\big(
\tau(U_{k}(\d))\geq m^{\a_k}-t
\,\big|\, 
Z^\t_0=z\big)\,.
\end{multline*} 
Let $m$ be large enough so that
$m^{\a_k}-m^{\a_{k-1}}<\exp(\b_{k-1}m)\,.$
For
$t<m^{\a_{k-1}}$,
\begin{multline*}
P\big(
\tau(U_{k}(\d))\geq m^{\a_k}-t
\,\big|\, 
Z^\t_0=z\big)\,\leq\,
P\big(
\tau(U_{k}(\d))\geq m^{\a_k}-m^{\a_{k-1}}
\,\big|\, 
Z^\t_0=z\big)\\=\,
P\big(
\tau(U_{k-1}(2\d')^c)\leq\exp(\b_{k-1}m),
\tau(U_{k}(\d))\geq m^{\a_k}-m^{\a_{k-1}}
\,\big|\, 
Z^\t_0=z\big)\\
+P\big(
\tau(U_{k-1}(2\d')^c)>\exp(\b_{k-1}m),
\tau(U_{k}(\d))\geq m^{\a_k}-m^{\a_{k-1}}
\,\big|\, 
Z^\t_0=z\big)\,.
\end{multline*}
By the induction hypothesis,
the first term in the sum is bounded above by
$\exp(-\b'_{k-1}m)$.
For the second term we have:
\begin{multline*}
P\big(
\tau(U_{k-1}(2\d')^c)>\exp(\b_{k-1}m),
\tau(U_{k}(\d))\geq m^{\a_k}-m^{\a_{k-1}}
\,\big|\, 
Z^\t_0=z\big)\\
\leq\,
P\big(
\tau(U_{k}(\d))\geq m^{\a_k}-m^{\a_{k-1}}
\,\big|\, 
\tau(U_{k-1}(2\d')^c)>\exp(\b_{k-1}m),
Z^\t_0=z\big)\,.
\end{multline*}
Since 
$\exp(\b_{k-1}m)>m^{\a_k}-m^{\a_{k-1}}$
and $2\d'<\d$,
conditionally on the event $\tau(U_{k-1}(2\d')^c)>\exp(\b_{k-1}m)$,
the event $\tau(U_{k}(\d))\geq m^{\a_k}-m^{\a_{k-1}}$
depends only on $Z^\t_t(k)$.
Moreover, by proposition~\ref{domiZ},
$$\forall t\in\{\, 0,\dots,\exp(\b_{k-1}m) \,\}\qquad
Z^L_t\,\leq\,Z^\t_t(k)\,\leq\,Z^U_t\,.$$
Therefore,
\begin{multline*}
P\big(
\tau(U_{k}(\d))\geq m^{\a_k}-m^{\a_{k-1}}
\,\big|\, 
\tau(U_{k-1}(2\d')^c)>\exp(\b_{k-1}m),
Z^\t_0=z\big)\\
\leq\,
P\big(
\tau^L(V_k(\d))\geq m^{\a_k}-m^{\a_{k-1}} 
\,|\,
Z^L_0=z_k
\big)\\+
P\big(
\tau^U(V_k(\d))\geq m^{\a_k}-m^{\a_{k-1}} 
\,|\,
Z^U_0=z_k
\big)\,.
\end{multline*}
Let $\a>0$ be given by proposition~\ref{polexpbdc}.
Choosing $\a_k$ large enough so that 
${m^{\a_k}-m^{\a_{k-1}}>m^\a}$,
this last expression is bounded by 
$2\exp(-\a' m)$ (by proposition \ref{polexpbdc}),
and this yields the desired bound for the hitting time of
$U_k(\d)$.

In order to show the bound on the exit time of
$U_k(2\d)$,
we argue in a similar way.
Let $z^0\in U_k(\d)$ be the starting point of the process.
Let $\b_{k-1}$ be given by the induction hypothesis
and let $\b_k>0$. 
We have
\begin{multline*}
P\big(
\tau(U_k(2\d)^c)\leq \exp(\b_k m)
\,\big|\, 
Z^\t_0=z^0
\big)\,=\\
P\big(
\tau(U_{k-1}(2\d)^c)\leq\exp(\b_{k-1}m),
\tau(U_k(2\d)^c)\leq \exp(\b_k m)
\,\big|\, 
Z^\t_0=z^0\big)\\
+P\big(
\tau(U_{k-1}(2\d)^c)> \exp(\b_{k-1}m),
\tau(U_k(2\d)^c)\leq \exp(\b_k m)
\,\big|\, 
Z^\t_0=z^0\big)\,.
\end{multline*}
By the induction hypothesis,
the first term in the sum is bounded above by
$\exp(-\b'_{k-1}m)$.
For the second term we have:
\begin{multline*}
P\big(\tau(U_{k-1}(2\d)^c)> \exp(\b_{k-1}m),
\tau(U_k(2\d)^c)\leq \exp(\b_k m)
\,\big|\, 
Z^\t_0=z^0\big)\\
\leq\,
P\big(
\tau(U_k(2\d)^c)\leq \exp(\b_k m)
\,\big|\, 
Z^\t_0=z^0,
\tau(U_{k-1}(2\d)^c)> \exp(\b_{k-1} m)\big)
\,.
\end{multline*}
Let $\b$ be given by proposition~\ref{polexpbdc}, and 
$\b_k>0$ such that $\b_k<\b_{k-1}\wedge\b$.
Then, conditionally on $\tau(U_{k-1}(2\d)^c)> \exp(\b_{k-1} m)$,
the event $\tau(U_k(\d))\leq \exp(\b_k m)$ 
only depends on $Z^\t_t(k)$.
Since 
$\tau(U_{k-1}(2\d)^c)> \exp(\b_{k-1} m)>\exp(\b_k m)$,
by proposition~\ref{domiZ} we have
$$\forall t \in \lbrace\, 0,\dots,\exp(\b_{k-1}m) \,\rbrace\qquad
Z^L_t\,\leq\,Z^\t_t(k)\,\leq\,Z^U_t\,.$$
Therefore,
\begin{multline*}
P\big(
\tau(U_k(2\d)^c)\leq \exp(\b_k m)
\,\big|\, 
Z^\t_0=z^0,
\tau(U_{k-1}(2\d)^c)> \exp(\b_{k-1} m)\big)
\\
\leq\,
P\big(
\tau^L(V_k(2\d)^c)\leq \exp(\b_k m)
\,|\,
Z^L_0=z^0_k
\big)\\
+
P\big(
\tau^U(V_k(2\d)^c)\leq \exp(\b_k m)
\,|\,
Z^U_0=z^0_k
\big)\,
\leq\,
2\exp(-\b' m)\,.
\end{multline*}
This completes the induction step.

\subsection{Dynamics of $\ztl$}
We study here the dynamics of the process
$\ztl$ in order to prove proposition~\ref{polexpbdc}.
First of all, we look for the points 
$\rho$ in $W_{k-1}(\d')$ 
that minimise and maximise the functions
$\d_i(\rho)$ and $\g_i(\rho)$.
Since we have 
$$\forall l\in\{\,1,\dots,k-1\,\}\qquad
M_H(l,k)\geq M_H(\t,k)\,,$$
the function
$\d_i(\rho_0,\dots,\rho_{k-1})$
is non--decreasing with respect to the variables
$\rho_1,\dots,\rho_{k-1}$.
Likewise,
$\g_i(\rho_0,\dots,\rho_{k-1})$
is non--increasing with respect to the variables
$\rho_1,\dots,\rho_{k-1}$.
Therefore, for all $i\in\zm$,
\begin{align*}
\d^L_i\,&=\,
\min_{\rho_0:|\rho_0-\rho^*_0|<\d'}
\d_i(\rho_0,\rho^*_1-\d',\dots,\rho^*_{k-1}-\d')\,,\\
\g^L_i\,&=\,
\max_{\rho_0:|\rho_0-\rho^*_0|<\d'}
\g_i(\rho_0,\rho^*_1-\d',\dots,\rho^*_{k-1}-\d')\,.
\end{align*}
Let us take the partial derivatives of $\d_i(\rho)$ and $\g_i(\rho)$
with respect to $\rho_0$:
\begin{multline*}
\frac
{\partial \d_i(\rho_0,\dots,\rho_{k-1})}
{\partial \rho_0}
\,=\,
\frac
{1-i/m}
{\big( (\s-1)\rho_0+1 \big)^2}
\bigg(
\s\Big(
M_H(0,k)-M_H(\t,k)
\Big)\\
-(\s-1)\sum_{l=1}^{k-1}\rho_l\Big(
M_H(l,k)-M_H(\t,k)
\Big)-
(\s-1)\frac{i}{m}\Big(
M_H(k,k)-M_H(\t,k)
\Big)
\bigg)\,,
\end{multline*}
\begin{multline*}
\frac
{\partial \g_i(\rho_0,\dots,\rho_{k-1})}
{\partial \rho_0}
\,=\,
\frac
{i/m}
{\big( (\s-1)\rho_0+1 \big)^2}
\bigg(
-\s \Big(
M_H(0,k)-M_H(\t,k)
\Big)+\\
(\s-1)\sum_{l=1}^{k-1}\rho_l\Big(
M_H(l,k)-M_H(\t,k)
\Big)+
(\s-1)\frac{i}{m}\Big(
M_H(k,k)-M_H(\t,k)
\Big)
\bigg)\,.
\end{multline*}
The sign of these partial derivatives does not depend on $\rho_0$.
In particular, for fixed
$\rho_1,\dots,\rho_{k-1}$,
the functions
$\d_i(\rho)$ and $\g_i(\rho)$
are monotone with respect to $\rho_0$.
Furthermore,
the partial derivatives above have opposite signs,
thus
\begin{multline*}
\frac
{\partial \d_i(\rho_0,\dots,\rho_{k-1})}
{\partial \rho_0}
\,=\,0\ \Longleftrightarrow\ 
\frac
{\partial \g_i(\rho_0,\dots,\rho_{k-1})}
{\partial \rho_0}\,=\,0\ \\
\Longleftrightarrow\
\s\Big(
M_H(0,k)-M_H(\t,k)
\Big)
-(\s-1)\sum_{l=1}^{k-1}\rho_l\Big(
M_H(l,k)-M_H(\t,k)
\Big)\\
-(\s-1)\frac{i}{m}\Big(
M_H(k,k)-M_H(\t,k)
\Big)\,=\,0\,.
\end{multline*}
We suppose that
$$\ell\to +\infty\,,\qquad m\to +\infty\,,\qquad q\to 0\,,$$
in such a way that
$${\ell q} \to a\in \,]0,+\infty[\,.$$
We have the following limits for the mutation probabilities:
$$
\lim_{
\genfrac{}{}{0pt}{1}{\ell\to\infty,\,
q\to 0}
{{\ell q} \to a}
}
\,
M_H(l,k)\,=\,
\begin{cases}
\displaystyle\quad\frac{a^{k-l}}{(k-l)!}\exa &\quad \text{if }\ l\leq k\,,\\
\quad 0 &\quad \text{if }\  l=\t\,.
\end{cases}
$$
For $\ell$ large enough and $q$ small enough,
the coefficient
$M_H(l,k)-M_H(\t,k)$
is positive.
Since the equation $\partial \d_i(\rho_0,\dots,\rho_{k-1})/\partial \rho_0 =0$
is linear with respect to $i$, 
we conclude that there exists an $i^*\in\zm$
(depending on $m,\rho^*_1,\dots,\rho^*_{k-1},\d'$)
such that:

$\bullet$ If $0\leq i\leq i^*$,
the function
$\rho_0\mapsto\d_i(\rho_0,\rho^*_1-\d',\dots,\rho^*_{k-1}-\d')$
is non--increasing,
the function
$\rho_0\mapsto\g_i(\rho_0,\rho^*_1-\d',\dots,\rho^*_{k-1}-\d')$
is non--decreasing,
and
\begin{align*}
\d^L_i\,&=\,
\d_i(\rho^*_0+\d',\rho^*_1-\d',\dots,\rho^*_{k-1}-\d')\,,\\
\g^L_i\,&=\,
\g_i(\rho^*_0+\d',\rho^*_1-\d',\dots,\rho^*_{k-1}-\d')\,.
\end{align*}

$\bullet$ If
$i^*< i\leq m$,
the function
$\rho_0\mapsto\d_i(\rho_0,\rho^*_1-\d',\dots,\rho^*_{k-1}-\d')$
is non--decreasing,
the function
$\rho_0\mapsto\g_i(\rho_0,\rho^*_1-\d',\dots,\rho^*_{k-1}-\d')$
is non--increasing,
and
\begin{align*}
\d^L_i\,&=\,
\d_i(\rho^*_0-\d',\rho^*_1-\d',\dots,\rho^*_{k-1}-\d')\,,\\
\g^L_i\,&=\,
\g_i(\rho^*_0-\d',\rho^*_1-\d',\dots,\rho^*_{k-1}-\d')\,.
\end{align*}

From the definition of $\d_i(\rho)$, $\g_i(\rho)$ we deduce that,
for $m\geq 2$,
\begin{align*}
\forall i\in\lbrace\,0\dots,m-1\,\rbrace\qquad
&\d_i^L\,\geq\,
\frac{\s(\rho^*_0-\d')M_H(0,k)}{m((\s-1)(\rho^*_0+\d')+1)}\,\geq\,
\frac{c}{m}\,,\\
\forall i\in\lbrace\,1\dots,m\,\rbrace\qquad
&\g_i^L\,\geq\,
\frac{\s(\rho^*_0-\d')(1-M_H(0,k))}{m((\s-1)(\rho^*_0+\d')+1)}\,\geq\,
\frac{c}{m}\,,
\end{align*}
where $c$ is a positive constant depending on $k,\d'$ but not on $m$.

We study now the products 
$\pi(i)$, which are defined by
$$\pi(0)\,=\,1\,,\qquad 
\pi(i)\,=\,\frac{\d_1\cdots\d_i}
{\g_1\cdots\g_i},\quad 1<i<m\,.$$
We study first the ratio
${\delta_i(\rho)}/{\gamma_i(\rho)}$.
For $0<i<m$, we have
$$
\frac{\delta_i(\rho)}{\gamma_i(\rho)}\,=\\
\phi\Big(M_H(0,k),\dots,M_H(k,k),M_H(\theta,k),\rho_0,\dots,\rho_{k-1},\frac{i}{m}\Big)\,,
$$
where the function
$\phi:\,
]0,1]^{k+1}\times[0,1[\,\times\,]0,1[^{k}\times\,]0,1[\,\lra \,]0,+\infty[\,$
is given by:

$\forall \b\in \,]0,1]^{k+1}\quad 
\forall \e\in [0,1[\,\quad
\forall \rho\in \,]0,1[^{k}\quad
\forall \eta\in \,]0,1[\,$
\begin{multline*}
\phi(\b,\e,\rho,\eta)\,=\\
\index{$\phi(\b,\e,\rho,\eta)$}
\frac{\displaystyle 
(1-\eta)
\bigg(
\sigma \rho_0\b_0+
\sum_{l=1}^{k-1}\rho_l\b_l+
\eta\b_k+
\Big(
1-\sum_{l=0}^{k-1}\rho_l-\eta
\Big)\e
\bigg)}
{\displaystyle\eta
\bigg(
\sigma \rho_0(1-\b_0)+
\sum_{l=1}^{k-1}\rho_l(1-\b_l)+
\eta(1-\b_k)+
\Big(
1-\sum_{l=0}^{k-1}\rho_l-\eta
\Big)(1-\e)
\bigg)
}\,.
\end{multline*}
In order to understand the behaviour of the products $\pi(i)$,
it is enough to know whether the value of
$\phi$ is larger or smaller than~$1$.
The equation $\phi(\b,\e,\rho,\eta)=1$ 
is linear with respect to $\eta$,
its only root being
$$\eta(\b,\e,\rho)\,=\,
\index{$\eta(\b,\e,\rho)$}
\frac
{\displaystyle
\sigma \rho_0\b_0+
\sum_{l=1}^{k-1}\rho_l\b_l+
\Big(
1-\sum_{l=0}^{k-1}\rho_l
\Big)\e
}
{(\s-1)\rho_0+1-\b_k+\e}\,.$$
Therefore,
\begin{align*}
\phi(\beta,\e,\rho,\eta)
>1 & \quad\text{ if }\quad \eta<\eta(\b,\e,\rho)\,,\cr
\phi(\b,\e,\rho,\eta)
<1 & \quad\text{ if }\quad \eta>\eta(\b,\e,\rho)\,.
\end{align*}
Moreover, the function $\phi(\b,\e,\rho,\eta)$
is continuous and non--decreasing with respect to the variables
$\b,\e$.
Take
$\psi:\,
]0,1]^{k+1}\times[0,1[\,\times\,]0,1[\,\lra \,]0,+\infty[\,$
to be the function defined by:
\begin{multline*}
\forall \b\in \,]0,1]^{k+1}\quad 
\forall \e\in [0,1[\,\quad
\forall \eta\in \,]0,1[\,\\
\psi(\b,\e,\eta)\,=\,
\begin{cases}
\ \,\phi(
\b,
\e,\rho^*_0+\d',\rho^*_1-\d',\dots,\rho^*_{k-1}-\d',\eta)& 
\ \text{if } \eta\leq i^*/m\,,\\
\ \,\phi(
\b,
\e,\rho^*_0-\d',\rho^*_1-\d',\dots,\rho^*_{k-1}-\d',\eta)& 
\ \text{if } \eta> i^*/m\,.
\end{cases}
\end{multline*}
We have the following large deviation estimates for the products $\pi(i)$.
\begin{proposition}\label{gdpizr}
Let $a \in \,]0,+\infty[\,$.
For
$\eta\in[0,1]$,
we have
$$
\lim_{
\genfrac{}{}{0pt}{1}{\ell,m\to\infty}
{q\to 0,\,
{\ell q} \to a}
}
\,\frac{1}{m}\ln\pi(
\lfloor\eta m\rfloor)\,=\,
\int_0^{\eta}\ln \psi\Big(
\exa\frac{a^k}{k!},\dots,\exa
,0,s\Big)\,ds\,.
$$
\end{proposition}
The proof is very similar to that of 
proposition 9.1 of \cite{Cerf}, so we omit it. 
Let us define
\begin{align*}
\rho^-\,=\,
\min\biggr\lbrace\,
\eta\Big(\exa\frac{a^k}{k!},&\dots,\exa,0,\rho^*_0+\d',\rho^*_1-\d',\dots,\rho^*_{k-1}-\d'\Big),\\
&\eta\Big(\exa\frac{a^k}{k!},\dots,\exa,0,\rho^*_0-\d',\rho^*_1-\d',\dots,\rho^*_{k-1}-\d'\Big)
\,\biggr\rbrace\,,
\\
\rho^+\,=\,
\max\biggr\lbrace\,
\eta\Big(\exa\frac{a^k}{k!},&\dots,\exa,0,\rho^*_0+\d',\rho^*_1-\d',\dots,\rho^*_{k-1}-\d'\Big),\\
&\eta\Big(\exa\frac{a^k}{k!},\dots,\exa,0,\rho^*_0-\d',\rho^*_1-\d',\dots,\rho^*_{k-1}-\d'\Big)
\,\biggr\rbrace\,.
\end{align*}
From the definitions, we see that
\begin{align*}
\psi\Big(
\exa\frac{a^k}{k!},\dots,\exa,
0,\eta
\Big)&>1\qquad\text{ for }\ \eta<\rho^-\,,\\
\psi\Big(
\exa\frac{a^k}{k!},\dots,\exa,
0,\eta
\Big)&<1\qquad\text{ for }\ \eta>\rho^+\,.
\end{align*} 
In particular, the function
$$\eta\mapsto
\int_0^{\eta}\ln \psi\Big(
\exa\frac{a^k}{k!},\dots,\exa,
0,s
\Big)\,ds$$
is non--decreasing on
$\,]0,\rho^-[\,$
and non--increasing on 
$\,]\rho^+,1[\,$.
Furthermore, when $\d'$ goes to 0, 
the points $\rho^-$ and $\rho^+$
converge to $\rho^*_k$:
$$\lim_{\d'\to 0}\,\rho^-\,=\,
\lim_{\d'\to 0}\,\rho^+\,=\,
\rho^*_k\,.$$
We also define
\begin{multline*}
\eta^-\,=\\
\min\Big\lbrace\,
\eta\big(M_H(0,k),\dots,M_H(k,k),M_H(\t,k),\rho^*_0+\d',\rho^*_1-\d',\dots,\rho^*_{k-1}-\d'\big),\\
\eta\big(M_H(0,k),\dots,M_H(k,k),M_H(\t,k),\rho^*_0-\d',\rho^*_1-\d',\dots,\rho^*_{k-1}-\d'\big)
\,\Big\rbrace\,,
\end{multline*}
\begin{multline*}
\eta^+\,=\\
\max\Big\lbrace\,
\eta\big(M_H(0,k),\dots,M_H(k,k),M_H(\t,k),\rho^*_0+\d',\rho^*_1-\d',\dots,\rho^*_{k-1}-\d'\big),\\
\eta\big(M_H(0,k),\dots,M_H(k,k),M_H(\t,k),\rho^*_0-\d',\rho^*_1-\d',\dots,\rho^*_{k-1}-\d'\big)
\,\Big\rbrace\,.
\end{multline*}
We then have
\begin{align*}
1\leq i\leq j\leq \eta^- m \qquad
&\Longrightarrow\qquad \pi(i)\,\leq \pi(j)\,,\cr
\eta^+ m
\leq i\leq j\leq
m
\qquad
&\Longrightarrow\qquad \pi(i)\,\geq \pi(j)\,,
\end{align*}
and the situation between $\eta^-m$ and $\eta^+m$ is somewhat more delicate.
Anyhow, when  $\ell,m\to\infty$, $q\to 0$ and $\ell q\to a$, we have
$$\lim_{
\genfrac{}{}{0pt}{1}{\ell,m\to\infty}
{q\to 0,\,
{\ell q} \to a}
} \eta^-\,=\,\rho^-\,,\qquad
\lim_{
\genfrac{}{}{0pt}{1}{\ell,m\to\infty}
{q\to 0,\,
{\ell q} \to a}
} \eta^+\,=\,\rho^+\,.$$
For $\d>0$, we set
$$V_k(\d)\,=\,
\Big\lbrace\,
i\in\zm:
\Big|
\frac{i}{m}-\rho^*_k
\Big|<\d
\,\Big\rbrace\,,$$ 
and we define the hitting time of a subset $A\subset\zm$ by
$$\tau^L(A)\,=\,
\inf\big\lbrace\,
t\geq 0 : Z^L_t\in A
\,\big\rbrace\,.$$
We recall that the definition of $\ztl$ depends on the parameter $\d'>0$.
The above results show that the birth and death Markov chain
$\ztl$
verifies assumptions 1,2,3 of section \ref{Bdchains},
we can therefore apply proposition \ref{polexpbd} to the process $\ztl$
and we obtain the following result.
\begin{corollary}\label{polexpX}
Let $\d>0$.
There exist positive real numbers 
$\a_k,\a_k',\b_k,\b'_k$
(depending on $\d,\d'$)
such that for 
$\ell,m$ large enough and
$q,\d'$ small enough,
\begin{align*}
\forall i\in\zm\qquad
&P\big(
\tau^L(V_k(\d))\geq m^{ \a_k}
\,\big|\,
Z^L_0=i
\big)\,\leq\,
\exp(-\a'_k m)\,,\\
\forall i\in V_k(\d)\qquad
&P\big(
\tau^L(V_k(2\d)^c)<\exp(\b_k m)
\,\big|\,
Z^L_0=i
\big)\,\leq\,
\exp(-\b_k' m)\,.
\end{align*}
\end{corollary}
Thus the estimates of proposition \ref{polexpbdc} 
for the lower process $\ztl$ are proved.

\subsection{Convergence}\label{Convergence}
In this section we will prove that when $\s\exa>1$,
the invariant probability measure
$\nu^\t$
converges to the Dirac mass at
$\rho^*$.
Let $a$ such that $\s\exa>1$.
Let $\rho^*$ be the point of $[0,1]^{K+1}$ given by:
$$
\forall k\geq 0\qquad
\rho^*_k\,=\,
(\s\exa-1)
\frac{a^k}{k!}
\sum_{i\geq 1}
\frac{i^k}{\s^i}\,.
$$
We consider the asymptotic regime
$$\displaylines{
\ell\to +\infty\,,\qquad m\to +\infty\,,\qquad q\to 0\,,\cr
{\ell q} \to a\,,
\qquad\frac{m}{\ell}\to\alpha\,.}$$
\begin{theorem}\label{conv}
For every continuous and increasing function
$f:\pml\ra\R$
such that $f(0)=0$,
we have
$$\lim_{
\genfrac{}{}{0pt}{1}{\ell,m\to\infty,\,
q\to 0
}
{{\ell q} \to a,\,
\frac{\scriptstyle m}{\scriptstyle \ell}\to\alpha
}
}\,
\int_{\textstyle\cE_K}
f\Big(
\frac{z_0+\cdots+z_K}{m} 
\Big)
d\nu^\t(z)\,=\,
f(\rho^*_0+\cdots+\rho^*_K)\,.
$$
\end{theorem}
\begin{proof}
Let $\d>0$ 
and let us define
$$
U_K(\d)\,=\,
\Big\lbrace\,
z\in\cE_K
:
\Big|
\frac{z_k}{m}-\rho^*_k
\Big|
<\d,\quad
0\leq k\leq K
\,\Big\rbrace\,.
$$
We define two sequences of stopping times 
$(T_n)_{n\geq0}$ and 
$(T^*_n)_{n\geq1}$ as follows.
Let
$T_0=0$ and set
\begin{align*}
&T^*_1 \,=\,\inf\,\big\{\,t\geq 0: 
Z^\t_t\in U_K(\d)
\,\big\}\,,\ \,
&&T_1 \,=\,\inf\,\big\{\,t\geq T^*_1: 
Z^\t_t\not\in U_K(2\d)
\,\big\}\,,
\cr
&\,\,\,\vdots
&&
\,\,\,\vdots
\cr
&T^*_k \,=\,\inf\,\big\{\,t\geq T_{k-1}: 
Z^\t_t\in U_K(\d)
\,\big\}\,,\ \,
&&T_k \,=\,\inf\,\big\{\,t\geq T^*_k: 
Z^\t_t\not\in U_K(2\d)
\,\big\}\,,
\cr
&\,\,\,\vdots
&&
\,\,\,\vdots
\end{align*}
The ergodic theorem for Markov chains implies that
$$
\int_{\textstyle\cE_K}
f\Big(
\frac{z_0+\cdots+z_K}{m} 
\Big)
d\nu^\t(z)\,=\,
\lim_{t\ra\infty}\,
\frac{1}{t}
E\bigg(
\sum_{i=0}^t
f\bigg(
\frac{Z^\t_i(0)+\cdots+Z^\t_i(K)}{m}
\bigg)
\bigg)\,.
$$
Let $t\geq 0$.
We decompose this last sum as follows:
\begin{multline*}
\sum_{i=0}^t
f\bigg(
\frac{Z^\t_i(0)+\cdots+Z^\t_i(K)}{m}
\bigg)\,=\,
\sum_{n\geq 1}\,
\sum_{i=T_{n-1}\wedge t}^{T^*_n\wedge t-1}\,
f\bigg(
\frac{Z^\t_i(0)+\cdots+Z^\t_i(K)}{m}
\bigg)
\\+\,
\sum_{n\geq 1}\,
\sum_{i=T^*_{n}\wedge t}^{T_n\wedge t-1}\,
f\bigg(
\frac{Z^\t_i(0)+\cdots+Z^\t_i(K)}{m}
\bigg)\,.
\end{multline*}
The function $f$ is continuous.
Let $\e>0$ 
and let us choose $\d$ 
small enough so that 
$$
\forall z\in U_K(\d)\qquad
\bigg|
f\Big(
\frac{z_0+\cdots+z_K}{m}
\Big)\,-\,
f(\rho^*_0+\cdots+\rho^*_K)
\bigg|\,<\,
\frac{\e}{2}\,.$$
We have then
\begin{multline*}
\bigg|
E\bigg(
\sum_{i=0}^t
f\bigg(
\frac{Z^\t_i(0)+\cdots+Z^\t_i(K)}{m}
\bigg)\bigg)\,-\,
t f(\rho^*_0+\cdots+\rho^*_K)
\bigg|\,\leq\\
\sum_{n\geq 1}
2f(1)
E\big(
T_n^*\wedge t-T_{n-1}\wedge t
\big)+
\frac{t\e}{2}\,.
\end{multline*}
Next we study the expression
$$\sum_{n\geq 1}
\big(
T_n^*\wedge t-T_{n-1}\wedge t
\big)\,.$$
Let us define
$$N(t)\,=\,
\max\big\lbrace\,
n\geq 0
:
T_n\leq t
\,\big\rbrace\,.$$
We can rewrite the previous sum as follows
$$
\sum_{n\geq 1}
\big(
T_n^*\wedge t-T_{n-1}\wedge t
\big)\,\leq\,
\sum_{n=1}^{N(t)}
\big(
T_n^*-T_{n-1}
\big)
+\big(
t-T_{N(t)}
\big)\,.
$$
We study now the random variable 
$N(t)$. 
Let
$n\in\N$, $b>0$ and $z\in\cE_K$. 
More precisely, we week estimates on the following probability: 
$$P\big(
N(n\exp(b m)/2)\geq n
\,\big|\,
Z^\t_0=z
\big)\,.$$
From the definition of $N(t)$, it follows that
$N(t)\geq n$ if and only if
$T_{n}\leq t$.
Thus
$$P\big(
N(n\exp(b m)/2)\geq n
\,\big|\,
Z^\t_0=z
\big)\,=\,
P\big(
T_{n}\leq n\exp(b m)/2
\,\big|\,
Z^\t_0=z
\big)\,.$$
Let us define for $i\geq 1$,
$$Y_i\,=\,T_i-T_{i-1}\,,\qquad
Y^*_i\,=\, T_i-T^*_i\,.$$
By theorem \ref{polexp},
there exist positive real numbers
$\b$ et $\b'$
such that, for all $i\geq 1$,
\begin{multline*}
P\big(
Y_i\leq \exp(\b m)
\,\big|\,
Z^\t_0=z
\big)\,\leq\,
P\big(
Y^*_i\leq \exp(\b m)
\,\big|\,
Z^\t_0=z
\big)\\
=\,
\sum_{z'\in U_K(\d)}
P\big(
Y^*_i\leq \exp(\b m)
\,\big|\,
Z^\t_0=z, Z^\t_{T^*_i}=z'
\big)P\big(
Z^\t_{T^*_i}=z'
\,\big|\,
Z^\t_0=z
\big)\\
=\,
\sum_{z'\in U_K(\d)}
P\big(
T_1\leq \exp(\b m)
\,\big|\,
Z^\t_0=z'
\big)P\big(
Z^\t_{T^*_i}=z'
\,\big|\,
Z^\t_0=z
\big)
\\
\leq\,
\exp(-\b' m)\,.
\end{multline*}
Let us define the following
Bernoulli random variables:
$$\forall i\geq 1,\qquad
\e_i\,=\,
1_{Y^*_i\leq \exp(\b m)}\,.
$$
Notice that 
$$T_n\,=\,
Y_1+\cdots+Y_n\,\geq\,
Y_1^*+\cdots+Y_n^*\,.$$
If
$T_{n}\leq n\exp(\b m)/2$,
then there exist  $n/2$ indices in $\{\,1,\dots,n\,\}$
such that $Y_i^*\leq \exp(\b m)$. Therefore,
$$T_{n}\,\leq\, \frac{1}{2}n\exp(\b m)\ \Lra\  \e_1+\cdots+\e_{n}\,\geq\, \frac{n}{2}\,.$$
Thus,
\begin{multline*}
P\Big(
N\big(n\exp(\b m)/2\big)
\geq n
\,\Big|\,
Z^\t_0=z
\Big)\,=\\
P\Big(
T_{n}
\leq n\exp(\b m)/2
\,\Big|\,
Z^\t_0=z
\Big)\,\leq\,
P\Big(
\e_1+\cdots+\e_{n}
\geq n/2
\Big)\,.
\end{multline*}
Let $\l\geq 0$, thanks to Chebyshev's exponential inequality we have
$$P\Big(
\e_1+\cdots+\e_{n}
\geq n/2
\Big)\,\leq\,
\exp\Big(-\l/2+\ln
E\big(
\exp(\l\e_1/n)\cdots\exp(\l\e_n/n)
\big)
\Big)\,.
$$
Since
$\e_1,\dots,\e_{n-1}$
are measurable with respect to
${\big(
Z^\t_t,\ 0\leq t\leq T^*_n
\big)}$,
\begin{multline*}
E\big(
\exp(\l\e_1/n)\cdots\exp(\l\e_n/n)
\big)\\
=\,
E\Big(
E\big(
\exp(\l\e_1/n)\cdots\exp(\l\e_n/n)
\,|\,
Z^\t_t,\ 0\leq t\leq T_n^*
\big)
\Big)\\
=\,
E\Big(
\exp(\l\e_1/n)\cdots\exp(\l\e_{n-1}/n)
E\big(
\exp(\l\e_n/n)
\,|\,
Z^\t_t,\ 0\leq t\leq T_n^*
\big)
\Big)\,.
\end{multline*}
Thanks to the strong Markov property,
the above conditional expectation can be rewritten as follows:
$$
E\big(
\exp(\l\e_n/n)
\,|\,
Z^\t_t,\ \, 0\leq t\leq T_n^*
\big)
\,=\,
E\big(
\exp(\l\e_1/n)
\,|\,
Z^\t_0=Z^\t_{\tau^*_n}
\big)\,.$$
Yet, for all $z'\in U_K(\d)$,
$$
E\big(
\exp(\l\e_1/n)
\,|\,
Z^\t_0=z'
\big)
\,\leq\,
\exp\Big(-\b'm+\frac{\l}{n}\Big)+1-\exp(-\b'm)\,.
$$
We iterate this procedure and we obtain
$$
E\big(
\exp(\l\e_1/n)\cdots\exp(\l\e_n/n)
\big)\,\leq\,
\bigg(
\exp\Big(-\b'm+\frac{\l}{n}\Big)+1-\exp(-\b'm)
\bigg)^n\,.
$$
The change of variables $\l\to n\l$ yields
\begin{multline*}
P\Big(
\e_1+\cdots+\e_{n}
\geq n/2
\Big)\\
\leq\,
\exp\bigg(
-n\Big(
\l/2-\ln\big(
\exp
(-\b'm+\l)+1-\exp(-\b'm)
\big)
\Big)
\bigg)
\,.
\end{multline*}
Let
$\Lambda^*(t)$
be the Cram\'er transform of the Bernoulli law with parameter ${p=\exp(-\b'm)}$:
$$
\Lambda^*(t)\,=\,
\sup_{\l\geq 0}\,
\Big(
\l t-
\ln
\big(
pe^\l+1-p
\big)
\Big)\,.
$$
Optimising the previous inequality over $\l\geq 0$,
we obtain
$$P\Big(
\e_1+\cdots+\e_{n}
\geq n/2
\Big)\,\leq\,
\exp\big(
-n\L^*(1/2)
\big)\,.$$
We can rewrite the function $\L^*(t)$ as
$$
\Lambda^*(t)\,=\,
t\ln\frac{t}{p}
+(1-t)\ln\frac{1-t}{1-p}
\,.$$
In our particular case, for $m$ large enough,
$$\Lambda^*(1/2)\,=\,
\frac{1}{2}\ln\frac{\exp(\b' m)}{2}
+\frac{1}{2}\ln\frac{\exp(\b' m)}{2(\exp(\b' m)-1)}
\,\geq\,
c(m)\,,$$
where
$c(m)$ is a positive constant depending on $m$
but not on $n$.
It follows that for $m$ large enough,
$$
\forall n\geq 1\qquad
P\Big(
N\big(n\exp(\b m)/2\big)
\geq n
\,\Big|\,
Z^\t_0=z
\Big)
\,\leq\,
\exp\big(
-nc(m)
\big)\,.
$$
Let $t\geq 0$.
We seek next an upper bound for the expectation
$$
E\bigg(
\sum_{k=1}^{N(t)}
\big(
T_k^*-T_{k-1}
\big)
+\big(
t-T_{N(t)}
\big)
\bigg)\,.
$$
The sum inside the parenthesis is at most $t$,
therefore, for $n\geq 1$
\begin{multline*}
E\bigg(
\sum_{k=1}^{N(t)}
\big(
T_k^*-T_{k-1}
\big)
+\big(
t-T_{N(t)}
\big)
\bigg)\,\leq\\
E\bigg(
\Big(
\sum_{k=1}^{N(t)}
\big(
T_k^*-T_{k-1}
\big)
+\big(
t-T_{N(t)}
\big)\Big)
1_{
N(t)<n}
\bigg)\,+\,
tP\big(
N(t)\geq n
\big)\,.
\end{multline*}
Let
$$n_t\,=\,
\min\,\Big\lbrace\,
n\in\N
:
t\leq
\frac{n\exp(\b m)}{2}
\,\Big\rbrace\,.$$
On one hand, the analysis of the random variable
$N(n\exp(\b m)/2)$
shows that taking $n=n_t$,
the second term is bounded by
$$
\frac{n_t}{2}e^{\b m}
P\big(
N(n_t e^{\b m}/2)\geq n_t
\big)\,\leq\,
\frac{n_t}{2}\exp(\b m-n_tc(m))\,,
$$
which goes to $0$
when $t$ goes to $\infty$.
On the other hand,
we can bound the first term 
thanks to theorem~\ref{polexp}:
\begin{multline*}
E\bigg(
\Big(
\sum_{k=1}^{N(t)}
\big(
T_k^*-T_{k-1}
\big)
+\big(
t-T_{N(t)}
\big)\Big)
1_{
N(t)<n_t}
\bigg)\\
\leq\,
E\bigg(
\sum_{k=1}^{n_t+1}
\big(
T_k^*-T_{k-1}
\big)
\bigg)\,=\,
\sum_{k=1}^{n_t+1}
E\big(
T_k^*-T_{k-1}
\big)\,\leq\,
(n_t+1)\frac{m^\a}{1-\exp(-\a'm)}\,,
\end{multline*}
where $\a,\a'>0$.
We combine the above inequalities, and we obtain for $m$ large enough 
and for all $t>0$,
\begin{multline*}
\frac{1}{t}E\bigg(
\sum_{n\geq 1}
\big(
T^*_n\wedge t
-T_{n-1}\wedge t
\big)
\bigg)\\
\leq\,
\frac{2}{(n_t-1)\exp(\b m)}
\bigg(
\frac{n_t}{2}\exp(\b m -n_tc(m))
+(n_t+1)\frac{m^\a}{1-\exp(-\a'm)}
\bigg)\,.
\end{multline*}
When $t$ goes to $\infty$ this expression goes to
$2m^\a/\exp(\b m)(1-\exp(-\a'm))$, which in turn goes to $0$ with $m$.
We deduce that, for
$m$ large enough,
there exists $t_m>0$
such that for all $t\geq t_m$,
$$\Bigg|
\frac{1}{t}
E\bigg(
\sum_{i=0}^t
f\Big(
\frac{Z^\t_i(0)+\cdots+Z^\t_i(K)}{m}
\Big)
\bigg)
-f(\rho^*_0+\cdots+\rho^*_K)
\Bigg|\,<\,\e\,.$$ 
Thus,
$$\lim_{
\genfrac{}{}{0pt}{1}{\ell,m\to\infty,\,
q\to 0
}
{{\ell q} \to a,\,
\frac{\scriptstyle m}{\scriptstyle \ell}\to\alpha
}
}\,
\int_{\textstyle\cE_K}
f\Big(
\frac{z_0+\cdots+z_K}{m} 
\Big)
d\nu^\t(z)\,=\,
f(\rho^*_0+\cdots+\rho^*_K)\,.
$$
\end{proof}

\section{The neutral phase}
Throughout this  section we take $a$ such that $\s\exa<1$.
Let $\mu_O$ be the invariant probability measure of the process $\Ot$
and let $\nu_K$ be the image measure of $\mu_O$ through the map
$$o\in\pml\longmapsto
\frac{1}{m}\big(
o(0)+\cdots+o(K)\big)\in[0,1]\,.$$
We will prove that when $\s\exa<1$,
the probability measure $\nu_K$ 
converges to the Dirac mass at 0.
In order to prove the convergence of $\nu_K$ to $\d_0$,
we slightly modify the occupancy process $\Ot$.
Let $\cW_K$ be the set of the occupancy distributions
having at least one individual in the Hamming classes
$0,\dots,K$, i.e.,
$$\cW_K\,=\,\big\{\,o\in\pml:o(0)+\cdots+o(K)\geq 1\,\big\}
\index{$\cW_K$}\,,
$$
and let $\cN_K$ be the set of the occupancy distributions
having no individuals in the Hamming classes
$0,\dots,K$, i.e.,
$$\cN_K
\,=\,\big\{\,o\in\pml:o(0)+\cdots +o(K)=0\,\big\}\,.
\index{$\cN_K$}
$$
Let us define the occupancy distribution $o_{\text{exit}}$
by:
$$\forall l\in\zl\qquad
o_{\text{exit}}(l)\,=\,\begin{cases}
\quad m & \quad \text{if }\ l=K+1\,,\\
\quad 0 & \quad \text{otherwise}\,.
\end{cases}$$
Let $\Phi_O$ be the coupling map
defined in section 7.1 of \cite{Cerf}.
We define a new coupling map $\Phi_\T$ 
by setting for $o\in\pml$ and $r\in\cR$,
\begin{equation*}
\Phi_\T(o,r)\,=\,
\begin{cases}
\quad 
\oe
&\text{if }o\in\cW_K \,\,{\text{ and }}\,\, 
\Phi_O(o,r)\in\cN_K\,, \\
\quad 
\Phi_{O}(o,r)
&\text{otherwise}\,.
\end{cases}
\end{equation*}
Since for all $o\in\cN_K$ we have $o\preceq\oe$, 
the map $\Phi_\T$ is above the map $\Phi_O$
in the following sense:
$$
\forall r\in \cR\quad 
\forall o\in\pml\qquad
\Phi_O(o,r)\,\preceq\,\Phi_\T(o,r)\,.
$$
Thus, we can build an upper process $\Tt$
with the help of the i.i.d. sequence of random vectors $(R_n)_{n\geq 0}$,
such that if the processes $\Ot$, $\Tt$
both start from the same occupancy distribution $o$,
then
$$\forall t\geq 0\qquad
O_t\,\preceq\, \T_t\,.$$
Let $\mu_\T$
be the invariant probability measure of the process $\Tt$.
We fix a non--decreasing function 
$f:[0,1]\lra\R$ such that $f(0)=0$.
We have
$$\int_{[0,1]} f\, d\nu_K\,\leq\,
\int_{\pml} f\Big(
\frac{o(0)+\cdots+o(K)}{m}
\Big)\,d\mu_\T(o)\,.$$
We apply now the renewal result of proposition~\ref{renewal}
to the process $\Tt$,
the set $\cW_K$, the function
${o\longmapsto f\big(\big(\ozk\big)/m\big)}$,
and the occupancy distribution $\oe$.
Set
\begin{align*}
\tau^*_K\,&=\,
\inf\lbrace\,
t\geq 0 : \T_t\in\cW_K
\,\rbrace\,,\\
\tau_K\,&=\,
\inf\lbrace\,
t\geq \tau^*_K : \T_t=\oe
\,\rbrace\,.\\
\end{align*}
We have then
\begin{multline*}
\int_{\textstyle\pml}
f\bigg( \frac{\ozk }{m} \bigg)
\,d\mu_\T(o)
\,=\\
\frac{
\displaystyle
E\bigg(\int_0^{\tau_K}
f\bigg( \frac{ 
\T_s(0)+\cdots+\T_s(K)
 }{m} \bigg)
\,ds\,\Big|\,
\T_0=\oe
\bigg)
}{
\displaystyle
E\big(\tau_K\,|\,
\T_0=\oe
\big)}\,.
\end{multline*}
Since $f(0)=0$, we have
$$\int_0^{\tau_K}
f\bigg( \frac{ 
\T_s(0)+\cdots+\T_s(K)
 }{m} \bigg)
\,ds\,=\,
\int_{\tau_K^*}^{\tau_K}
f\bigg( \frac{ 
\T_s(0)+\cdots+\T_s(K)
 }{m} \bigg)
\,ds\,.
$$
Moreover, since $f$ is non--decreasing,
\begin{multline*}
\int_{\textstyle\pml}
f\bigg( \frac{\ozk }{m} \bigg)
\,d\mu_\T(o)
\,\leq\\
\frac{
\displaystyle
f(1)
E\big(\tau_K-\tau_K^*\,|\,
\T_0=\oe
\big)
}{
\displaystyle
E\big(\tau_K^*\,|\,
\T_0=\oe
\big)
+
E\big(\tau_K-\tau_K^*\,|\,
\T_0=\oe
\big)
}\,.
\end{multline*}
Our aim is to show that the ratio of the right--hand side of this expression
converges to $0$.
We call $\tau^*_K$
the discovery time of the first $K+1$ Hamming classes,
and $\tau_K-\tau^*_K$
the persistence time of the first $K+1$ Hamming classes.
Let $o\in\cN_K$.
Starting from any occupancy distribution $o$,
the processes $\Ot$ and $\Tt$
both have the same dynamics until time $\tau_K$.
Thus, it is enough to estimate the expectations of $\tau^*_K$
and $\tau_K-\tau^*_K$
for the original occupancy process $\Ot$.
With a slight abuse of notation, we set
\begin{align*}
\tau^*_K\,&=\,
\inf\lbrace\,
t\geq 0 : O_t\in\cW_K
\,\rbrace\,,\\
\tau_K\,&=\,
\inf\lbrace\,
t\geq \tau^*_K : O_t=\oe
\,\rbrace\,.\\
\end{align*}
We deduce from the last inequality that
$$
\int_{[0,1]}
f
\,d\nu_K
\,\leq\,
\frac{
\displaystyle
f(1)
E\big(\tau_K-\tau_K^*\,|\,
O_0=\oe
\big)
}{
\displaystyle
E\big(\tau_K^*\,|\,
O_0=\oe
\big)
+
E\big(\tau_K-\tau_K^*\,|\,
O_0=\oe
\big)
}\,.
$$
We estimate next the expectations appearing in the above inequality.
\begin{proposition}\label{decouverte}
For all $o\in\cN_K$, we have
$$\lim_{
\genfrac{}{}{0pt}{1}{\ell,m\to\infty,\,
q\to 0
}
{{\ell q} \to a,\,
\frac{\scriptstyle m}{\scriptstyle \ell}\to\alpha
}
}\,
\frac{1}{\ell}
\ln
E\big(
\tau^*_K
\,\big|\,
O_0=o
\big)\,=\,
\ln\k\,.
$$
\end{proposition}
\begin{proof}
Let $\tau^*$
be the discovery time of the master sequence,
i.e.,
$$\tau^*\,=\,
\inf\big\lbrace\,
t\geq 0
:
O_t(0)\geq 1
\,\big\rbrace\,.$$
Obviously,
$\tau^*\geq\tau^*_K$.
Moreover, thanks to propositions 10.3 and 10.6 of \cite{Cerf},
we know that for any occupancy distribution
$o$
such that
$o(0)=0$, we have
$$\lim_{
\genfrac{}{}{0pt}{1}{\ell,m\to\infty,\,
q\to 0
}
{{\ell q} \to a,\,
\frac{\scriptstyle m}{\scriptstyle \ell}\to\alpha
}
}\,
\frac{1}{\ell}
\ln
E\big(
\tau^*
\,\big|\,
O_0=o
\big)\,=\,
\ln\k\,.
$$
We immediately deduce the upper bound for $\tau^*_K$:
$$
\forall o\in\cN_K\qquad
\limsup_{
\genfrac{}{}{0pt}{1}{\ell,m\to\infty,\,
q\to 0
}
{{\ell q} \to a,\,
\frac{\scriptstyle m}{\scriptstyle \ell}\to\alpha
}
}\,
\frac{1}{\ell}
\ln
E\big(
\tau^*_K
\,\big|\,
O_0=o
\big)\,\leq\,
\ln\k\,.
$$
We next prove the lower bound 
following the strategy proposed in \cite{Cerf} 
to estimate the discovery time of the master sequence.
We work with the distance process
$\Dt$
introduced in chapter 7 of \cite{Cerf}.
We set
\begin{align*}
\cN_K\,&=\,
\big\lbrace\,
d\in\zl^m
:
d(i)>K,\ 1\leq i\leq m
\,\big\rbrace\,,
\\[3pt]
\cW_K \,&=\,\big\lbrace\,
d\in\zl^m
:\exists\, i\in\{\, 1,\dots,m \,\} \text{ such that }
d(i)\leq K
\,\big\rbrace
\,.
\end{align*}
Let
$$\tau^*_K\,=\,
\inf\big\lbrace\,
t\geq 0
:
D_t\in\cW_K
\,\big\rbrace\,.$$
The law of the discovery time $\tk$
is the same for the distance process and for the occupancy process.
Let $b\in\zl$ and let us denote by
$(b)^m$ the column vector whose entries are all $b$:
$$(b)^m\,=\,\left(
\begin{matrix}
b\\
\vdots\\
b
\end{matrix}
\right)
\,.$$
On the sharp peak landscape
the distance process restricted to the set $\cN$
is monotone (corollary 8.6 of \cite{Cerf}).
Thus,
$$
\forall d\in\cN_K\qquad
E\big(
\tk
\,\big|\,
D_0=d
\big)\,\geq\,
E\big(
\tk
\,\big|\,
D_0=(K+1)^m
\big)\,.
$$
In order to estimate this last expectation
we use lemmas 11.4 and 11.5 of \cite{Cerf},
which we state next.
We denote by $\cB$ 
the binomial law with parameters 
$\ell$ and $1-1/\k$.
\begin{lemma}\label{binborn}
For $b\leq \ell/2$, we have
$$\frac{1}{\kappa^\ell}
\left(\frac{\ell}{2b}\right)^b\,\leq\,
\cB(b)\,\leq\,\frac{\ell^b}{\kappa^{\ell-b}}\,.$$
\end{lemma}
\begin{lemma}\label{bingd}
For $\rho\in[0,1]$, we have
$$\lim_{\ell\to\infty}
\,\frac{1}{\ell}\ln\cB(\lfloor \rho \ell\rfloor)\,=\,
-(1-\rho)\ln\big(\kappa(1-\rho)\big)
-\rho\ln
\frac{\kappa\rho}{\kappa-1}\,.
$$
\end{lemma}
The following result is a variation of lemma 10.15 of \cite{Cerf}.
The proof is similar and we omit it.
\begin{lemma}\label{premhit}
For 
$b\in\{\, K+1,\dots,\ell \,\}$, we have
$$\forall n\geq 0\qquad P\big(
\tau_K^*\leq n
\,|\,
D_{0}=(b)^m
\big)
\,\leq\,
nm
\frac{\cB(0)+\cdots+\cB(K)}{\cB(b)}\,.
$$
\end{lemma}
We show next the lower bound on the hitting time
$\tau^*_K$:
$$\liminf_{
\genfrac{}{}{0pt}{1}{\ell,m\to\infty,\,
q\to 0
}
{{\ell q} \to a,\,
\frac{\scriptstyle m}{\scriptstyle \ell}\to\alpha
}
}\,
\frac{1}{\ell}
\ln
E\big(
\tau^*_K
\,\big|\,
D_0=(K+1)^m
\big)\,\geq\,
\ln\k\,.
$$
Let $\cE$ be the event given by
$$\cE\,=\,\big\{\,
\forall n\leq m\ltq\quad\forall l\leq\ln\ell\quad
U_{n,l}>p/\kappa\,\big\}\,.$$
If the event $\cE$ happens, then until time
$m\ltq$, none of the mutation events in the process
$(D_n)_{n\geq 0}$ can create a chromosome in one of the classes
$0,\dots,K$. Indeed,
on $\cE$,
\begin{multline*}
\forall b\in\kl\quad
\forall n\leq m\ltq
\qquad\\[3pt]
\cMH(b,U_{n,1},\dots,U_{n,\ell})\,\geq\,
\cMH(K+1,U_{n,1},\dots,U_{n,\ell})\cr
\,\geq\,
K+1+\sum_{l=K+2}^\ell1_{U_{n,l}>1-p(1-1/\kappa)}\,\geq\,K+1\,.
\end{multline*}
Thus, on the event $\cE$, we have $\tau_K^*\geq
m\ltq$.
The probability of $\cE$ is
$$
P(\cE)\,=\,\Big(1-\frac{p}{\kappa}\Big)^{
m\ltq\ln\ell}\,.$$
Let $\e>0$.
Let us suppose that the process starts at
$(K+1)^m$
and let us estimate the probability
\begin{multline*}
P\big(\tau_K^*>
{\kappa^{\ell(1-\e)}}\big)\,\geq\,
P\big(\tau_K^*>
{\kappa^{\ell(1-\e)}},\,\cE\big)\cr
\,\geq\,
P\Big(\forall t\in\{\,m\ltq,\dots,
{\kappa^{\ell(1-\e)}}\,\}\quad
D_t\in\cN_K,\,\cE\Big)\cr
\,=\,\sum_{d\in\cN_K}
P\Big(\forall t\in\{\,m\ltq,\dots,
{\kappa^{\ell(1-\e)}}\,\}\quad
D_t\in\cN_K,\,
D_{m\ltq}=d,\,\cE\Big)\cr
\,\geq\,\sum_{d\geq(\ln \ell)^m}
P\Big(\forall t\in\{\,m\ltq,\dots,
{\kappa^{\ell(1-\e)}}\,\}\quad
D_t\in\cN_K\,|\,
D_{m\ltq}=d,\,\cE\Big)\cr
\hfill
\times
P( D_{m\ltq}=d
,\,\cE)\,.\qquad\qquad \big(\Sigma_1\big)
\end{multline*}
Using the Markov property we obtain
\begin{multline*}
P\Big(\forall t\in\{\,m\ltq,\dots,
{\kappa^{\ell(1-\e)}}\,\}\quad
D_t\in\cN_K\,|\,
D_{m\ltq}=d,\,\cE\Big)
\cr
\,=\,
P\Big(
\forall t\in\{\,0,\dots,
{\kappa^{\ell(1-\e)}}
-m\ltq
\,\}\quad
D_t\in\cN_K\,|\,
D_{0}=d\Big)\cr
\,=\,
P\Big(
\tau_K^*>
{\kappa^{\ell(1-\e)}}
-m\ltq
\,|\,
D_{0}=d\Big)
\,\geq\,
P\Big(
\tau_K^*>
{\kappa^{\ell(1-\e)}}
\,|\,
D_{0}=d\Big)
\,.
\end{multline*}
In the neutral case the distance process is monotone (corollary 8.6 of \cite{Cerf}).
Thus,
for all $d\geq(\ln \ell)^m$,
$$
P\big(
\tau_K^*>
{\kappa^{\ell(1-\e)}}
\,|\,
D_{0}=d\big)\,\geq\,
P\big(
\tau_K^*>
{\kappa^{\ell(1-\e)}}
\,|\,
D_{0}=
(\ln \ell)^m
\big)\,.
$$
Therefore, we can rewrite inequality~$(\Sigma_1)$ as follows:
\begin{multline*}
P\big(\tau_K^*>
{\kappa^{\ell(1-\e)}}\big)
\,\geq\,
\cr
P\big(
\tau_K^*>
{\kappa^{\ell(1-\e)}}
\,|\,
D_{0}=
(\ln \ell)^m
\big)
\,P\big( 
D_{m\ltq}\geq
(\ln \ell)^m
,\,\cE\big)\,.
\quad\,
\big(\circ\big)
\end{multline*}
Since the distance process is monotone, we have
\begin{multline*}
P\big( 
D_{m\ltq}\geq
(\ln \ell)^m
,\,\cE
\,|\,
D_0=(K+1)^m
\big)
\,\geq\\
P\big( 
D_{m\ltq}\geq
(\ln \ell)^m
,\,\cE
\,|\,
D_0=(1)^m
\big)\,.
\end{multline*}
We also have the following estimate (section 10.5 of \cite{Cerf}):
\begin{multline*}
\,P\big( 
D_{m\ltq}\geq
(\ln \ell)^m
,\,\cE
\,|\,
D_0=(1)^m
\big)\,\geq\qquad\qquad\qquad\big(\bigtriangledown\big)\\
\Big(
1-
m\exp\Big(-\frac{1}{3}(\ln\ell)^2\Big)\Big)
\,\Big(1-\frac{p}{\kappa}\Big)^{
m\ltq\ln\ell}
\,.
\end{multline*}
We estimate next
$$
P(
\tau_K^*>
{\kappa^{\ell(1-\e)}}
\,|\,
D_{0}=
(\ln \ell)^m
)\,.$$
Let $\e'>0$.
We have
\begin{multline*}
P\Big(
\tau_K^*>
{\kappa^{\ell(1-\e)}}
\,\big|\,
D_{0}=
(\ln \ell)^m
\Big)
\cr
\,=\,
P\Big(
\tau_K^*>m\ell^2 ,\,
D_t\in\cN_K\text{ for }
m\ell^2\leq t\leq
{\kappa^{\ell(1-\e)}}
\,\big|\,
D_{0}=
(\ln \ell)^m
\Big)\cr
\,=\,\sum_{d\in\cN_K}
P\Big(
\begin{matrix}
\tau_K^*>m\ell^2 ,\, D_{m\ell^2}=d\\
D_t\in\cN_K\text{ for }
m\ell^2\leq t\leq
{\kappa^{\ell(1-\e)}}
\end{matrix}
\,\Big|\,
D_{0}=
(\ln \ell)^m
\Big)\cr
\geq\,
\sum_{d\geq (\lk(1-\e'))^m}
P\Big(
D_t\in\cN_K\text{ for }
m\ell^2\leq t\leq
{\kappa^{\ell(1-\e)}}
\,\big|\,
\tau_K^*>m\ell^2 
,\,
D_{m\ell^2}=d
\Big)
\cr
\hfill \times
P\Big(
\tau_K^*>m\ell^2, \,
D_{m\ell^2}=d\,
\big|\,
D_{0}=
(\ln \ell)^m
\Big)\,.\qquad\qquad\big(\Sigma_3\big)
\end{multline*}
The Markov property and the monotonicity of the process $(D_t)_{t\geq 0}$ 
give for
${d\geq (\lk(1-\e'))^m}$,
\begin{multline*}
P\Big(
D_t\in\cN_K\text{ for }
m\ell^2\leq t\leq
{\kappa^{\ell(1-\e)}}
\,\big|\,
\tau_K^*>m\ell^2 
,\,
D_{m\ell^2}=d
\Big)
\cr
\,=\,
P\Big(
\forall t\in\{\,0,\dots,
{\kappa^{\ell(1-\e)}}
-m\ell^2
\,\}\quad
D_t\in\cN_K\,\big|\,
D_{0}=d\Big)\cr
\,=\,
P\Big(
\tau_K^*>
{\kappa^{\ell(1-\e)}}
-m\ell^2
\,\big|\,
D_{0}=d\Big)
\,\geq\,P\Big(
\tau_K^*>
{\kappa^{\ell(1-\e)}}
\,\big|\,
D_{0}=d\Big)
\cr
\,\geq\,
P\big(
\tau_K^*>
{\kappa^{\ell(1-\e)}}
\,\big|\,
D_{0}=
(\lk(1-\e'))^m
\big)\,.
\end{multline*}
Reporting back into the inequality~$(\Sigma_3)$,
\begin{multline*}
P\Big(\tau_K^*>
{\kappa^{\ell(1-\e)}}
\,\big|\,
D_{0}=(\ln\ell)^m\Big)
\,\geq\,
P\big(
\tau_K^*>
{\kappa^{\ell(1-\e)}}
\,\big|\,
D_{0}=
(\lk(1-\e'))^m
\big)
\cr\hfill\times
P\Big(
\tau_K^*>m\ell^2 
,\,
D_{m\ell^2}\geq
(\lk(1-\e'))^m
\,\big|\,
D_{0}=
(\ln \ell)^m
\Big)\,.\qquad\big(\heartsuit\big)
\end{multline*}
{\bf Estimation of
$P(
\tau_K^*>m\ell^2 
,\,
D_{m\ell^2}\geq
(\lk(1-\e'))^m
\,|\,
D_{0}=
(\ln \ell)^m
)$\,.}
We write
\begin{multline*}
P\Big(
\tau_K^*>m\ell^2 
,\,
D_{m\ell^2}\geq
(\lk(1-\e'))^m
\,\big|\,
D_{0}=
(\ln \ell)^m
\Big)\,\geq\,
\qquad\qquad\big(\natural\big)
\cr
P\Big(
D_{m\ell^2}\geq
(\lk(1-\e'))^m
\,\big|\,
D_{0}=
(\ln \ell)^m
\Big)-
P\Big(
\tau_K^*\leq m\ell^2 
\,\big|\,
D_{0}=
(\ln \ell)^m
\Big)\,.
\end{multline*}
We control the last term by applying lemma~\ref{premhit} with $n=m\ell ^2$
and $b=\ln\ell$:
$$P\big(
\tau_K^*\leq m\ell^2
\,|\,
D_{0}=(\ln\ell)^m
\big)
\,\leq\,
(m\ell) ^2 
\frac{\cB(0)+\cdots+\cB(K)}{\cB(\ln\ell)}\,.
$$
Using lemma~\ref{binborn} we get
$$\frac{\cB(0)+\cdots+\cB(K)}{\cB(\ln\ell)}\,
\leq\,
\frac{1-(\ell\k)^{K+1}}{1-\ell\k}
\Big(\frac{2\ln\ell}{\ell}\Big)^{\ln\ell}
\,.
$$
Thus,
$$P\big(
\tau^*\leq m\ell^2
\,|\,
D_{0}=(\ln l)^m
\big)\,
\leq\,
(m\ell) ^2 
\frac{1-(\ell\k)^{K+1}}{1-\ell\k}
\Big(\frac{2\ln\ell}{\ell}\Big)^{\ln\ell}
\,.
\qquad\big(\flat\big)
$$
An estimate from section 10.5 of \cite{Cerf} will help us control the other term:
there exists a constant
$c(\e')>0$
such that for $\ell$ large enough, we have
$$
P\Big(
D_{m\ell^2}\geq
(\lk(1-\e'))^m
\,\big|\,
D_{0}=
(\ln \ell)^m
\Big)\,\geq\,
1-m\exp{
\big(
-\frac{1}{2}c(\e')\ell\big)}
\,.
$$
This inequality together with estimates~$(\natural)$ and~$(\flat)$ give
\begin{multline*}
P\Big(
\tau_K^*>m\ell^2 
,\,
D_{m\ell^2}\geq
(\lk(1-\e'))^m
\,\big|\,
D_{0}=
(\ln \ell)^m
\Big)\,\geq\,\qquad\big(\clubsuit\big)\cr
1-m\exp{
\big(
-\frac{1}{2}c(\e')\ell\big)}
\,-\,
(m\ell) ^2 
\frac{1-(\ell\k)^{K+1}}{1-\ell\k}
\Big(\frac{2\ln\ell}{\ell}\Big)^{\ln\ell}
\,.
\end{multline*}
{\bf Estimation of
$P(
\tau_K^*>
{\kappa^{\ell(1-\e)}}
\,|\,
D_{0}=
(\lk(1-\e'))^m
)\,.$}
We use the inequality of lemma~\ref{premhit} with
$n={\kappa^{\ell(1-\e)}}$
and
$b=\lk(1-\e')$:
$$P\big(
\tau_K^*\leq
{\kappa^{\ell(1-\e)}}
\,\big|\,
D_{0}=
(\lk(1-\e'))^m
\big)\,\leq\,
{\kappa^{\ell(1-\e)}}m
\frac{\cB(0)+\cdots+\cB(K)}{\cB(
\lk(1-\e')
)}\,.
$$
For $\e'$ small enough, 
the large deviation estimates of lemma~\ref{bingd}
imply the existence of a constant $c(\e,\e')>0$ such that,
for $\ell$ large enough,
$$P\big(
\tau_K^*\leq
{\kappa^{\ell(1-\e)}}
\,\big|\,
D_{0}=
(\lk(1-\e'))^m
\big)\,\leq\,
\exp(-c(\e,\e')\ell)\,.\qquad\big(\spadesuit\big)$$
Plugging the inequalities $(\clubsuit)$ and $(\spadesuit)$
into $(\heartsuit)$ we obtain
$$\displaylines{
P\Big(\tau_K^*>
{\kappa^{\ell(1-\e)}}
\,\big|\,
D_0=(\ln\ell)^m 
\Big)
\,\geq\,
\Big(1-\exp(-c(\e,\e')\ell)\Big)
\qquad\qquad\big(\bigtriangleup\big)\hfill
\cr\hfill
\bigg(
1-m\exp{
\big(
-\frac{1}{2}c(\e')\ell\big)}
\,-\,
(m\ell) ^2 
\frac{1-(\ell\k)^{K+1}}{1-\ell\k}
\Big(\frac{2\ln\ell}{\ell}\Big)^{\ln\ell}
\bigg)
\,.
}$$
We now use the inequalities $(\circ)$, $(\bigtriangledown)$, $(\bigtriangleup)$
to conclude that, for $\ell$ large enough,
\begin{multline*}
P\Big(\tau_K^*>
{\kappa^{\ell(1-\e)}}
\,\big|\,
D_0=(K+1)^m 
\Big)
\,\geq\\
\Big(
1-
m\exp\Big(-\frac{1}{3}(\ln\ell)^2\Big)\Big)
\,\Big(1-\frac{p}{\kappa}\Big)^{
m\ltq\ln\ell}
\Big(1-\exp(-c(\e,\e')\ell)\Big)\\
\times\bigg(
1-m\exp{
\big(
-\frac{1}{2}c(\e')\ell\big)}
\,-\,
(m\ell) ^2 
\frac{1-(\ell\k)^{K+1}}{1-\ell\k}
\Big(\frac{2\ln\ell}{\ell}\Big)^{\ln\ell}
\bigg)
\,.
\end{multline*}
Moreover, thanks to the Markov inequality,
$$
E\Big(\tau^*_K\,|\,
D_0=(K+1)^m 
\Big)
\,\geq\,
{\kappa^{\ell(1-\e)}}\,
P\Big(\tau_K^*\geq
{\kappa^{\ell(1-\e)}}
\,\big|\,
D_0=(K+1)^m 
\Big)\,.$$
We finally deduce that
$$\liminf_{
\genfrac{}{}{0pt}{1}{\ell,m\to\infty,\,
q\to 0
}
{{\ell q} \to a,\,
\frac{\scriptstyle m}{\scriptstyle \ell}\to\alpha
}
}
\,\frac{1}{\ell}\ln
E\Big(\tau^*_K\,|\,
D_0=(K+1)^m 
\Big)
\,\geq\,(1-\e)\ln\kappa\,.$$
Sending $\e$ to $0$ gives the desired lower bound.
\end{proof}

We estimate next
$E\big(
\tau_K-\tau^*_K
\,|\,
O_0=\oe
\big)$.
Let 
$\phi:\,]0,+\infty[
\,\lra [0,+\infty]$
be the function defined in \cite{Cerf}
by setting
$\phi(a)=0$
if $a\geq \ln\s$
and 
$$
\forall a<\ln\s\qquad
\phi(a)\,=\,
\frac
{ \displaystyle \sigma(1-e^{-a})
\ln\frac{\displaystyle\sigma(1-e^{-a})}{\displaystyle\sigma-1}
+\ln(\sigma e^{-a})}
{
\displaystyle (1-\sigma(1-e^{-a}))
 }\,.
$$
\begin{proposition}\label{persistence}
We have
$$
\lim_{
\genfrac{}{}{0pt}{1}{\ell,m\to\infty,\,
q\to 0
}
{{\ell q} \to a,\,
\frac{\scriptstyle m}{\scriptstyle \ell}\to\alpha
}
}
\frac{1}{m}\ln
E\big(
\tau_K-\tau^*_K
\,|\,
O_0=\oe
\big)\,=\,
\phi(a)\,.
$$
\end{proposition}
\begin{proof}
For any subset $E\subset\pml$
we define the hitting time of $E$ by
$$\tau(E)\,=\,
\inf\big\lbrace\,
t\geq 0
:
O_t\in E
\,\big\rbrace\,.$$
Let us define also the following occupancy distributions: 
$$
o_1\,=\,(1,m-1,0,\dots,0)\,,\qquad\qquad
o_2\,=\,(1,0,\dots,0,m-1)\,,$$
$$\forall l\in\zl\qquad 
o_3(l)\,=\,\begin{cases}
\quad 1 & \quad \text{ if }\ l=K\,,\\
\quad m-1 & \quad \text{ if }\ l=\ell\,,\\
\quad 0 & \quad \text{ otherwise}\,.
\end{cases}
$$
Thanks to the monotonicity of the process $\Ot$ we have
$$
E\big(
\tau(\cN_K)
\,|\,
O_0=o_3
\big)
\,\leq\,
E\big(
\tau_K-\tau^*_K
\,|\,
O_0=\oe
\big)\,\leq\,
E\big(
\tau(\cN_K)
\,|\,
O_0=o_1
\big)\,.
$$
We prove first the lower bound, i.e.,
$$
\liminf_{
\genfrac{}{}{0pt}{1}{\ell,m\to\infty,\,
q\to 0
}
{{\ell q} \to a,\,
\frac{\scriptstyle m}{\scriptstyle \ell}\to\alpha
}
}
\frac{1}{m}\ln
E\big(
\tau(\cN_K)
\,|\,
O_0=o_3
\big)
\,\geq\,
\phi(a)\,.
$$
Indeed, we have
\begin{multline*}
E\big(
\tau(\cN_K)
\,|\,
O_0=o_3
\big)
\,\geq\\
E\big(
\tau(\cN_K)
\,|\,
O_0=o_3,
O_1=o_2
\big)P\big(
O_1=o_2
\,|\,
O_0=o_3
\big)\,.
\end{multline*}
On one hand, thanks to the Markov property and the monotonicity of the process
$$
E\big(
\tau(\cN_K)
\,|\,
O_0=o_3,
O_1=o_2
\big)
\,\geq\,
1+
E\big(
\tau(\cN_K)
\,|\,
O_0=o_2
\big)\,.
$$
Since
$
\tau(\cN_K)
\geq 
\tau(\cN)
$,
thanks to the estimate for $\tau(\cN)$ (corollary 9.2 of \cite{Cerf}), we obtain
\begin{multline*}
\liminf_{
\genfrac{}{}{0pt}{1}{\ell,m\to\infty,\,
q\to 0
}
{{\ell q} \to a,\,
\frac{\scriptstyle m}{\scriptstyle \ell}\to\alpha
}
}
\frac{1}{m}\ln
E\big(
\tau(\cN_K)
\,|\,
O_0=o_2
\big)
\,\geq\\
\liminf_{
\genfrac{}{}{0pt}{1}{\ell,m\to\infty,\,
q\to 0
}
{{\ell q} \to a,\,
\frac{\scriptstyle m}{\scriptstyle \ell}\to\alpha
}
}
\frac{1}{m}\ln
E\big(
\tau(\cN)
\,|\,
O_0=o_2
\big)
\,\geq\,
\phi(a)\,.
\end{multline*}
On the other hand, we have
\begin{multline*}
P\big(
O_1=o_2
\,|\,
O_0=o_3
\big)\,=\,
\frac{1}{m^2}\big(
M_H(K,0)+
(m-1)M_H(\ell,0)
\big)\\
\geq\,
\frac{1}{m^2}M_H(K,0)\,=\,
\frac{1}{m^2}
\Big(
1-p\big(
1-\frac{1}{\k}
\big)
\Big)^{\ell-K}
\Big(
\frac{p}{\k}
\Big)^K\,.
\end{multline*}
Thus,
$$
\liminf_{
\genfrac{}{}{0pt}{1}{\ell,m\to\infty,\,
q\to 0
}
{{\ell q} \to a,\,
\frac{\scriptstyle m}{\scriptstyle \ell}\to\alpha
}
}
\frac{1}{m}\ln
P\big(
O_1=o_2
\,|\,
O_0=o_3
\big)\,\geq\,0\,.
$$
The above inequalities give the desired lower bound:
$$
\liminf_{
\genfrac{}{}{0pt}{1}{\ell,m\to\infty,\,
q\to 0
}
{{\ell q} \to a,\,
\frac{\scriptstyle m}{\scriptstyle \ell}\to\alpha
}
}
\frac{1}{m}\ln
E\big(
\tau(\cN_K)
\,|\,
O_0=o_3
\big)
\,\geq\,
\phi(a)\,.
$$
We prove next the upper bound:
$$
\limsup_{
\genfrac{}{}{0pt}{1}{\ell,m\to\infty,\,
q\to 0
}
{{\ell q} \to a,\,
\frac{\scriptstyle m}{\scriptstyle \ell}\to\alpha
}
}
\frac{1}{m}\ln
E\big(
\tau(\cN_K)
\,|\,
O_0=o_1
\big)
\,\leq\,
\phi(a)\,.
$$
To alleviate the notation we denote by
$P_o$ and $E_o$
the probability and the expectation for the process
$\Ot$
starting from the occupancy distribution~$o$.
Let $o'$ be the occupancy distribution given by
$o'=(0,m,0,\dots,0)$.
We have
$$
E_{o_1}\big(
\tau(\cN_K)
\big)
\,=\,
E_{o_1}\big(
\tau(\cN)
\big)+
E_{o_1}\big(
\tau(\cN_K)-\tau(\cN)
\big)
\,.
$$
Yet, thanks to the strong Markov property
and the monotonicity of the process, and since 
$\tau(\cN_K)\geq\tau(\cN)$,
\begin{multline*}
E_{o_1}\big(
\tau(\cN_K)-\tau(\cN)
\big)\,=\,
\sum_{o\in\cN}
E_{o}\big(
\tau(\cN_K)
\big)
P_{o_1}\big(
O_{\tau(\cN)}=o
\big)
\\
\leq\,
\sum_{o\in\cN}
E_{o'}\big(
\tau(\cN_K)
\big)
P_{o_1}\big(
O_{\tau(\cN)}=o
\big)\,=\,
E_{o'}\big(
\tau(\cN_K)
\big)\,.
\end{multline*}
We develop this last expectation as follows:
$$
E_{o'}\big(
\tau(\cN_K)
\big)\,=\,
E_{o'}\big(
\tau(\cW\cup\cN_K)
\big)+
E_{o'}\big(
\tau(\cN_K)-\tau(\cW\cup\cN_K)
\big)\,.$$
Yet,
\begin{multline*}
E_{o'}\big(
\tau(\cN_K)-\tau(\cW\cup\cN_K)
\big)\\
=\,
\sum_{o\in\cW}
E_{o}\big(
\tau(\cN_K)
\big)
P_{o'}\big(O_{\tau(\cW\cup\cN_K)}=o
\big)
\\
\leq\,
\sum_{o\in\cW}
E_{o_1}\big(
\tau(\cN_K)
\big)
P_{o'}\big(O_{\tau(\cW\cup\cN_K)}=o
\big)
\\
=\,
E_{o_1}\big(
\tau(\cN_K)
\big)
P_{o'}\big(
\tau(\cW)<\tau(\cN_K)\big)
\,.
\end{multline*}
Thus,
$$ 
E_{o'}\big(
\tau(\cN_K)
\big)
\,\leq\,
E_{o'}\big(
\tau(\cW\cup\cN_K)
\big)+
E_{o_1}\big(
\tau(\cN_K)
\big)
P_{o'}\big(
\tau(\cW)<\tau(\cN_K)
\big)\,.
$$
Therefore,
$$E_{o_1}\big(
\tau(\cN_K)
\big)
\,\leq\,
\frac{1}{
\displaystyle
P_{o'}\big(
\tau(\cN_K)<\tau(\cW)
\big)
}
\Big(
E_{o_1}\big(
\tau(\cN)
\big)+
E_{o'}\big(
\tau(\cW\cup\cN_K)
\big)
\Big)\,.
$$
We estimate next the three terms 
appearing on the right--hand side of this formula.
We control the first expectation with the help of the estimate
for the persistence time of the master sequence (corollary 9.2 of \cite{Cerf}) 
and we obtain:
$$
\limsup_{
\genfrac{}{}{0pt}{1}{\ell,m\to\infty,\,
q\to 0
}
{{\ell q} \to a,\,
\frac{\scriptstyle m}{\scriptstyle \ell}\to\alpha
}
}
\frac{1}{m}\ln
E_{o_1}\big(
\tau(\cN)
\big)
\,\leq\,
\phi(a)\,.
$$
We bound from below the probability in the denominator 
using the estimates on the discovery time $\tau^*_K$.
On the event
$$\cE\,=\,\big\{\,
\forall n\leq m\ltq\quad\forall l\leq\ln\ell\quad
U_{n,l}>p/\kappa\,\big\}\,,$$
if
$D_{m\ltq}\geq
(\ln \ell)^m$,
we have
$\tau(\cN_K)<\tau(\cW)$. 
Therefore, using $\big(\bigtriangledown\big)$,
\begin{multline*}
P_{o'}\big(
\tau(\cN_K)<\tau(\cW)
\big)\,\geq
\,P\big( 
D_{m\ltq}\geq
(\ln \ell)^m
,\,\cE
\,|\,
D_0=(1)^m
\big)\,\geq\\
\Big(
1-
m\exp\Big(-\frac{1}{3}(\ln\ell)^2\Big)\Big)
\,\Big(1-\frac{p}{\kappa}\Big)^{
m\ltq\ln\ell}\,.
\end{multline*}
It remains to estimate the expectation
$E_{o'}\big(
\tau(\cW\cup\cN_K)
\big)$.
We estimate first, for
$n\geq 0$,
the probability
$
P_{o'}\big(
\tau(\cW\cup\cN_K)>n
\big)
$.
We have
\begin{multline*}
P_{o'}\big(
\tau(\cW\cup\cN_K)>n
\big)
\,=\\
\sum_{o\in\cW_K\setminus\cW}
P_{o'}\big(
\tau(\cW\cup\cN_K)>n,
O_{n-1}=o,
\tau(\cW\cup\cN_K)>n-1
\big)
\\
=\,
\sum_{o\in\cW_K\setminus\cW}
P_{o'}\big(
\tau(\cW\cup\cN_K)>n
\,|\,
O_{n-1}=o,
\tau(\cW\cup\cN_K)>n-1
\big)\\
\times P_{o'}\big(
O_{n-1}=o,
\tau(\cW\cup\cN_K)>n-1
\big)
\,.
\end{multline*}
Thanks to the Markov property
and the monotonicity of the process,
we have for all
$o\in\cW_K\setminus\cW$,
\begin{multline*}
P_{o'}\big(
\tau(\cW\cup\cN_K)>n
\,|\,
O_{n-1}=o,
\tau(\cW\cup\cN_K)>n-1
\big)
\,=\\
P_o\big(
\tau(\cW\cup\cN_K)>1
\big)
\,=\,
1-P_{o}\big(
\tau(\cW\cup\cN_K)=1
\big)
\\
\leq\,
1-P_{o}\big(
\tau(\cW)=1
\big)
\,\leq\,
1-P\big(
\tau(\cW)=1
\,|\,
O_0=o_3
\big)\\
\leq\,
1-\frac{M_H(K,0)}{m^2}\,.
\end{multline*}
Thus,
$$
P_{o'}\big(
\tau(\cW\cup\cN_K)>n
\big)
\,\leq\,
\Big(
1-\frac{M_H(K,0)}{m^2}
\Big)
P_{o'}\big(
\tau(\cW\cup\cN_K)>n-1
\big)\,.
$$
We iterate this inequality and we obtain
$$
P_{o'}\big(
\tau(\cW\cup\cN_K)>n
\big)
\,\leq\,
\Big(
1-\frac{M_H(K,0)}{m^2}
\Big)^n\,.
$$
Finally,
\begin{multline*}
E_{o'}\big(
\tau(\cW\cup\cN_K)
\big)
\,=\,
\sum_{n\geq 0}
P_{o'}\big(
\tau(\cW\cup\cN_K)>n
\big)
\\
\leq\,
\sum_{n\geq 0}
\Big(
1-\frac{M_H(K,0)}{m^2}
\Big)^n
\,=\,
\frac{m^2}{M_H(K,0)}
\,=\,
\frac{m^2}
{\displaystyle
\Big(
1-p\big(
1-\frac{1}{\k}
\big)
\Big)^{\ell-K}
\Big(
\frac{p}{\k}
\Big)^K
}
\,.
\end{multline*}
We put together the above inequalities and we obtain the desired upper bound.
\end{proof}

Let 
$\phi:\,]0,+\infty[
\,\ra [0,+\infty]$
be the function defined by
$\phi(a)=0$
if $a\geq \ln\s$
and 
$$
\forall a<\ln\s\qquad
\phi(a)\,=\,
\frac
{ \displaystyle \sigma(1-e^{-a})
\ln\frac{\displaystyle\sigma(1-e^{-a})}{\displaystyle\sigma-1}
+\ln(\sigma e^{-a})}
{
\displaystyle (1-\sigma(1-e^{-a}))
 }\,.
$$
From the estimates obtained in this section
we conclude that for
${\alpha\in[0,+\infty[}$ or $\alpha=+\infty$,
$$\lim_{
\genfrac{}{}{0pt}{1}{\ell,m\to\infty,\,
q\to 0
}
{{\ell q} \to a,\,
\frac{\scriptstyle m}{\scriptstyle \ell}\to\alpha
}
}
\frac{
E\big(
\tau_K-\tau^*_K
\,|\,
O_0=\oe
\big)}
{
\displaystyle
E\big(\tau^*_K\,|\,
O_0=\oe
\big)
}\,=
\begin{cases}
\quad 0 &\text{if }\alpha\,\phi(a)<\ln\kappa\\
\,\,+\infty &\text{if }\alpha\,\phi(a)>\ln\kappa\\
\end{cases}
$$
In particular, if
$\alpha\,\phi(a)<\ln\kappa$,
we have
$$\lim_{
\genfrac{}{}{0pt}{1}{\ell,m\to\infty,\,
q\to 0
}
{{\ell q} \to a,\,
\frac{\scriptstyle m}{\scriptstyle \ell}\to\alpha
}
}\,
\int_{[0,1]}f\,d\nu_K\,=\,
0\,.
$$

\section{Synthesis}
Let us look at the formula given 
at the end of section~\ref{Bounds}:
\begin{multline*}
\int_{\textstyle\pml}
f\bigg( \frac{ \ozk }{m} \bigg)
\,d\mu_O^\theta(o)
\,=\,
\cr
\frac{
\displaystyle
E\bigg(\int_0^{\tau^*}
f\bigg( \frac{ 
O^\theta_s(0)+\cdots+O^\theta_s(K)
 }{m} \bigg)
\,ds\,\Big|\,
O^\theta_0=\ote
\bigg)
}{
\displaystyle
E\big(\tau^*\,|\,
O^\theta_0=\ote
\big)+
E\big(\tau_0\,|\,
Z^\t_0=z^\t\big)
}\\
+\frac{
\displaystyle
1
+
E\big(\tau_0\,|\,
Z^\t_0=z^\t\big)
}{
\displaystyle
E\big(\tau^*\,|\,
O^\theta_0=\ote
\big)+
E\big(\tau_0\,|\,
Z^\t_0=z^\t\big)
}
\int_{\textstyle\cE_K}
f\Big( \frac{ z_0+\cdots+z_K }{m} \Big)
d\nu^\theta(z)
\,.
\end{multline*}
The integral appearing in the first term of the right--hand side 
is bounded as follows:
$$
0\,\leq\,
E\bigg(\int_0^{\tau^*}
f\bigg( \frac{ 
O^\theta_s(0)+\cdots+O^\theta_s(K)
 }{m} \bigg)
\,ds\bigg)
\,\leq\,
f(1)E\big(\tau^*\,|\,
O^\theta_0=\ote
\big)\,.
$$
The stopping times $\tau^*$ and $\tau_0$
are the same as the discovery time of the master sequence
and the persistence time of the master sequence from \cite{Cerf}.
As shown in \cite{Cerf}, the following estimates hold:
\begin{align*}
\lim_{
\genfrac{}{}{0pt}{1}{\ell,m\to\infty}
{q\to 0,\,
{\ell q} \to a}
}
\frac{1}{\ell}\ln
E\big(\tau^*\,|\,
O^\theta_0=\ote
\big)\,&=\,
\ln\k\,,\\
\lim_{
\genfrac{}{}{0pt}{1}{\ell,m\to\infty}
{q\to 0,\,
{\ell q} \to a}
}
\frac{1}{m}\ln
E\big(\tau_0\,|\,
Z^\t_0=z^\t\big)\,&=\,
\phi(a)\,,
\end{align*}
where the function 
$\phi:{\mathbb R}^+\to
{\mathbb R}^+\cup\{\,+\infty\,\}$ 
is given by
$\phi(a)=0$
if $a\geq\ln\sigma$,
and
$$
\forall a<\ln\sigma\qquad
\phi(a)\,=\,
\frac
{ \displaystyle \sigma(1-e^{-a})
\ln\frac{\displaystyle\sigma(1-e^{-a})}{\displaystyle\sigma-1}
+\ln(\sigma e^{-a})}
{
\displaystyle (1-\sigma(1-e^{-a}))
}\,.
\index{$\phi(a)$}
$$
Therefore,
$$\lim_{
\genfrac{}{}{0pt}{1}{\ell,m\to\infty,\,
q\to 0
}
{{\ell q} \to a,\,
\frac{\scriptstyle m}{\scriptstyle \ell}\to\alpha
}
}
\frac{
E\big(\tau_0\,|\,
Z^\t_0=z^\t\big)}
{
\displaystyle
E\big(\tau^*\,|\,
O^\theta_0=\ote
\big)
}\,=
\begin{cases}
\quad 0 &\text{si }\alpha\,\phi(a)<\ln\kappa\,,\\
\,\,+\infty &\text{si }\alpha\,\phi(a)>\ln\kappa\,.\\
\end{cases}
$$
This, together with the results form the previous section
and section~\ref{Convergence} imply that

$\bullet$ 
If $\a\phi(a)<\ln\kappa$, then
$$\lim_{
\genfrac{}{}{0pt}{1}{\ell,m\to\infty,\,
q\to 0
}
{{\ell q} \to a,\,
\frac{\scriptstyle m}{\scriptstyle \ell}\to\alpha
}
}
\int_{\textstyle\pml}
f\bigg( \frac{ \ozk }{m} \bigg)
\,d\mu_O^\theta(o)
\,=\,0\,.
$$

$\bullet$ 
If $\a\phi(a)>\ln\kappa$, then
$$\lim_{
\genfrac{}{}{0pt}{1}{\ell,m\to\infty,\,
q\to 0
}
{{\ell q} \to a,\,
\frac{\scriptstyle m}{\scriptstyle \ell}\to\alpha
}
}
\int_{\textstyle\pml}
f\bigg( \frac{ \ozk }{m} \bigg)
\,d\mu_O^\theta(o)
\,=\,
f(\rho^*_0+\cdots+\rho^*_K)\,.
$$
\nocite{DBMJ}
\nocite{SAA1}
\nocite{PEM}

\bibliographystyle{plain}
\bibliography{qart}

\begin{thebibliography}{10}

\bibitem{AF}
Domingos Alves and Jose~Fernando Fontanari.
\newblock Error threshold in finite populations.
\newblock {\em Phys. Rev. E}, 57:7008--7013, 1998.

\bibitem{ADL}
Jon~P. Anderson, Richard Daifuku, and Lawrence~A. Loeb.
\newblock Viral error catastrophe by mutagenic nucleosides.
\newblock {\em Annual Review of Microbiology}, 58(1):183–205, 2004.

\bibitem{Cerf}
Rapha\"el Cerf.
\newblock Critical population and error threshold on the sharp peak landscape
  for a {M}oran model.
\newblock {\em preprint}, 2012.

\bibitem{CCA}
Shane Crotty, Craig~E. Cameron, and Raul Andino.
\newblock {RNA} virus error catastrophe: Direct molecular test by using
  ribavirin.
\newblock {\em Proceedings of the National Academy of Sciences},
  98(12):6895–6900, 2001.

\bibitem{DSS}
Lloyd Demetrius, Peter Schuster, and Karl Sigmund.
\newblock Polynucleotide evolution and branching processes.
\newblock {\em Bulletin of Mathematical Biology}, 47(2):239 -- 262, 1985.

\bibitem{DSV}
Narendra~M. Dixit, Piyush Srivastava, and Nisheeth~K. Vishnoi.
\newblock A finite population model of molecular evolution: theory and
  computation.
\newblock {\em J. Comput. Biol.}, 19(10):1176--1202, 2012.

\bibitem{Domingo}
Esteban Domingo.
\newblock Quasispecies theory in virology.
\newblock {\em Journal of Virology}, 76(1):463--465, 2002.

\bibitem{DBMJ}
Esteban Domingo, Christof Biebricher, Manfred Eigen, and John~J. Holland.
\newblock {\em Quasispecies and RNA virus evolution: principles and
  consequences}.
\newblock Landes Bioscience, Austin, Tex., 2001.

\bibitem{Eigen1}
Manfred Eigen.
\newblock Self-organization of matter and the evolution of biological
  macromolecules.
\newblock {\em Naturwissenschaften}, 58(10):465--523, 1971.

\bibitem{EMS}
Manfred Eigen, John McCaskill, and Peter Schuster.
\newblock The molecular quasi-species.
\newblock {\em Advances in Chemical Physics}, 75:149--263, 1989.

\bibitem{ES1}
Manfred Eigen and Schuster Peter.
\newblock The hypercycle. {A} principle of natural self--organization. {P}art
  {A}: {E}mergence of the hypercycle.
\newblock {\em Naturwissenschaften}, 64(11):541 -- 565, 1977.

\bibitem{ES2}
Manfred Eigen and Schuster Peter.
\newblock The hypercycle. {A} principle of natural self--organization. {P}art
  {B}: {T}he abstract hypercycle.
\newblock {\em Naturwissenschaften}, 65(1):7 -- 41, 1978.

\bibitem{ES3}
Manfred Eigen and Schuster Peter.
\newblock The hypercycle. {A} principle of natural self--organization. {P}art
  {C}: {T}he realistic hypercycle.
\newblock {\em Naturwissenschaften}, 65(7):341 -- 369, 1978.

\bibitem{FW}
Mark~I. Freidlin and Alexander~D. Wentzell.
\newblock {\em Random perturbations of dynamical systems}, volume 260 of {\em
  Grundlehren der Mathematischen Wissenschaften [Fundamental Principles of
  Mathematical Sciences]}.
\newblock Springer, Heidelberg, third edition, 2012.
\newblock Translated from the 1979 Russian original by Joseph Sz{\"u}cs.

\bibitem{McCaskill}
John McCaskill.
\newblock A stochastic theory of macromolecular evolution.
\newblock {\em Biological Cybernetics}, 50:63--73, 1984.

\bibitem{Moran}
Pat Moran.
\newblock Random processes in genetics.
\newblock {\em Proc. Cambridge Philos. Soc.}, 54:60--71, 1958.

\bibitem{Musso}
Fabio Musso.
\newblock {A stochastic version of the Eigen model.}
\newblock {\em Bull. Math. Biol.}, 73(1):151--180, 2011.

\bibitem{NS}
Martin~A. Nowak and Peter Schuster.
\newblock Error thresholds of replication in finite populations. {M}utation
  frequencies and the onset of {M}uller's ratchet.
\newblock {\em Journal of theoretical Biology}, 137 (4):375--395, 1989.

\bibitem{PEM}
Jeong-Man Park, Enrique Mu\~noz, and Michael~W. Deem.
\newblock Quasispecies theory for finite populations.
\newblock {\em Phys. Rev. E}, 81:011902, 2010.

\bibitem{SAA1}
David~B. Saakian, Michael~W. Deem, and Chin-Kun Hu.
\newblock Finite population size effects in quasispecies models with
  single-peak fitness landscape.
\newblock {\em Europhysics Letters}, 98(1):18001, 2012.

\bibitem{Schuster}
Peter Schuster.
\newblock Mathematical modeling of evolution. {S}olved and open problems.
\newblock {\em Theory Biosci.}, 130:71--89, 2011.

\bibitem{SD}
Ricard~V. Sol\'e and Thomas~S. Deisboeck.
\newblock An error catastrophe in cancer?
\newblock {\em Journal of Theoretical Biology}, 228:47--54, 2004.

\bibitem{TBVD}
Kushal Tripathi, Rajesh Balagam, Nisheeth~K. Vishnoi, and Narendra~M. Dixit.
\newblock Stochastic simulations suggest that {HIV}-1 survives close to its
  error threshold.
\newblock {\em PLoS Comput. Biol.}, 8(9):e1002684, 14, 2012.

\bibitem{Wilke}
Claus Wilke.
\newblock Quasispecies theory in the context of population genetics.
\newblock {\em BMC Evolutionary Biology}, 5:1--8, 2005.

\end{thebibliography}
\end{document}